\documentclass{article}
\usepackage[utf8]{inputenc}
\usepackage[margin=3cm]{geometry}

\usepackage{amsmath,amsthm,amssymb,amsfonts}
\usepackage{mathtools}
\usepackage{tikz}
\usepackage{tikz-cd}
\usepackage{ifthen}
\usepackage[normalem]{ulem}
\usepackage{lipsum}
\usepackage{authblk}
\usepackage{physics}
\usepackage{accents }
\usepackage{extpfeil}
\usepackage{ wasysym }

\usetikzlibrary{arrows}

\usepackage{hyperref}

\usepackage{xcolor}

\usepackage{algorithm}
\usepackage{algorithmic}

\newtheorem{theorem}{Theorem}
\newtheorem{definition}{Definition}
\newtheorem{proposition}{Proposition}
\newtheorem{remark}{Remark}
\newtheorem{lemma}{Lemma}
\newtheorem{example}{Example}
\newtheorem{conjecture}{Conjecture}
\newtheorem{corollary}{Corollary}
\theoremstyle{definition}

\newcommand{\D}{\partial}
\newcommand{\Z}{\mathbb{Z}}
\newcommand{\F}{\mathbb{F}}
\newcommand{\local}{\underline{\mathcal{L}}}
\newcommand{\R}{\mathbb{R}}
\newcommand{\Sp}{\mathbb{S}}

\newcommand{\N}{\mathbb{N}}
\newcommand{\tran}{\mathbf{t}}
\newcommand{\vect}{\mathbf{vect}}

\newcommand{\im}{\mathrm{im}}
\newcommand{\coim}{\mathrm{coim}}
\newcommand{\coker}{\mathrm{coker}}
\newcommand{\surj}{\twoheadrightarrow}
\newcommand{\inj}{\hookrightarrow}

\newcommand{\Ab}{\mathcal{A}}
\newcommand{\torus}{\mathbb{T}}
\newcommand{\Hg}{\mathrm{H}}
\newcommand{\Ig}{\mathrm{I}}

\newcommand{\bisheafcat}{\mathbf{BiSh}}
\newcommand{\sheafcat}{\mathbf{Sh}}
\newcommand{\cosheafcat}{\mathbf{coSh}}
\newcommand{\localsystemcat}{\mathbf{Loc}}
\newcommand{\colocalsystemcat}{\mathbf{coLoc}}

\newcommand{\Fsheaf}{\overline{F}}
\newcommand{\Fcosheaf}{\underline{F}}
\newcommand{\Fbisheaf}{\overline{\underline{F}}}
\newcommand{\Esheaf}{\overline{E}}
\newcommand{\Mcosheaf}{\underline{M}}
\newcommand{\Kcosheaf}{\underline{K}}
\newcommand{\Ibisheaf}{\overline{\underline{I}}}

\newcommand{\sheaf}[1]{\overline{#1}}
\newcommand{\cosheaf}[1]{\underline{#1}}
\newcommand{\bisheaf}[1]{\overline{\underline{#1}}}

\newcommand{\Rsheaf}[1]{\widetilde{#1}}
\newcommand{\Rcosheaf}[1]{\undertilde{#1}}
\newcommand{\Rbisheaf}[1]{\widetilde{\undertilde{#1}}}
\newcommand{\Sub}{\mathrm{Sub}}
\newcommand{\Cl}{\mathbf{cl}}
\newcommand{\St}{\mathbf{st}}

\newcommand{\ballcolour}{gray}
\newcommand{\circlecolour}{red}
\newcommand{\pathcolour}{yellow}
\newcommand{\bordercolour}{black}
\newcommand{\pointcolour}{blue}

\newcommand\scalemath[2]{\scalebox{#1}{\mbox{\ensuremath{\displaystyle #2}}}}
\newcommand{\scalesize}{0.8}

\title{Persistent Local Systems of Periodic Spaces}

\author[1]{Adam Onus\footnote{\href{a.onus@qmul.ac.uk}{a.onus@qmul.ac.uk}} }
\author[1]{Primoz Skraba\footnote{\href{p.skraba@qmul.ac.uk}{p.skraba@qmul.ac.uk}}}
\affil[1]{School of Mathematical Sciences, Queen Mary University of London, U.K.}

\date{\today}

\begin{document}

\maketitle

\begin{abstract}
The topology of periodic spaces has attracted a lot of interest in recent years in order to study and classify crystalline structures and other large homogeneous data sets, such as the distribution of galaxies in cosmology.
In practice, these objects are studied by taking a finite sample and introducing periodic boundary conditions, however this introduces and removes many subtle homological features.
Here, build on the work of Onus and Robins \cite{onus2022quantifying} and Onus and Skraba \cite{onus2023computing} to investigate whether one can recover the (persistent) homology of a periodic cell complex $K$ from a finite quotient space $G$ of equivalence classes under translations.
In particular, we search for a computationally friendly method to identify all ``toroidal cycles'' of $G$ which do not lift to cycles in $K$. 
We show that all toroidal and non-toroidal cycles of $G$ of arbitrary homology degree can be completely classified for $K$ of arbitrary periodicity using the recently developed machinery of bisheaves and persistent local systems.
In doing so, we also introduce a framework for a computationally viable persistence theory of periodic spaces.
Finally, we outline algorithms for how to apply our results to real data, including a polynomial time algorithm for calculating the canonical persistent local system attributed to a given bisheaf.
\end{abstract}

\section{Introduction}
Spatially periodic topological spaces arise naturally in many areas of mathematics, physics, and materials science.  
These spaces exhibit a repeating pattern or symmetry over regular intervals, making them important models for understanding phenomena such as the formation of (quasi)-crystals, repeating patterns in dynamical systems, and, more generally, systems with periodic boundary conditions~\cite{cramer2020evolution,webber2005recurrence,coley2023spatial}.  
Periodic spaces also appear in the study of tiling patterns, lattices, and periodic functions~\cite{kolbe2022tile,von1991nodal}.

In this paper we study the (persistent) homology of periodic spaces.  
As a motivating example, consider a periodic point set $\mathcal{P}$ that arises from the $d$-dimensional integer lattice $\mathbb{Z}^d$.  
Let $P\subset[0,1]^d$ and set $\mathcal{P}:=P+\mathbb{Z}^d$.  
We can then examine the $r$-offset of $\mathcal{P}$ -- either for a fixed $r$, as in Figure \ref{fig:motivation}, or for varying $r$ in persistent homology.  
Because $\mathcal{P}$ is infinite, we embed it into a flat $d$-dimensional torus, which represents a single unit cell in~$\mathbb{R}^d$ with periodic boundary conditions.  
This practice is common in simulations in materials science, cosmology, and other fields, as it captures the relevant physics while minimising memory requirements.  
For instance, complexes constructed in this way have recently improved machine-learning pipelines for perovskite design~\cite{hu2025quotient}.

\begin{figure}
    \centering
        \begin{tikzpicture}
    \path[use as bounding box] (-2,-2) rectangle (6,6);

    \foreach \i in {-1,...,1}
        {
        \foreach \j in {-1,...,1}
            {
            \fill[\ballcolour!50] (0.19+4*\i, 1.39+4*\j) circle (0.32);
            \fill[\ballcolour!50] (1.32+4*\i, 1.41+4*\j) circle (0.32);
            \fill[\ballcolour!50] (1.20+4*\i, 3.64+4*\j) circle (0.32);
            \fill[\ballcolour!50] (2.64+4*\i, 0.61+4*\j) circle (0.32);
            \fill[\ballcolour!50] (1.27+4*\i, 2.64+4*\j) circle (0.32);
            \fill[\ballcolour!50] (0.36+4*\i, 3.83+4*\j) circle (0.32);
            \fill[\ballcolour!50] (2.15+4*\i, 1.25+4*\j) circle (0.32);
            \fill[\ballcolour!50] (0.81+4*\i, 1.90+4*\j) circle (0.32);
            \fill[\ballcolour!50] (0.97+4*\i, 3.53+4*\j) circle (0.32);
            \fill[\ballcolour!50] (2.61+4*\i, 1.65+4*\j) circle (0.32);
            \fill[\ballcolour!50] (2.41+4*\i, 3.25+4*\j) circle (0.32);
            \fill[\ballcolour!50] (2.69+4*\i, 3.39+4*\j) circle (0.32);
            \fill[\ballcolour!50] (3.61+4*\i, 3.57+4*\j) circle (0.32);
            \fill[\ballcolour!50] (0.07+4*\i, 2.65+4*\j) circle (0.32);
            \fill[\ballcolour!50] (0.04+4*\i, 2.23+4*\j) circle (0.32);
            \fill[\ballcolour!50] (1.01+4*\i, 2.90+4*\j) circle (0.32);
            \fill[\ballcolour!50] (1.60+4*\i, 0.20+4*\j) circle (0.32);
            \fill[\ballcolour!50] (0.40+4*\i, 0.88+4*\j) circle (0.32);
            \fill[\ballcolour!50] (0.51+4*\i, 0.88+4*\j) circle (0.32);
            \fill[\ballcolour!50] (2.84+4*\i, 2.27+4*\j) circle (0.32);
            \fill[\ballcolour!50] (0.34+4*\i, 0.94+4*\j) circle (0.32);
            \fill[\ballcolour!50] (3.96+4*\i, 3.79+4*\j) circle (0.32);
            \fill[\ballcolour!50] (1.94+4*\i, 1.92+4*\j) circle (0.32);
            \fill[\ballcolour!50] (3.59+4*\i, 3.49+4*\j) circle (0.32);
            \fill[\ballcolour!50] (3.89+4*\i, 1.98+4*\j) circle (0.32);
            \fill[\ballcolour!50] (1.46+4*\i, 2.24+4*\j) circle (0.32);
            \fill[\ballcolour!50] (0.08+4*\i, 3.71+4*\j) circle (0.32);
            \fill[\ballcolour!50] (2.56+4*\i, 2.69+4*\j) circle (0.32);
            \fill[\ballcolour!50] (3.36+4*\i, 2.21+4*\j) circle (0.32);
            \fill[\ballcolour!50] (2.28+4*\i, 3.60+4*\j) circle (0.32);
            \fill[\ballcolour!50] (3.95+4*\i, 1.35+4*\j) circle (0.32);
            \fill[\ballcolour!50] (1.55+4*\i, 2.32+4*\j) circle (0.32);
            \fill[\ballcolour!50] (2.81+4*\i, 1.04+4*\j) circle (0.32);
            \fill[\ballcolour!50] (0.10+4*\i, 0.36+4*\j) circle (0.32);
            \fill[\ballcolour!50] (1.05+4*\i, 3.83+4*\j) circle (0.32);
            \fill[\ballcolour!50] (0.75+4*\i, 1.18+4*\j) circle (0.32);
            \fill[\ballcolour!50] (0.85+4*\i, 0.36+4*\j) circle (0.32);
            \fill[\ballcolour!50] (2.78+4*\i, 2.90+4*\j) circle (0.32);
            \fill[\ballcolour!50] (2.63+4*\i, 2.45+4*\j) circle (0.32);
            \fill[\ballcolour!50] (0.17+4*\i, 2.81+4*\j) circle (0.32);
            \fill[\ballcolour!50] (1.54+4*\i, 2.18+4*\j) circle (0.32);
            \fill[\ballcolour!50] (1.16+4*\i, 3.23+4*\j) circle (0.32);
            \fill[\ballcolour!50] (1.28+4*\i, 0.37+4*\j) circle (0.32);
            \fill[\ballcolour!50] (0.63+4*\i, 3.99+4*\j) circle (0.32);
            \fill[\ballcolour!50] (1.27+4*\i, 1.37+4*\j) circle (0.32);
            \fill[\ballcolour!50] (1.96+4*\i, 0.75+4*\j) circle (0.32);
            \fill[\ballcolour!50] (0.08+4*\i, 3.89+4*\j) circle (0.32);
            \fill[\ballcolour!50] (1.04+4*\i, 3.75+4*\j) circle (0.32);
            \fill[\ballcolour!50] (1.96+4*\i, 0.15+4*\j) circle (0.32);
            \fill[\ballcolour!50] (2.32+4*\i, 0.93+4*\j) circle (0.32);
            \fill[\ballcolour!50] (0.65+4*\i, 3.65+4*\j) circle (0.32);
            \fill[\ballcolour!50] (3.50+4*\i, 0.08+4*\j) circle (0.32);
            \fill[\ballcolour!50] (1.92+4*\i, 3.47+4*\j) circle (0.32);
            \fill[\ballcolour!50] (0.79+4*\i, 1.99+4*\j) circle (0.32);
            \fill[\ballcolour!50] (3.66+4*\i, 0.35+4*\j) circle (0.32);
            \fill[\ballcolour!50] (2.55+4*\i, 3.24+4*\j) circle (0.32);
            \fill[\ballcolour!50] (2.22+4*\i, 2.35+4*\j) circle (0.32);
            \fill[\ballcolour!50] (3.41+4*\i, 2.70+4*\j) circle (0.32);
            \fill[\ballcolour!50] (3.68+4*\i, 0.93+4*\j) circle (0.32);
            \fill[\ballcolour!50] (1.55+4*\i, 1.85+4*\j) circle (0.32);
            }
        }

    \foreach \i in {-1,...,1}
        {
        \foreach \j in {-1,...,1}
            {
            \node[shape=circle,fill=\pointcolour,inner sep=0pt,minimum size=0.12cm] () at (0.19+4*\i, 1.39+4*\j) {};
            \node[shape=circle,fill=\pointcolour,inner sep=0pt,minimum size=0.12cm] () at (1.32+4*\i, 1.41+4*\j) {};
            \node[shape=circle,fill=\pointcolour,inner sep=0pt,minimum size=0.12cm] () at (1.20+4*\i, 3.64+4*\j) {};
            \node[shape=circle,fill=\pointcolour,inner sep=0pt,minimum size=0.12cm] () at (2.64+4*\i, 0.61+4*\j) {};
            \node[shape=circle,fill=\pointcolour,inner sep=0pt,minimum size=0.12cm] () at (1.27+4*\i, 2.64+4*\j) {};
            \node[shape=circle,fill=\pointcolour,inner sep=0pt,minimum size=0.12cm] () at (0.36+4*\i, 3.83+4*\j) {};
            \node[shape=circle,fill=\pointcolour,inner sep=0pt,minimum size=0.12cm] () at (2.15+4*\i, 1.25+4*\j) {};
            \node[shape=circle,fill=\pointcolour,inner sep=0pt,minimum size=0.12cm] () at (0.81+4*\i, 1.90+4*\j) {};
            \node[shape=circle,fill=\pointcolour,inner sep=0pt,minimum size=0.12cm] () at (0.97+4*\i, 3.53+4*\j) {};
            \node[shape=circle,fill=\pointcolour,inner sep=0pt,minimum size=0.12cm] () at (2.61+4*\i, 1.65+4*\j) {};
            \node[shape=circle,fill=\pointcolour,inner sep=0pt,minimum size=0.12cm] () at (2.41+4*\i, 3.25+4*\j) {};
            \node[shape=circle,fill=\pointcolour,inner sep=0pt,minimum size=0.12cm] () at (2.69+4*\i, 3.39+4*\j) {};
            \node[shape=circle,fill=\pointcolour,inner sep=0pt,minimum size=0.12cm] () at (3.61+4*\i, 3.57+4*\j) {};
            \node[shape=circle,fill=\pointcolour,inner sep=0pt,minimum size=0.12cm] () at (0.07+4*\i, 2.65+4*\j) {};
            \node[shape=circle,fill=\pointcolour,inner sep=0pt,minimum size=0.12cm] () at (0.04+4*\i, 2.23+4*\j) {};
            \node[shape=circle,fill=\pointcolour,inner sep=0pt,minimum size=0.12cm] () at (1.01+4*\i, 2.90+4*\j) {};
            \node[shape=circle,fill=\pointcolour,inner sep=0pt,minimum size=0.12cm] () at (1.60+4*\i, 0.20+4*\j) {};
            \node[shape=circle,fill=\pointcolour,inner sep=0pt,minimum size=0.12cm] () at (0.40+4*\i, 0.88+4*\j) {};
            \node[shape=circle,fill=\pointcolour,inner sep=0pt,minimum size=0.12cm] () at (0.51+4*\i, 0.88+4*\j) {};
            \node[shape=circle,fill=\pointcolour,inner sep=0pt,minimum size=0.12cm] () at (2.84+4*\i, 2.27+4*\j) {};
            \node[shape=circle,fill=\pointcolour,inner sep=0pt,minimum size=0.12cm] () at (0.34+4*\i, 0.94+4*\j) {};
            \node[shape=circle,fill=\pointcolour,inner sep=0pt,minimum size=0.12cm] () at (3.96+4*\i, 3.79+4*\j) {};
            \node[shape=circle,fill=\pointcolour,inner sep=0pt,minimum size=0.12cm] () at (1.94+4*\i, 1.92+4*\j) {};
            \node[shape=circle,fill=\pointcolour,inner sep=0pt,minimum size=0.12cm] () at (3.59+4*\i, 3.49+4*\j) {};
            \node[shape=circle,fill=\pointcolour,inner sep=0pt,minimum size=0.12cm] () at (3.89+4*\i, 1.98+4*\j) {};
            \node[shape=circle,fill=\pointcolour,inner sep=0pt,minimum size=0.12cm] () at (1.46+4*\i, 2.24+4*\j) {};
            \node[shape=circle,fill=\pointcolour,inner sep=0pt,minimum size=0.12cm] () at (0.08+4*\i, 3.71+4*\j) {};
            \node[shape=circle,fill=\pointcolour,inner sep=0pt,minimum size=0.12cm] () at (2.56+4*\i, 2.69+4*\j) {};
            \node[shape=circle,fill=\pointcolour,inner sep=0pt,minimum size=0.12cm] () at (3.36+4*\i, 2.21+4*\j) {};
            \node[shape=circle,fill=\pointcolour,inner sep=0pt,minimum size=0.12cm] () at (2.28+4*\i, 3.60+4*\j) {};
            \node[shape=circle,fill=\pointcolour,inner sep=0pt,minimum size=0.12cm] () at (3.95+4*\i, 1.35+4*\j) {};
            \node[shape=circle,fill=\pointcolour,inner sep=0pt,minimum size=0.12cm] () at (1.55+4*\i, 2.32+4*\j) {};
            \node[shape=circle,fill=\pointcolour,inner sep=0pt,minimum size=0.12cm] () at (2.81+4*\i, 1.04+4*\j) {};
            \node[shape=circle,fill=\pointcolour,inner sep=0pt,minimum size=0.12cm] () at (0.10+4*\i, 0.36+4*\j) {};
            \node[shape=circle,fill=\pointcolour,inner sep=0pt,minimum size=0.12cm] () at (1.05+4*\i, 3.83+4*\j) {};
            \node[shape=circle,fill=\pointcolour,inner sep=0pt,minimum size=0.12cm] () at (0.75+4*\i, 1.18+4*\j) {};
            \node[shape=circle,fill=\pointcolour,inner sep=0pt,minimum size=0.12cm] () at (0.85+4*\i, 0.36+4*\j) {};
            \node[shape=circle,fill=\pointcolour,inner sep=0pt,minimum size=0.12cm] () at (2.78+4*\i, 2.90+4*\j) {};
            \node[shape=circle,fill=\pointcolour,inner sep=0pt,minimum size=0.12cm] () at (2.63+4*\i, 2.45+4*\j) {};
            \node[shape=circle,fill=\pointcolour,inner sep=0pt,minimum size=0.12cm] () at (0.17+4*\i, 2.81+4*\j) {};
            \node[shape=circle,fill=\pointcolour,inner sep=0pt,minimum size=0.12cm] () at (1.54+4*\i, 2.18+4*\j) {};
            \node[shape=circle,fill=\pointcolour,inner sep=0pt,minimum size=0.12cm] () at (1.16+4*\i, 3.23+4*\j) {};
            \node[shape=circle,fill=\pointcolour,inner sep=0pt,minimum size=0.12cm] () at (1.28+4*\i, 0.37+4*\j) {};
            \node[shape=circle,fill=\pointcolour,inner sep=0pt,minimum size=0.12cm] () at (0.63+4*\i, 3.99+4*\j) {};
            \node[shape=circle,fill=\pointcolour,inner sep=0pt,minimum size=0.12cm] () at (1.27+4*\i, 1.37+4*\j) {};
            \node[shape=circle,fill=\pointcolour,inner sep=0pt,minimum size=0.12cm] () at (1.96+4*\i, 0.75+4*\j) {};
            \node[shape=circle,fill=\pointcolour,inner sep=0pt,minimum size=0.12cm] () at (0.08+4*\i, 3.89+4*\j) {};
            \node[shape=circle,fill=\pointcolour,inner sep=0pt,minimum size=0.12cm] () at (1.04+4*\i, 3.75+4*\j) {};
            \node[shape=circle,fill=\pointcolour,inner sep=0pt,minimum size=0.12cm] () at (1.96+4*\i, 0.15+4*\j) {};
            \node[shape=circle,fill=\pointcolour,inner sep=0pt,minimum size=0.12cm] () at (2.32+4*\i, 0.93+4*\j) {};
            \node[shape=circle,fill=\pointcolour,inner sep=0pt,minimum size=0.12cm] () at (0.65+4*\i, 3.65+4*\j) {};
            \node[shape=circle,fill=\pointcolour,inner sep=0pt,minimum size=0.12cm] () at (3.50+4*\i, 0.08+4*\j) {};
            \node[shape=circle,fill=\pointcolour,inner sep=0pt,minimum size=0.12cm] () at (1.92+4*\i, 3.47+4*\j) {};
            \node[shape=circle,fill=\pointcolour,inner sep=0pt,minimum size=0.12cm] () at (0.79+4*\i, 1.99+4*\j) {};
            \node[shape=circle,fill=\pointcolour,inner sep=0pt,minimum size=0.12cm] () at (3.66+4*\i, 0.35+4*\j) {};
            \node[shape=circle,fill=\pointcolour,inner sep=0pt,minimum size=0.12cm] () at (2.55+4*\i, 3.24+4*\j) {};
            \node[shape=circle,fill=\pointcolour,inner sep=0pt,minimum size=0.12cm] () at (2.22+4*\i, 2.35+4*\j) {};
            \node[shape=circle,fill=\pointcolour,inner sep=0pt,minimum size=0.12cm] () at (3.41+4*\i, 2.70+4*\j) {};
            \node[shape=circle,fill=\pointcolour,inner sep=0pt,minimum size=0.12cm] () at (3.68+4*\i, 0.93+4*\j) {};
            \node[shape=circle,fill=\pointcolour,inner sep=0pt,minimum size=0.12cm] () at (1.55+4*\i, 1.85+4*\j) {};
            }
        }

    \draw[thick,\bordercolour!40] (-4,0) -- (0,0) -- (0,-4);
    \draw[thick,\bordercolour!40] (-4,4) -- (0,4) -- (0,8);
    \draw[thick,\bordercolour!40] (4,-4) -- (4,0) -- (8,0);
    \draw[thick,\bordercolour!40] (4,8) -- (4,4) -- (8,4);
    \draw[thick,\bordercolour!80] (0,0) -- (0,4) -- (4,4) -- (4,0) -- cycle;

    \draw[densely dotted,\circlecolour,line width=0.7mm] (4,0.8) circle (0.4);
    \draw[densely dotted,\pathcolour,line width=0.7mm] (0,1.3) -- (1.2,1.3) -- (2.5,2.5) -- (3.9,1.9) -- (4,1.3);

    \fill[white] (-4.5,-4.5) -- (-4.5,-2) -- (8.5,-2) -- (8.5,-4.5) -- cycle;
    \fill[white] (-4.5,8.5) -- (-4.5,6) -- (8.5,6) -- (8.5,8.5) -- cycle;
    \fill[white] (8.5,-2) -- (6,-2) -- (6,6) -- (8.5,6) -- cycle;
    \fill[white] (-4.5,-2) -- (-2,-2) -- (-2,6) -- (-4.5,6) -- cycle;
    \end{tikzpicture}
\caption{A 2-periodic point cloud, with the $r$-offset (for $r=0.08$) shaded in \ballcolour.  
    Horizontal and vertical lines indicate the lattice structure.  
    The dashed \circlecolour\, circle and dashed \pathcolour\, path indicate 1-cycles, the latter appearing only after imposing periodic boundary conditions on $[0,1]^2$.}
    \label{fig:motivation}
\end{figure}
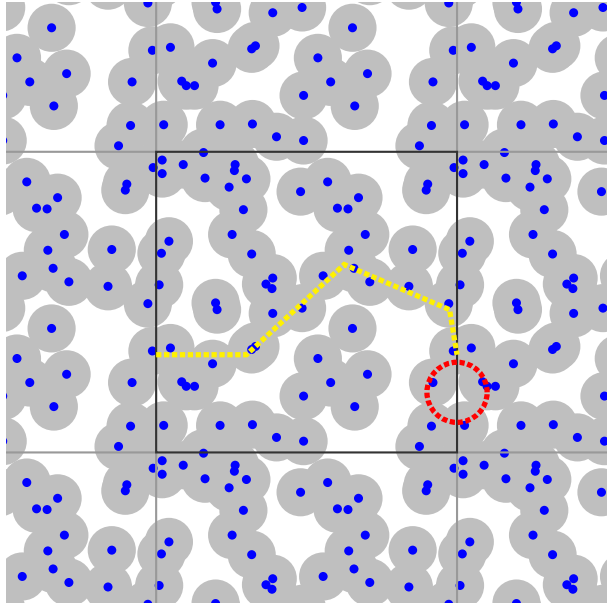

If we wish to compute the homology of the $r$-offset of $\mathcal{P}\cap W$ for a bounded region~$W$, the behaviour can be surprisingly subtle even when $W$ itself is simple.  
On the one hand, periodicity allows a decomposition into repeating units; on the other, boundary effects introduce intermittent homological features that complicate computation.  
Take Figure \ref{fig:motivation} as an example.  
When $W=[0,1]^2$ (the highlighted box) is viewed \emph{without} boundary identifications, the dashed \circlecolour\, circle is invisible to homology.  
Imposing periodic boundary conditions turns the dashed \pathcolour\, path into a cycle that does \emph{not} lift to a homology class of the infinite periodic space—we call such cycles \emph{toroidal}.  
Classifying toroidal cycles is essential: it reveals the true topology of the periodic point cloud and captures physically relevant properties.  
For instance, the appearance of a toroidal 1-cycle, as opposed to a ``true'' 1-cycle, signals the merger of large connected components—features that are useful when comparing three-dimensional layered materials~\cite{edelsbrunner2024merge,thygesen2023level}.

This paper builds on earlier work on the homology of periodic cellular complexes, which treated only special cases—degree-0 and degree-1 homology~\cite{onus2022quantifying} or one-dimensional periodicity~\cite{onus2023computing}—and it contributes to the growing literature on topological and geometric properties of highly symmetric data.  
Related work on periodic point clouds~\cite{edelsbrunner2024merge} introduced a complete and computationally efficient degree-0 \emph{persistent} homology theory for periodic spaces, represented by a variant of merge trees, but the underlying techniques do not extend straightforwardly to higher degrees.  
Independently, \cite{mcmanus2024computing} used similar methods to classify real crystals, and \cite{yim2025inferring} employed covering-space theory and monodromy to infer the ambient topology of non-Euclidean data.  
Applying periodic boundary conditions in Figure \ref{fig:motivation}, for example, is akin to using $\mathcal{P}$ to learn the topology of the (a~priori unknown) space $\torus^2$ in which it is embedded.  
Finally, \cite{adams2024persistent} introduced persistent \emph{equivariant} cohomology, hinting at extensions of our approach from translational to more general symmetries.

Here, we extend the results of~\cite{onus2022quantifying,onus2023computing} by characterizing the (persistent) homology in all degrees of finite quotient spaces of periodic cellular complexes with arbitrary-dimensional periodicity.  
This work is the first to apply the theory of \emph{bisheaves} and their canonical invariants—\emph{persistent local systems} (PLSs) from~\cite{macpherson2021persistent}—to periodic data.  
Bisheaves furnish a homology theory for maps to oriented manifolds and support richer structure than ordinary (co)sheaves; their associated PLSs can be computed efficiently using parallel algorithms.  
By leveraging bisheaves and PLSs, we unite ideas from covering-space theory, monodromy, and scalable topological data analysis, establishing a comprehensive bridge between the (persistent) homology of a periodic space and that of its finite quotient.

Our main contributions are:
\begin{enumerate}
    \item a complete classification of toroidal cycles (e.g.\ the yellow ``cycle'' in Figure \ref{fig:motivation}) that appear when periodic boundary conditions are imposed (\emph{Theorems \ref{thm:1d-toroidal} and \ref{thm:toroidal}});
    \item a proof that for $\mathbb{Z}^d$-actions it suffices to study one periodic direction at a time (\emph{Proposition \ref{prop:toroidal-compression}});
    \item polynomial-time algorithms for computing the canonical PLS of a bisheaf (\emph{Algorithms \ref{alg:epi}–\ref{alg:mono}} and \emph{Theorem \ref{thm:algorithms}}).
\end{enumerate}
These results show that bisheaves and PLSs constitute a practical computational toolkit for analysing periodic spaces.


\section{Preliminaries} \label{sec:prelim}
We divide the preliminaries into two sections. First, we introduce the required background on cellular bisheaves and local systems, largely following \cite{nanda2020canonical}. Then we describe the background on periodic spaces, largely following \cite{onus2022quantifying} and \cite{onus2023computing}.

\subsection{Bisheaves and Local Systems}
Throughout this paper, we fix a cell complex $X$ and take (co)homology with coefficients in some fixed field $\F$.
We recall that $X$ can be considered a poset whose objects are the cells of the $X$ and whose partial ordering is $\sigma<\tau$ when $\tau$ is a face of $\sigma$.
Equivalently, $\sigma\leq\tau$ if and only if $\St(\sigma)\subseteq\St(\tau)$, where $\St(\sigma)$ denotes the open star of the cell $\sigma$ containing all cofaces of $\sigma$.
While the standard convention is to have the opposite inequality (i.e. $\sigma<\tau$ implies $\sigma$ is a face of $\tau$), this convention is better aligned to the sheaf-theoretic techniques we work with.

Let $\Ab$ be an (arbitrary) abelian category and $\vect$ the (abelian) category of finite vector spaces over some field $\F$.
All of our results apply to any field $\F$, although the examples we give in the paper are in practice usually set for $\F=\Z/2\Z$ for simplicity.
Where possible, we also state all definitions and results in full generality for $\Ab$.
The following definitions come from \cite{macpherson2021persistent} and \cite{nanda2020canonical}, to which we direct the reader for more information.

\begin{definition}
A (cellular) sheaf $\Fsheaf$ over $X$ valued in $\Ab$ is a contravariant functor from the poset category of $X$ to $\Ab$.
If, in addition, $\Fsheaf(\sigma\leq\tau):\Fsheaf(\tau)\to\Fsheaf(\sigma)$ is an epimorphism for every face relation $\sigma\leq\tau$ then we say $\Fsheaf$ is an episheaf.
\end{definition}

We denote by $\sheafcat(X)$ the (abelian) category of cellular sheaves over $X$, where a morphism $\Fsheaf\to\sheaf{G}$ is a collection of cell-wise maps $\Fsheaf(\sigma)\to\sheaf{G}(\sigma)$ which commute with the face relation maps of $\Fsheaf$ and $\sheaf{G}$.
For the family of sheaves with which we will work, consider a cell map $f:Y\to X$, by which we mean $f$ is a continuous map which is constructible in the sense of Definition~2.2 of \cite{macpherson2021persistent}.
Note that this includes most standard definitions of simplicial maps for simplicial complexes.

We can study $Y$ by pulling back along the fibers of $f$, defining a sheaf over $X$ valued in $\vect$ as follows.
First, to every cell $\sigma$ of $X$, we attribute the following relative homology
\[
\Fsheaf_\bullet(\sigma) = \Hg_\bullet\left(Y,Y\setminus f^{-1}(\St(\sigma))\right)
\]
For $Y$ locally compact and paracompact it follows that
\[
\Hg_\bullet\left(\Cl\left(f^{-1}(\St(\sigma))\right),\D f^{-1}(\St(\sigma))\right) \cong \Tilde{\Hg}_\bullet\left(\Cl\left(f^{-1}(\St(\sigma))/\D f^{-1}(\St(\sigma))\right)\right)
\]
where $\Cl$ denotes the set closure, $\D$ denotes the boundary and $\Tilde{\Hg}$ denotes the reduced homology.
That is, we can associate the relative homology of $(Y,Y\setminus f^{-1}(\St(\sigma)))$ with the reduced homology of its one-point compactification.
Moreover, for each $\sigma\leq \tau$, the inclusion $\St(\sigma)\inj\St(\tau)$ means $Y\setminus f^{-1}(\St(\tau))\inj Y\setminus f^{-1}(\St(\sigma))$ which in turn induces a canonical map 
\[
\Fsheaf_\bullet(\sigma\leq\tau):\Hg_\bullet\left(Y,Y\setminus f^{-1}(\St(\tau))\right)\to\Hg_\bullet\left(Y,Y\setminus f^{-1}(\St(\sigma))\right),
\]
completing the sheaf construction.

\begin{figure}
    \centering
    \begin{tikzpicture}[scale=0.9]
        \node (ldots) at (-8,2) {$\cdots$};
        \node (rdots) at (6,2) {$\cdots$};
        \foreach \i in {-7,...,5}
            {
                \foreach \j in {1,...,3}
    	    {
    	           \node[shape=circle,fill=black,inner sep=0pt,minimum size=0.15cm] (\i,\j) at (\i,\j) { };
    	    }
            }
        \foreach \i in {-7,...,5}
            {
                \draw (\i-0.5,3) to (\i+0.5,3);
                \draw (\i,3) to (\i+0.5,2.5);
                \draw (\i-0.5,2.5) to (\i,2);
                \draw (\i,2) to (\i+0.5,1.5);
                \draw (\i-0.5,1.5) to (\i,1);
                \draw[gray] (\i,1) to (\i+0.5,2);
                \draw[gray] (\i-0.5,2) to (\i,3);
            }
    \end{tikzpicture}
    \caption{A graph $K$ in $\R^2$ with vertex set $\Z\times\{0,\pm1\}$. $K$ is 1-periodic (see Definition~\ref{def:periodic}) and for each $n\in \mathbb{N}$, the quotient space $K/n\Z$ has a natural projection to the 1-sphere $\Sp^1$ defined by $(x,y)\mapsto e^{2\pi i x/n}$.}
    \label{fig:running-example}
\end{figure}
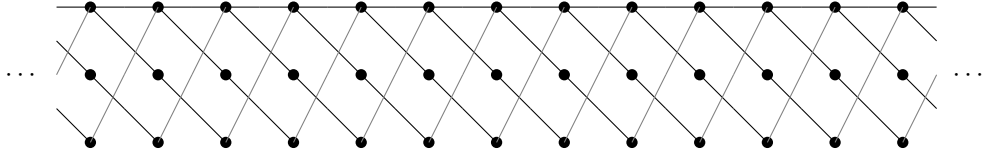

\begin{example}
Consider the map $K/n\Z\to \Sp^1$ described in Figure~\ref{fig:running-example} for $n\geq 3$.
We take a cellular decomposition of $\Sp^1$ whose vertex set is $\{e^{2\pi i m/n\,:\,m=0,\dots,n-1}\}$ and edge set $\{(e^{2\pi i m/n,2\pi i (m+1)/n)\,:\,m=0,\dots,n-1}\}$.
Observe that $K$ is compact and paracompact, so locally, the sheaf over $\Sp^1$ is determined by taking the reduced homology of the following spaces and maps.
\[
\begin{tikzpicture}
\node[shape=circle,fill=gray,inner sep=0pt,minimum size=0.15cm] (linf) at (-3,3) { };
\node[shape=circle,fill=black,inner sep=0pt,minimum size=0.15cm] (lv1) at (-3,2.5) { };
\node[shape=circle,fill=black,inner sep=0pt,minimum size=0.15cm] (lv2) at (-3,2) { };
\node[shape=circle,fill=black,inner sep=0pt,minimum size=0.15cm] (lv3) at (-3,1.5) { };

\node[shape=circle,fill=gray,inner sep=0pt,minimum size=0.15cm] (minf) at (0,2.5) { };

\node[shape=circle,fill=gray,inner sep=0pt,minimum size=0.15cm] (rinf) at (3,3) { };
\node[shape=circle,fill=black,inner sep=0pt,minimum size=0.15cm] (rv1) at (3,2.5) { };
\node[shape=circle,fill=black,inner sep=0pt,minimum size=0.15cm] (rv2) at (3,2) { };
\node[shape=circle,fill=black,inner sep=0pt,minimum size=0.15cm] (rv3) at (3,1.5) { };

\node (tlmap) at (-1.4,2.25) {\Huge $\twoheadrightarrow$};
\node (trmap) at (1.4,2.25) {\Huge $\twoheadleftarrow$};

\draw[red,line width=0.3mm] (linf) to[bend right = 30] (lv1);
\draw[red,line width=0.3mm] (linf) to[bend right = 70] (lv1);
\draw[red,line width=0.3mm] (linf) to[bend right = 80] (lv2);
\draw[red,line width=0.3mm] (linf) to[bend right = 90] (lv3);
\draw[blue,line width=0.3mm] (linf) to[bend left = 30] (lv1);
\draw[blue,line width=0.3mm] (linf) to[bend left = 70] (lv1);
\draw[blue,line width=0.3mm] (linf) to[bend left = 80] (lv2);
\draw[blue,line width=0.3mm] (linf) to[bend left = 90] (lv3);

\draw[blue,line width=0.3mm] (minf) to[out=300,in=240,looseness=12] (minf);
\draw[blue,line width=0.3mm] (minf) to[out=310,in=230,looseness=18] (minf);
\draw[blue,line width=0.3mm] (minf) to[out=320,in=220,looseness=24] (minf);
\draw[blue,line width=0.3mm] (minf) to[out=330,in=210,looseness=30] (minf);

\draw[blue,line width=0.3mm] (rinf) to[bend right = 30] (rv1);
\draw[blue,line width=0.3mm] (rinf) to[bend right = 70] (rv1);
\draw[blue,line width=0.3mm] (rinf) to[bend right = 80] (rv2);
\draw[blue,line width=0.3mm] (rinf) to[bend right = 90] (rv3);
\draw[green,line width=0.3mm] (rinf) to[bend left = 30] (rv1);
\draw[green,line width=0.3mm] (rinf) to[bend left = 70] (rv1);
\draw[green,line width=0.3mm] (rinf) to[bend left = 80] (rv2);
\draw[green,line width=0.3mm] (rinf) to[bend left = 90] (rv3);

\end{tikzpicture}
\]
The left- and right-most spaces are the (one point compactification of the) pullback of the star of a vertex and the centre space is the (one point compactification of the) pullback of an edge.
In particular, the maps are defined so that colours are preserved and otherwise everything maps to the compactification point (in gray).
With an appropriate choice of basis, this means the degree 1 sheaf is made of repeated copies of the following diagram.
In particular, $\Fsheaf_1$ is an episheaf.
\[
\begin{tikzcd}[ampersand replacement=\&,column sep = 6em]
\cdots 
\& \F^5
	\arrow[l]
	\arrow[r,"{\scalemath{\scalesize}{\begin{bsmallmatrix} 1&0&1&0&0\\0&0&1&0&0\\0&0&0&0&1\\0&0&0&1&0 \end{bsmallmatrix}}}" description]
\& \F^4
\& \F^5
	\arrow[l,swap,"{\scalemath{\scalesize}{\begin{bsmallmatrix} 1&1&0&0&0\\0&0&0&1&0\\0&1&0&0&0\\0&0&0&0&1 \end{bsmallmatrix}}}" description]
	\arrow[r,"{\scalemath{\scalesize}{\begin{bsmallmatrix} 1&0&1&0&0\\0&0&1&0&0\\0&0&0&0&1\\0&0&0&1&0 \end{bsmallmatrix}}}" description]
\& \F^4
\& \cdots
	\arrow[l]
\end{tikzcd}
\]
\end{example}

\begin{definition}
Given a sheaf $\Fsheaf$, we say that $\Esheaf$ is a \textit{sub-episheaf} if $\Esheaf$ is an episheaf and there is a monomorphism $\Esheaf\inj\Fsheaf$ in $\sheafcat(X)$.
\end{definition}
For any sheaf $\Fsheaf$, the zero-sheaf $\sheaf{0}$ consisting of trivial objects and trivial morphisms is always a sub-episheaf.
In fact, it is the \textit{smallest} sub-episheaf of $\Fsheaf$.
Conversely, given two sub-episheaves $\Esheaf_1\inj\Fsheaf$ and $\Esheaf_2\inj\Fsheaf$ one can construct another sub-episheaf $\Esheaf_3\inj\Fsheaf$ satisfying $\Esheaf_1\inj\Esheaf_3\hookleftarrow\Esheaf_2$; this is shown explicitly in \cite{macpherson2021persistent} for the case of group-valued bisheaves, but in Appendix~\ref{sec:PLS-lemmas} we present a general construction for abelian categories.
Then via a standard application of Zorn's Lemma, one may also attribute a \textit{maximal} sub-episheaf to $\Fsheaf$.
The construction of a maximal sub-episheaf is known as \textit{epification}.
Epification is functorial, i.e. there exists a functor $\mathbf{Epi}:\sheafcat(X)\to\sheafcat(X)$ which assigns to every sheaf its maximal sub-episheaf, and we note that there is a canonical inclusion natural transformation $\sheaf{\eta}:\mathbf{Epi}\Rightarrow \mathbf{id}_{\sheafcat(X)}$.
Appendix~A of \cite{macpherson2021persistent} details how one can construct the epification of a sheaf, however in Section~\ref{sec:isobisheaf} we will introduce an explicit epification algorithm.

\begin{definition}
A (cellular) cosheaf $\Fcosheaf$ under $X$ valued in $\Ab$ is a covariant functor from the poset category of $X$ to $\Ab$.
If, in addition, $\Fcosheaf(\sigma\leq\tau):\Fcosheaf(\sigma)\to\Fcosheaf(\tau)$ is a monomorphism for every face relation $\sigma\leq\tau$ then we say $\Fcosheaf$ is an monocosheaf.
\end{definition}

Dualising the notion of $\sheafcat(X)$, we denote by $\cosheafcat(X)$ the (abelian) category of cellular cosheaves under $X$.
For the family of cosheaves with which we will work, consider again a cell map $f:Y\to X$.
We can construct a cosheaf analogous to the previous sheaf construction procedure with ordinary homology instead of relative homology.
To every cell $\sigma$ of $X$, we attribute the homology
\[
\Fcosheaf_\bullet(\sigma) = \Hg_\bullet\left(f^{-1}(\St(\sigma))\right).
\]
In addition, the embeddings $\St(\sigma)\inj\St(\tau)$ for each $\sigma\leq\tau$ induce the maps
\[
\Fcosheaf_\bullet(\sigma\leq\tau):\Hg_\bullet\left(f^{-1}(\St(\sigma))\right)\to\Hg_\bullet\left(f^{-1}(\St(\tau))\right),
\]
completing the cosheaf construction.

\begin{example}
Consider again the map $K/n\Z\to \Sp^1$ described in Figure~\ref{fig:running-example} for $n\geq 3$.
Locally, the homology cosheaf under $\Sp^1$ is determined by taking the homology of the following spaces and inclusion maps.
\[
\begin{tikzpicture}
\node (tlmap) at (-1.75,1) {\Huge $\hookrightarrow$};
\node (trmap) at (1.75,1) {\Huge $\hookleftarrow$};

\draw[red!40,line width=0.3mm] (-3.4,0.2) to (-2.6,1.8);
\draw[red,line width=0.3mm] (-3.4,2) to (-2.6,2);
\draw[red,line width=0.3mm] (-3.4,1.9) to (-2.6,1.1);
\draw[red,line width=0.3mm] (-3.4,0.9) to (-2.6,0.1);

\draw[red!40,line width=0.3mm] (-0.9,0.2) to (0,2);
\draw[red,line width=0.3mm] (-0.9,2) to (0,2);
\draw[red,line width=0.3mm] (-0.9,1.9) to (0,1);
\draw[red,line width=0.3mm] (-0.9,0.9) to (0,0);
\draw[blue!40,line width=0.3mm] (0,0) to (0.9,1.8);
\draw[blue,line width=0.3mm] (0,2) to (0.9,1.1);
\draw[blue,line width=0.3mm] (0,2) to (0.9,2);
\draw[blue,line width=0.3mm] (0,1) to (0.9,0.1);

\draw[blue!40,line width=0.3mm] (2.6,0.2) to (3.4,1.8);
\draw[blue,line width=0.3mm] (2.6,2) to (3.4,2);
\draw[blue,line width=0.3mm] (2.6,1.9) to (3.4,1.1);
\draw[blue,line width=0.3mm] (2.6,0.9) to (3.4,0.1);

\node[shape=circle,fill=black,inner sep=0pt,minimum size=0.15cm] (v1) at (0,2) { };
\node[shape=circle,fill=black,inner sep=0pt,minimum size=0.15cm] (v2) at (0,1) { };
\node[shape=circle,fill=black,inner sep=0pt,minimum size=0.15cm] (v3) at (0,0) { };

\end{tikzpicture}
\]
The left- and right-most spaces are the pullback of the star of a vertex and the centre space is the pullback of an edge, and the maps are natural inclusions.
With an appropriate choice of basis, this means the degree 0 cosheaf is made of repeated copies of the following diagram.
In particular, $\Fcosheaf_0$ is \textit{not} a monocosheaf.
\[
\begin{tikzcd}[ampersand replacement=\&,column sep = 6em]
\cdots 
	\arrow[r]
\& \F^3
\& \F^4
	\arrow[l,swap,"{\scalemath{\scalesize}{\begin{bsmallmatrix} 1&1&0&0\\0&0&0&1\\0&0&1&0 \end{bsmallmatrix}}}" description]
	\arrow[r,"{\scalemath{\scalesize}{\begin{bsmallmatrix} 1&0&1&0\\0&1&0&0\\0&0&0&1 \end{bsmallmatrix}}}" description]
\& \F^3
\& \F^4
	\arrow[l,swap,"{\scalemath{\scalesize}{\begin{bsmallmatrix} 1&1&0&0\\0&0&0&1\\0&0&1&0 \end{bsmallmatrix}}}" description]
\& \cdots
	\arrow[l]
\end{tikzcd}
\]
\end{example}

By dualising the notion of epification, one obtains a similar notion for cosheaves.

\begin{definition}
Given a cosheaf $\Fcosheaf$, we say that $\Mcosheaf$ is a \textit{quotient-monocosheaf} if $\Mcosheaf$ is a monocosheaf and $\Fcosheaf\surj\Mcosheaf$.
\end{definition}

For any cosheaf $\Fcosheaf$, the zero-cosheaf $\cosheaf{0}$ is always the \textit{largest}\footnote{In this case, \textit{largest} refers to the fact $\cosheaf{0}\cong \Fcosheaf/\Kcosheaf$ for $\Kcosheaf$ maximal with the property that $\Fcosheaf/\Kcosheaf$ is a monocosheaf.} quotient-monocosheaf of $\Fcosheaf$.
Dual to the case for sheaves, given two quotient-monocosheaves $\Fcosheaf\surj\Mcosheaf_1$ and $\Fcosheaf\surj\Mcosheaf_2$ one can always construct a smaller quotient-monocosheaf $\Fcosheaf\surj\Mcosheaf_3$ satisfying $\Mcosheaf_1\twoheadleftarrow\Mcosheaf_3\surj\Mcosheaf_2$, and an application of Zorn's Lemma proves the existence of a \textit{minimal} quotient-monocosheaf of $\Fcosheaf$.
Again, in Appendix~\ref{sec:PLS-lemmas} we provide a general construction of $\Mcosheaf_3$ for cosheaves valued in abelian categories.
The construction of a minimal quotient-monocosheaf is known as \textit{monofication}.
Monofication is also functorial, i.e. there exists a functor $\mathbf{Mono}:\cosheafcat(X)\to\cosheafcat(X)$ which assigns to every sheaf its maximal sub-episheaf, and we note that there is a canonical quotient natural transformation $\cosheaf{\eta}:\mathbf{id}_{\cosheafcat(X)}\Rightarrow\mathbf{Mono}$.
Appendix~B of \cite{macpherson2021persistent}  details how one can construct the monofication of a cosheaf, however in Section~\ref{sec:isobisheaf} we will introduce an explicit monofication algorithm.



\begin{definition}
A \textit{(co)local system} is a locally constant (co)sheaf.
That is, a (co)local system is a sheaf $\Fsheaf$ (respectively, cosheaf $\Fcosheaf$) with the property that every morphism $\Fsheaf(\sigma\leq\tau):\Fsheaf(\tau)\to\Fsheaf(\sigma)$ ($\Fcosheaf(\sigma\leq\tau):\Fcosheaf(\sigma)\to\Fcosheaf(\tau)$) is an isomorphism.
\end{definition}

We denote by $\localsystemcat(X)$ and $\colocalsystemcat(X)$ the abelian category of local systems and colocal systems (respectively) of $X$.
$\localsystemcat(X)$ is a full subcategory of $\sheafcat(X)$ and $\colocalsystemcat(X)$ is a full subcategory of $\cosheafcat(X)$.
We have that $\localsystemcat(X)$ and $\colocalsystemcat(X)$ are isomorphic (abelian) categories.
For each local system $\Fsheaf$, the isomorphism assigns the colocal system $\Fcosheaf$ defined so that $\Fcosheaf(\sigma)=\Fsheaf(\sigma)$ for each cell $\sigma$ and $\Fcosheaf(\sigma\leq\tau)=\Fsheaf(\sigma\leq\tau)^{-1}$ for each face relation $\sigma\leq\tau$.
For each colocal system $\Fcosheaf$, the isomorphism assigns the local system $\Fsheaf$ analogously.
Finally, as (co)local system morphisms are defined cell-wise, the isomorphism is the ``identity'' between local system morphisms and colocal system morphisms.

\begin{definition}
A (cellular) \textit{bisheaf} around $X$ is a triple $\Fbisheaf=(\Fsheaf,\Fcosheaf,F)$ where $\Fsheaf$ is a sheaf over $X$, $\Fcosheaf$ is a cosheaf under $X$ and $F=\{F_\sigma\,:\,\Fsheaf(\sigma)\to\Fcosheaf(\sigma)\}_{\sigma\in X}$ is a set of morphisms which satisfy the following commutative diagram for every $\sigma\leq \tau$.
\[
\begin{tikzpicture}[thick]
\node (sheaft) at (0,2) { $\Fsheaf(\tau)$};
\node (sheafs) at (4,2) { $\Fsheaf(\sigma)$};
\node (cosheaft) at (0,0) {  $\Fcosheaf(\tau)$};
\node (cosheafs) at (4,0) { $\Fcosheaf(\sigma)$};

\path
(sheaft) edge[->] node [above]  { $\Fsheaf(\sigma\leq\tau)$} (sheafs)
(cosheafs) edge[->] node [above] { $\Fcosheaf(\sigma\leq\tau)$} (cosheaft)
(sheaft) edge[->] node [right]  { $F_\tau$} (cosheaft)
(sheafs) edge[->] node [right]  { $F_\sigma$} (cosheafs);
\end{tikzpicture}
\]
\end{definition}

Consider a cell map $f:Y\to X$ and assume, in addition, that $X$ is the cellulation of a connected, oriented $d$-dimensional manifold.
We construct a bisheaf $\Fbisheaf$ by first assigning $\Fsheaf$ to be the relative homology sheaf, and $\Fcosheaf$ to be the homology cosheaf constructed earlier.
We will then use cap products to construct morphisms $F_\sigma:\Fsheaf(\sigma)\to\Fcosheaf(\sigma)$.

The orientation of $X$ is determined by a generator of the 1-dimensional vector space $\Hg^d_c(X)\cong\F$, where $\Hg^d_c(X)$ is the compactly support cohomology of $X$
\[
\Hg_c^d(X) = \lim_{\mathcal{K}\subseteq X\text{ compact}}\Hg^d(X,X\setminus\mathcal{K}).
\]
Note that if $X$ is compact then $\Hg_c^d(X)\cong\Hg^d(X)$.
Choose $\varphi$ which generates $\Hg^d_c(X)$.
For every cell $\sigma$ of $X$, the inclusion $i^\sigma:\St(\sigma)\inj X$ induces an isomorphism $(i^\sigma)_*:\Hg^d_c(\St(\sigma))\xrightarrow{\cong} \Hg^d_c(X)$.
Set $\varphi^\sigma = (i^\sigma)_*^{-1}\varphi$ so that $\varphi^\sigma$ generates $\Hg^d_c(\St(\sigma))$ and take the pullback $f^*(\varphi^\sigma)=\varphi^\sigma\circ f\in\Hg_c^d\left(f^{-1}(\St(\sigma)\right)$.
Finally, we apply the cap product
\[
\Hg_{\bullet+d}\left(Y,Y\setminus f^{-1}(\St(\sigma))\right)\otimes \Hg_c^d\left(f^{-1}(\St(\sigma))\right)\xrightarrow{\frown} \Hg_\bullet\left(f^{-1}(\St(\sigma))\right)
\]
which is well-defined by applying excision to $\Hg_c^\bullet\left(f^{-1}(\St(\sigma))\right)$ and $\Hg_\bullet\left(Y,Y\setminus f^{-1}(\St(\sigma))\right)$ as necessary.
Thus, we have a well-defined map
\[
F_\sigma=\underline{\hspace{0.3cm}}\frown[f^*(\varphi^\sigma)]:\Hg_{\bullet+d}\left(Y,Y\setminus f^{-1}(\St(\sigma))\right)\to\Hg_\bullet\left(f^{-1}(\St(\sigma))\right)
\]
This completes a canonical construction for a bisheaf attributed to $f:Y\to X$.

\begin{example}
Once more, consider the map $K/n\Z\to \Sp^1$ described in Figure~\ref{fig:running-example} for $n\geq 3$.
One has $\Hg^1(\Sp^1)\cong\F$ and fixing $1$ as the multiplicative identity of $\F$, we choose a generator $[\varphi]$ of $\Hg^1(\Sp^1)$ such that for every cell $\sigma$, the pullback $\varphi^\sigma$ evaluates to $1$ on some compact interval along $\St(\sigma)$ and 0 otherwise.
Without loss of generality, for the vertex $v_m=e^{2\pi im/n}$ we may set $\varphi_{v_m} = \chi_{\{e^{2\pi i\alpha/n}\,:\,\alpha\in[m+1/3,m+2/3]\}}$ (where $\chi_A$ is the indicator function of $A$) and for $e_m=(e^{2\pi im/n},e^{2\pi i(m+1)/n})$ we set $\varphi_{e_m}=\varphi_{v_m}$.
Pulling back to $K/n\Z$ and applying the cap product to relative homology, this means cell-wise that
\begin{itemize}
    \item all edges above $\St(e_m)$ map to their corresponding connected component (after removing the one point compactification)
    \item all edges to the left of the vertices above $\St(v_m)$ map to 0
    \item all edges to the right of the vertices above $\St(v_m)$ map to their corresponding connected component (after removing the one point compactification).
\end{itemize}
One extends this linearly to maps $\Fsheaf_1(\sigma)\to\Fcosheaf_0(\sigma)$ for each cell $\sigma$, which locally locally looks like the following bisheaf with the choices of bases as in the previous examples.
\[
\begin{tikzcd}[ampersand replacement=\&,column sep = 6em, row sep=3em]
\cdots 
\& \F^5
	\arrow[l]
	\arrow[r,"{\scalemath{\scalesize}{\begin{bsmallmatrix} 1&0&1&0&0\\0&0&1&0&0\\0&0&0&0&1\\0&0&0&1&0 \end{bsmallmatrix}}}" description]
	\arrow[d,"{\scalemath{\scalesize}{\begin{bsmallmatrix} 1&0&0&0&0\\0&0&0&1&0\\0&0&0&0&1 \end{bsmallmatrix}}}" description]
\& \F^4
	\arrow[d,"{id}" description]
\& \F^5
	\arrow[l,swap,"{\scalemath{\scalesize}{\begin{bsmallmatrix} 1&1&0&0&0\\0&0&0&1&0\\0&1&0&0&0\\0&0&0&0&1 \end{bsmallmatrix}}}" description]
	\arrow[r,"{\scalemath{\scalesize}{\begin{bsmallmatrix} 1&0&1&0&0\\0&0&1&0&0\\0&0&0&0&1\\0&0&0&1&0 \end{bsmallmatrix}}}" description]
	\arrow[d,"{\scalemath{\scalesize}{\begin{bsmallmatrix} 1&0&0&0&0\\0&0&0&1&0\\0&0&0&0&1 \end{bsmallmatrix}}}" description]
\& \F^4
	\arrow[d,"{id}" description]
\& \cdots
	\arrow[l]
\\
\cdots 
	\arrow[r]
\& \F^3
\& \F^4
	\arrow[l,swap,"{\scalemath{\scalesize}{\begin{bsmallmatrix} 1&1&0&0\\0&0&0&1\\0&0&1&0 \end{bsmallmatrix}}}" description]
	\arrow[r,"{\scalemath{\scalesize}{\begin{bsmallmatrix} 1&0&1&0\\0&1&0&0\\0&0&0&1 \end{bsmallmatrix}}}" description]
\& \F^3
\& \F^4
	\arrow[l,swap,"{\scalemath{\scalesize}{\begin{bsmallmatrix} 1&1&0&0\\0&0&0&1\\0&0&1&0 \end{bsmallmatrix}}}" description]
	\arrow[r]
\& \cdots
\end{tikzcd}
\]
\end{example}

We denote by $\bisheafcat(X)$ the category of cellular bisheaves around $X$.
Morphisms $\varphi:\Fbisheaf\to\bisheaf{G}$ of $\bisheafcat(X)$ consist of a sheaf map $\sheaf{\varphi}:\Fsheaf\to\sheaf{G}$ and a cosheaf map $\cosheaf{\varphi}:\Fcosheaf\to\cosheaf{G}$ which commute with the bisheaf structures.
$\bisheafcat(X)$ is an abelian category, as it inherits many of the properties of $\sheafcat(X)$ and $\cosheafcat(X)$.

\begin{remark}
Our definition of $\bisheafcat(X)$ differs from the original definition in \cite{macpherson2021persistent}, where $\varphi:\Fbisheaf\to\bisheaf{G}$ was said to consist of consist of a sheaf map $\sheaf{\varphi}:\Fsheaf\to\sheaf{G}$ and a cosheaf map $\cosheaf{\varphi}:\cosheaf{G}\to\Fcosheaf$.
We make this change so that $\bisheafcat(X)$ can be an abelian category with our definiton, which is not the case with the previous definition, since such a category does not contain pushouts.
\end{remark}

\begin{definition}[\cite{macpherson2021persistent}]
Let $\Fbisheaf=(\Fsheaf,\Fcosheaf,F)$ be a bisheaf. $\Fbisheaf$ is an \textit{isobisheaf} if $\Fsheaf$ is an episheaf and $\Fcosheaf$ is a monocosheaf.
\end{definition}


Given a bisheaf $\Fbisheaf=(\Fsheaf,\Fcosheaf,F)$, after epifying $\Fsheaf$ and monofying $\Fcosheaf$, we obtain an isobisheaf $\Ibisheaf=(\Esheaf,\Mcosheaf,I)$ which satisfies the following canonical commutative diagram.
\[
\begin{tikzpicture}[scale=0.9]
\node (E) at (0,5.4) {$\Esheaf(\tau)$};
\node (Fup) at (0,3.6) {$\Fsheaf(\tau)$};
\node (Fdo) at (0,1.8) {$\Fcosheaf(\tau)$};
\node (M) at (0,0) {$\Mcosheaf(\tau)$};
\node (Ep) at (3,5.4) {$\Esheaf(\sigma)$};
\node (Fpup) at (3,3.6) {$\Fsheaf(\sigma)$};
\node (Fpdo) at (3,1.8) {$\Fcosheaf(\sigma)$};
\node (Mp) at (3,0) {$\Mcosheaf(\sigma)$};

\draw [right hook->] (E) -- (Fup);
\draw [right hook->] (Ep) -- (Fpup);

\path[every node/.style={font=\sffamily\small}]
(E) edge[->>] node [above] {\footnotesize $\Esheaf(\sigma\leq\tau)$} (Ep)
(Fup) edge[->] node [above] {\footnotesize $\Fsheaf(\sigma\leq\tau)$} (Fpup)
(Fpdo) edge[->] node [above] {\footnotesize $\Fcosheaf(\sigma\leq\tau)$} (Fdo)
(Mp) edge[left hook->] node [above] {\footnotesize $\Mcosheaf(\sigma\leq\tau)$} (M)
(Fdo) edge[->>] (M)
(Fpdo) edge[->>] (Mp)
(Fup) edge[->] node [right]  {\footnotesize $F_\tau$} (Fdo)
(Fpup) edge[->] node [left]  {\footnotesize $F_\sigma$} (Fpdo)
(E) edge[->,bend right = 40] node [left] {\footnotesize $I_\tau$} (M)
(Ep) edge[->,bend left = 40] node [right] {\footnotesize $I_\sigma$} (Mp);
\end{tikzpicture}
\]
The process of obtaining $\Ibisheaf$ from $\Fbisheaf$ is known as \textit{isobisheafification}.

Another important property of an isobisheaf $\Ibisheaf$ that we will use extensively throughout this paper is that the image of $I$ is a well-defined colocal system which embeds into $\Mcosheaf$.

\begin{lemma}\label{lem:colocal}
Let $\Ibisheaf = (\Esheaf,\Mcosheaf,I)$ be an isobisheaf valued in an abelian category $\Ab$.
Then image of $I$ defines a colocal system which we call the \emph{persistent local system} of $\Ibisheaf$.
\end{lemma}
This is proved in \cite[Proposition~5.6]{macpherson2021persistent} for the case that the persistent local system is valued in a category induced by additional structure on sets (e.g. groups, modules, vector spaces).
In Appendix~\ref{sec:PLS-lemmas} we present a concise proof in full generality for abelian categories.

The persistent local system $\local$ of an isobisheaf is a sub-monocosheaf of $\Mcosheaf$.
Dually, if we were to take the coimage of $F$ then we would obtain a quotient-epicosheaf of $\Esheaf$.
This corresponding local system is exactly that which one obtains via the isomorphism of categories $\colocalsystemcat(X)\cong\localsystemcat(X)$.
Observe also that isobisheafification is a functor $\bisheafcat(X)\to\mathbf{iso}\bisheafcat(X)$ from bisheaves to the full subcategory of isobisheaves, and every PLS is an object in the image of a canonical functor $\mathbf{iso}\bisheafcat(X)\to\colocalsystemcat(X)$.
The composition of these functors is then a canonical functor $\bisheafcat(X)\to\colocalsystemcat(X)$ assigning to every (not necessarily iso-) bisheaf $\Fbisheaf$ a persistent local system $\local$.

\subsection{Periodic Spaces}

\begin{definition}\label{def:periodic}
$K$ is a $d$-periodic complex if it is endowed with a free action of $\Z^d$.
\end{definition}

For the remainder of this paper, we assume $K$ is a locally compact, paracompact $d$-periodic complex.

We denote by $T\subset \mathrm{Aut}(K)$ the translation group of automorphisms represented by this action and note that the action of $\Z^d$ on $K$ need not be maximal or obey any specific requirements.
Our aim is to study the homology of $K$ via its quotient space $G=K/\sim$, where $\sim$ is the relation $x\sim y$ if and only if $y=\tran(x)$ for some $\tran\in T$.
For a more thorough introduction to periodic spaces and their quotient spaces, we direct the reader to \cite{onus2022quantifying} and \cite{onus2023computing}.

If $K$ is embedded in some Euclidean space $\R^n$ for $n\geq d$ such that $T$ extends to translations on $\R^n$, then $q$ extends to a map $q:\R^n\surj \torus^d\times\R^{n-d}$, where $\torus^d=(\R/\Z)^d$ is the $d$-dimensional torus.
In this way, one has a canonical embedding of $G$ in $\torus^d\times\R^{n-d}$.
One then links this to bisheaf theory by composing with the natural projection map $\torus^d\times\R^{n-d}\surj\torus^d$ to obtain a canonical map $\pi:G\to\torus^d$ as outlined in the following commutative diagram.
\begin{equation}
\label{diag:projections}
\begin{tikzcd}[column sep={2cm,between origins},row sep={1cm,between origins}]
         K & \R^n & \\
         & & \torus^d\\
         G & \torus^d\times\R^{n-d} & \\
        \arrow[hook, from=1-1, to=1-2]
        \arrow[hook, dashed, from=3-1, to=3-2]
        \arrow[two heads, "q", from=1-1, to=3-1]
        \arrow[two heads, "q", from=1-2, to=3-2]
        \arrow[two heads, from=3-2, to=2-3]
        \arrow[dashed, from=1-2, to=2-3]
\end{tikzcd}
\end{equation}
\begin{remark}
In practice, we do not need to assume that $K\inj\R^n$ is an embedding, as long as the commutative diagram is satisfied.
For example, if $K$ is a Vietoris-Rips complex of a periodic point cloud, the map $K\to\R^n$ sending $K$ to its geometric realization satisfies this property and all of the following results will still apply.
\end{remark}
For ease of notation, we now write $\R/\Z=\Sp^1$ to denote the circle.
As $\torus^d=\Sp^1\times\Sp^1\times\cdots\times\Sp^1$, the $d$-torus is equipped with many natural projections onto subsets of it's coordinates.
Specifically, for every non-empty subset $L\subseteq\{1,2,\dots,d\}$ with $\ell=|L|$, there is a natural projection $\rho_L:\torus^d\to\torus^\ell$ obtained by projecting onto the copies of $\Sp^1$ indicated by $L$.
These projection maps satisfy the property that if $L,L'$ are disjoint (non-empty) subsets of $\{1,2,\dots,d\}$, then $\rho_{L\cup L'}=\rho_L\circ\rho_{L'}=\rho_{L'}\circ\rho_L$ (where we have slightly abused notation).
We thus have hierarchical family of maps $\pi_L:G\to\torus^\ell$.

A particular point of interest is in determining when cycles of $G$ do not represent cycles of $K$, which leads to the following definition.
\begin{definition}
A \emph{non-toroidal cycle} of $G$ is a cycle representing an element of $\Ig_\bullet:=q_*\left(\Hg_\bullet(K)\right)$, where $q_*:\Hg_\bullet(K)\to\Hg_\bullet(G)$ is the map induced by the projection $q:K\to G$.
Conversely, a \emph{toroidal cycle} of $G$ is a cycle representing a non-trivial element of $\Hg_\bullet(G)/\Ig_\bullet$ (i.e. the cokernel of $q_*$).
\end{definition}

Note that $\Hg_\bullet(K)$ will in general be infinite dimensional, which makes algebraically calculating the cokernel of $q_*$ -- and hence determining toroidal cycles -- highly non-trivial.
Geometrically, toroidal cycles arise because of the ambient toroidal topology from the embedding $G\inj \torus^d\times\R^{n-d}$.
Note, however, that toroidal cycles are not necessarily cycles of this ambient torus.
For instance, the toroidal 2-cycle of the quotient space of the cylinders in Figure~\ref{fig:toroidal-vs-non} will be a trivial cycle in $\torus^2\times\R$ but the projection of the plane will be exactly the generator of $\Hg_2(\torus^2\times\R)$.
This means that the problem of classifying toroidal cycles is subtle and much more difficult than simply finding the preimage of homology classes of the torus.

\begin{figure}
    \centering
    \includegraphics[scale=0.86]{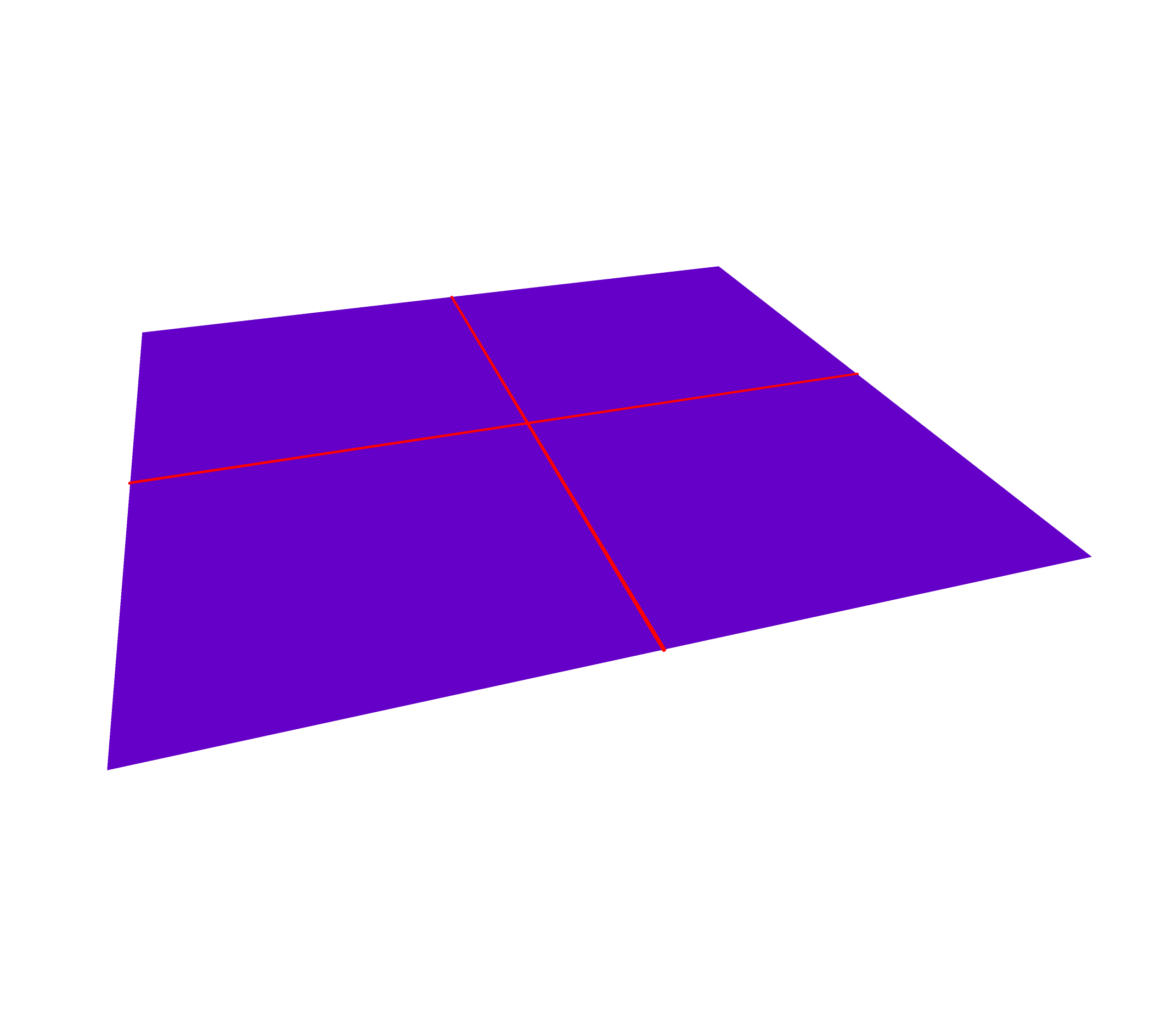}
    \includegraphics[scale=0.7]{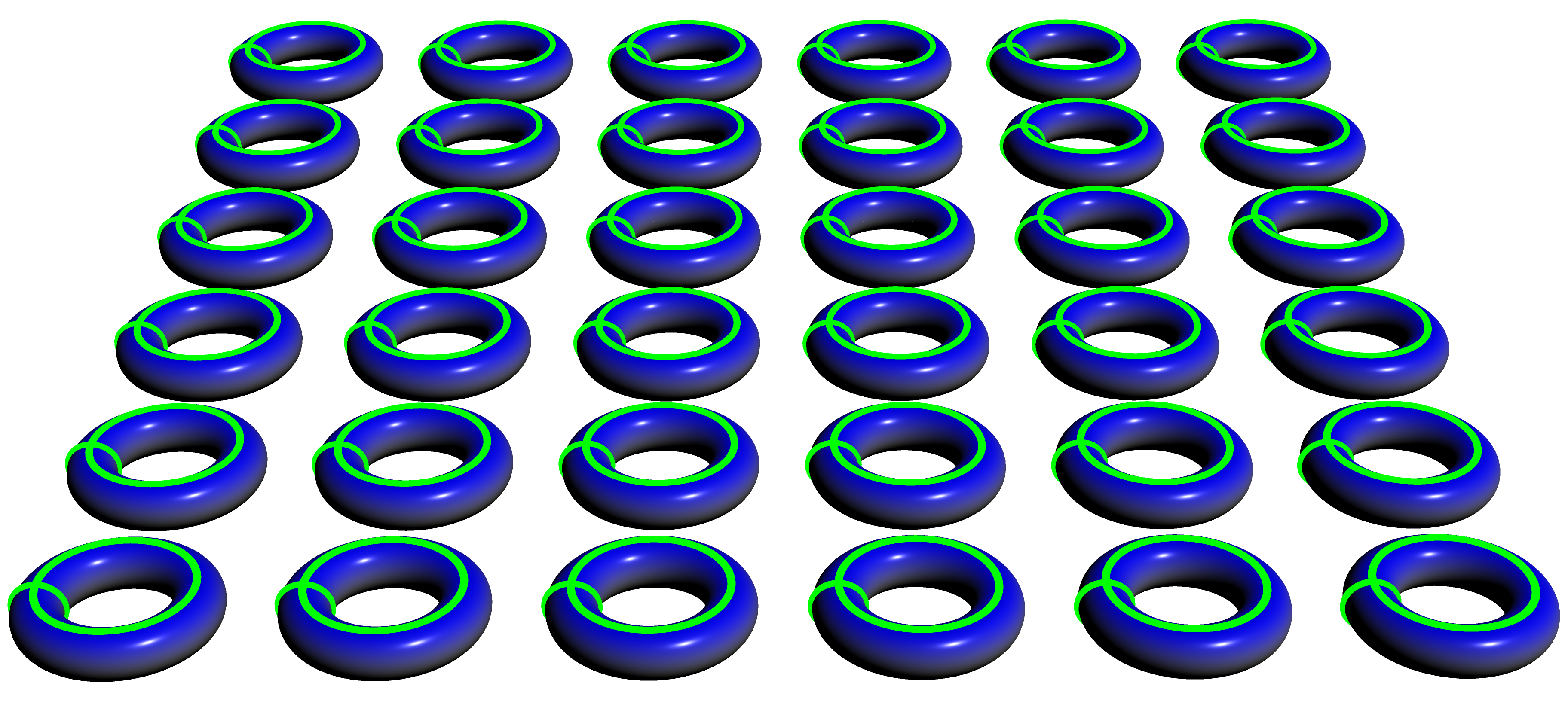}
    \includegraphics[scale=1.55]{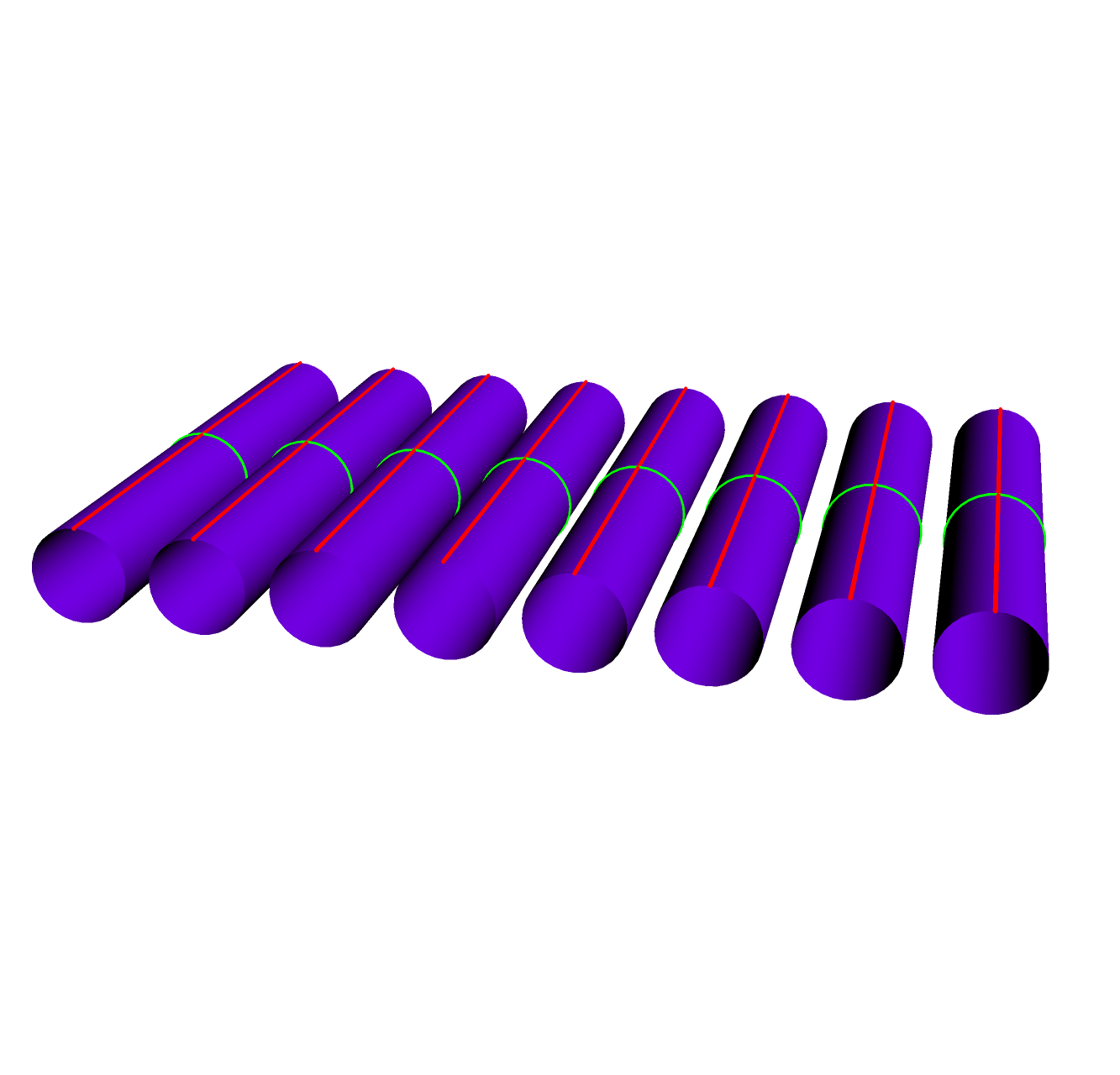}
    \caption{Three examples of 2-dimensional 2-periodic spaces (implied to extend out indefinitely) in $\R^3$ which project onto a standard 2-torus $\torus^2$ in their quotient spaces with respect to translations of $\Z^2\times 0$. Left: $\R^2\times 0$, where the red lines project onto toroidal 1-cycles and the purple surface projects onto a toroidal 2-cycle. Center: disjoint copies of $\torus^2$, where the quotient action projects the green circles onto non-toroidal 1-cycles, and the blue surfaces onto a non-toroidal 2-cycles. Right: disjoint infinite cylinders, where the quotient action projects the purple surfaces onto a toroidal 2-cycle, the red lines onto a toroidal 1-cycle, and the green circles onto a non-toroidal 1-cycle.}
    \label{fig:toroidal-vs-non}
\end{figure}

\section{Classifying Toroidal Cycles} \label{sec:toroidal}
We describe a classification for detecting toroidal cycles from finite information.
Our main results are an embedding Theorem for 1-periodic toroidal cycles (Theorem~\ref{thm:1d-toroidal}) and an embedding Theorem for the more general case of $d$-periodic toroidal cycles (Theorem~\ref{thm:toroidal}).
The section is divided into two subsections -- we first describe the 1-peoridic case, as this allows us then to build up to the more general case.


\subsection{The 1-periodic case}
In this subsection, we investigate the simplest case of a periodic space --- where the translation group is $T\cong \Z$. While this case was studied in \cite{onus2023computing} using different machinery, the resulting information is different from \cite{onus2023computing} and proves instrumental to understanding higher dimensional translation actions. Here we 
heavily rely on Algorithm~\ref{alg:epi}, Algorithm~\ref{alg:mono} and Theorem~\ref{thm:algorithms} in the proofs of Lemma~\ref{lem:unravel}, Lemma~\ref{lem:1d-nontoroidal} and Theorem~\ref{thm:1d-toroidal}. For an in-depth understanding of the proofs, the reader may wish to read Section~\ref{sec:isobisheaf} first before continuing.

Observe that for any cell $\sigma$, and any orientation $\varphi$ of $\torus^d$ with (compactly supported) pullback $\varphi^\sigma$ of $\pi^{-1}(\St(\sigma))$ we have the following commutative diagram
\[
\begin{tikzcd}[column sep={3.5cm,between origins},row sep={1.5cm,between origins}]
         {H_{\bullet+d}(G)} & {\Hg_{\bullet+d}\left(G,G\setminus f^{-1}(\St(\sigma))\right)} \\
         {H_{\bullet}(G)} & {H_{\bullet}(\pi^{-1}(\St(\sigma)))}
        \arrow["\pi^*", from=1-1, to=1-2]
        \arrow["\pi_*", from=2-1, to=2-2]
        \arrow[two heads, "\frown\varphi", from=1-1, to=2-1]
        \arrow[two heads, "\frown\varphi^\sigma", from=1-2, to=2-2]
\end{tikzcd}
\]
where horizontal maps are induced by inclusion and we recall $\Hg^d_c(X)$ denotes the compactly supported cohomology of $X$.
Let $\Fbisheaf$ be the bisheaf associated to $\pi$ and $\local(G)$ the corresponding persistent local system (PLS).
By construction, the image of $\Hg_\bullet(G)\to \Hg_\bullet\left(G,G\setminus f^{-1}(\St(\sigma))\right)$ is a sub-episheaf of $\Fsheaf$.
Dually, the coimage of $\Hg_\bullet(\pi^{-1}(\St(\sigma))) \to \Hg_\bullet(G)$ is a quotient-monocosheaf of $\Fcosheaf$, where we recall that the coimage of $T:V\to W$ in the category of vector spaces is the space $V/\ker(T)$.
Recall now that $\Fsheaf_\bullet(\sigma)=\Hg_\bullet\left(G,G\setminus f^{-1}(\St(\sigma))\right)$ and $\Fcosheaf_\bullet=H_{\bullet}(\pi^{-1}(\St(\sigma)))$, and for ease of notation in what follows set $\sigma_{\St}=\pi^{-1}(\St(\sigma))$.
This suggests the full picture of the diagram above extends to the following diagram
\begin{equation}
\begin{tikzcd}[column sep={1.4cm,between origins},row sep={1.2cm,between origins}]
	&& {\Hg_{k+d+1}(G)} &&& {\Hg_{k+d}(G)} &&& {\Hg_{k+d-1}(G)} \\
	{\rotatebox{60}{$\ddots$}} &&& {\Hg_{k+d}(G\setminus \sigma_{\St})} &&& {\Hg_{k+d-1}(G\setminus \sigma_{\St})} &&& {\rotatebox{60}{$\ddots$}} \\
	& {\Fsheaf_{k+d+1}(\sigma)} &&& {\Fsheaf_{k+d}(\sigma)} &&& {\Fsheaf_{k+d-1}(\sigma)} \\
	&& {\Hg_{k+1}(G)} &&& {\Hg_{k}(G)} &&& {\Hg_{k-1}(G)} \\
	\ddots &&& {\Hg_{k+1}(G,\sigma_{\St})} &&& {\Hg_{k}(G,\sigma_{\St})} &&& \ddots \\
	& {\Fcosheaf_{k+1}(\sigma)} &&& {\Fcosheaf_{k}(\sigma)} &&& {\Fcosheaf_{k-1}(\sigma)}
	\arrow[from=1-3, to=3-2]
	\arrow["\frown\varphi"{pos=0.7}, from=1-3, to=4-3]
	\arrow[from=1-6, to=3-5]
	\arrow["\frown\varphi"{pos=0.7}, from=1-6, to=4-6]
	\arrow[from=1-9, to=3-8]
	\arrow["\frown\varphi"{pos=0.7}, from=1-9, to=4-9]
	\arrow[from=2-1, to=1-3]
	\arrow[from=2-4, to=1-6]
	\arrow[from=2-7, to=1-9]
	\arrow[from=3-2, to=2-4,crossing over]
	\arrow["{\frown\varphi^\sigma}", from=3-2, to=6-2]
	\arrow[from=3-5, to=2-7,crossing over]
	\arrow["{\frown\varphi^\sigma}", from=3-5, to=6-5]
	\arrow[from=3-8, to=2-10,crossing over]
	\arrow["{\frown\varphi^\sigma}", from=3-8, to=6-8]
	\arrow[from=4-3, to=5-4]
	\arrow[from=4-6, to=5-7]
	\arrow[from=4-9, to=5-10]
	\arrow[from=5-1, to=6-2]
	\arrow[from=5-4, to=6-5]
	\arrow[from=5-7, to=6-8]
	\arrow[from=6-2, to=4-3]
	\arrow[from=6-5, to=4-6]
	\arrow[from=6-8, to=4-9]
 \end{tikzcd}
 \label{diag:LES}
\end{equation}
where the top zig-zag is the long exact sequence of homologies for the pair $(G,G\setminus \sigma_{\St})$ and the bottom zigzag is the long exact sequence for the pair $(G,\sigma_{\St})$.
These facts combined lead us to study toroidal and non-toroidal cycles through $\Fbisheaf$ and vice versa, and in particular by their representatives in $\local(G)$.

\begin{remark} \label{rem:epi-co}
In this paper we work with cellular bisheaves, which form a restricted class of bisheaves.
Traditionally, a pre(co)sheaf over (respectively, under) $X$ is a contravariant (covariant) functor $\Fsheaf:\mathrm{Open}(X)^\mathrm{op}\to \mathcal{A}$ ($\Fcosheaf:\mathrm{Open}(X)\to \mathcal{A}$), where $\mathcal{C}^\mathrm{op}$ is the dual category of $\mathcal{C}$ and $\mathrm{Open}(X)$ is the poset of open sets of $X$.
If, in addition, $\Fsheaf(U)$ ($\Fcosheaf(U)$) is isomorphic to the colimit (limit) of every open cover of $U$ for every $U\in \mathrm{Open}(X)$ we say that $\Fsheaf$ ($\Fcosheaf$) is a (co)sheaf.
The theory of bisheaves is defined analogously in this more general setting to the cellular setting, however the maximal episheaf of $\Fsheaf$ will explicitly be $\Esheaf(U)=\im(\Fsheaf(U\subseteq X))$ and the minimal quotient-monocosheaf of $\Fcosheaf$ will be $\coim(\Fcosheaf(U\subseteq X))$.
For the bisheaf associated to to the map $\pi:G\to \torus^d$, this means $\Esheaf_\bullet(U)=\im\left(\Hg_\bullet(G)\to\Hg_\bullet\left(G,G\setminus \sigma_{\St}\right)\right)$ and $\Mcosheaf_\bullet(U)=\coim(\Hg_\bullet(U)\to\Hg_\bullet(G))$.
However, we cannot a priori pullback/pushout a cellular (co)sheaf to all of $X$, so in the cellular setting we will only get an embedding $\im\left(\Hg_\bullet(G)\to \Hg_\bullet\left(G,G\setminus \sigma_{\St}\right)\right)\inj\Esheaf_\bullet(\sigma)$ and a projection $\Mcosheaf_\bullet(G)\surj\coim\left(\Hg_\bullet(\pi^{-1}(\St(\sigma))) \to \Hg_\bullet(G)\right)$.
For a complete treatment of cellular sheaves and cosheaves, we direct the reader to the thesis of Allen Shepherd \cite{shepard1985cellular} or the thesis of Justin Curry for a more accessible introduction \cite{curry2014sheaves}.
\end{remark}

While one may wish to simply work with simpler objects than bisheaves and isobisheafification, it turns out that this is not sufficient for detecting toroidal cycles, as we will see in the example in Section~\ref{sec:needing-sheaves}.
In this example, the obstruction for a cycle to be toroidal can be non-local, requiring the global properties of an isobisheaf.

Recall Diagram~\ref{diag:projections} and note that the canonical map $G\to \Sp^1$ for 1-periodic cell complexes lifts to a canonical map $K\to \R$ via standard covering space theory.
Choosing a cellulation of $\R$ inherited from $\Sp^1$, we gain another bisheaf $\Rbisheaf{F}=(\Rsheaf{F},\Rcosheaf{F},\dot{F})$.
We also set $\Rbisheaf{I}=(\Rsheaf{E},\Rcosheaf{M},\dot{I})$ to be the canonical isobisheaf and $\local(K)$ the canonical PLS associated to $\Rbisheaf{F}$ (where $\local(K)=\Rcosheaf{M}\cap\im(\dot{I})$).
The notation, $\Rbisheaf{F}$, indicates that this is a lift of $\Fbisheaf$ and in particular that the bisheaves are locally isomorphic.

\begin{lemma} \label{lem:unravel}
If the cellulation of $\Sp^1$ has at least two vertices then the bisheaf $\Rbisheaf{F}$ is locally isomorphic to the bisheaf $\Fbisheaf$.
That is, for any cell relations $\sigma\leq\tau$ of $\Sp^1$ and $\Tilde{\sigma}\leq\Tilde{\tau}$ of $\R$ such that $q(\Tilde{\sigma})=\sigma$ and $q(\Tilde{\tau})=\tau$
\begin{itemize}
    \item $\Rsheaf{F}(\tilde{\sigma})\cong \Fsheaf(\sigma)$ and $\Rcosheaf{F}(\tilde{\sigma})\cong \Fcosheaf(\sigma)$
    \item $\dot{F}_{\tilde{\sigma}}$ is equivalent to $F_\sigma$
    \item $\Rsheaf{F}(\Tilde{\sigma}\leq\Tilde{\tau})$ is equivalent to $\Fsheaf(\sigma\leq\tau)$ and $\Rcosheaf{F}(\Tilde{\sigma}\leq\Tilde{\tau})$ is equivalent to $\Fcosheaf(\sigma\leq\tau)$
\end{itemize}
Moreover, their associated isobisheaves and PLS's are also locally isomorphic.
\end{lemma}
\begin{proof}
Let $\sigma\leq\tau$ and $\Tilde{\sigma}\leq\Tilde{\tau}$ be as above.
Since $K\to G$ is a covering space it is a local isomorphism, and since $\Sp^1$ has at least two vertices guarantees $\St(\sigma)\subsetneq G$, this means $\St(\tilde{\sigma})\cong \St(\sigma)$.
Similarly, the map $\St(\tilde{\sigma})\inj\St(\tilde{\tau})$ is equivalent to $\St(\sigma)\inj\St(\tau)$.
Applying the (relative) homology functors then tells us $\Rbisheaf{F}$ and $\Fbisheaf$ are locally isomorphic\footnote{If $\Sp^1$ had a cellulation with only one vertex $v$ then $\St(v)=G$ and $\St(\tilde{v})=K\cap[0,1]\times\R^{n-1}$ so this isomorphism is no longer guaranteed}.

Now consider the Algorithm~\ref{alg:epi} and Algorithm~\ref{alg:mono}.
We initialise $\Fsheaf_0:=\Fsheaf$ and $\Rbisheaf{F}_0:=\Rbisheaf{F}$ for the Epification Algorithm and let the index $j$ denote the (co)sheaves after the $j^\mathrm{th}$ iteration of the algorithm.
Both sheaves are locally isomorphic at the initial step, so suppose that $\Fsheaf_j$ and $\Rsheaf{F}_j$ are locally isomorphic and consider $\Fsheaf_{j+1}$ and $\Rsheaf{F}_{j+1}$.
Let $e,f$ be edges of $\Sp^1$ and $v,w$ vertices such that $e\leq v$, $e\leq w$ and $f\leq v$, and similarly let $\tilde{e}\in q^{-1}(e)$, $\tilde{f}\in q^{-1}(f)$, $\tilde{v}\in q^{-1}(v)$ and $\tilde{w}\in q^{-1}(w)$ be such that $\tilde{e}\leq\tilde{v}$, $\tilde{e}\leq\tilde{w}$ and $\tilde{f}\leq\tilde{v}$.
Then Algorithm~\ref{alg:epi} says
\begin{align*}
\Fsheaf_{j+1}(v) & = \Fsheaf_j(e\leq v)^{-1}\left(\Fsheaf_j(e)\right)\cap\Fsheaf_j(f\leq v)^{-1}\left(\Fsheaf_j(f)\right) \\
\Rsheaf{F}_{j+1}(\tilde{v}) & = \Rsheaf{F}_j(\tilde{e}\leq \tilde{v})^{-1}\left(\Rsheaf{F}_j(\tilde{e})\right)\cap\Rsheaf{F}_j(\tilde{f}\leq \tilde{v})^{-1}\left(\Rsheaf{F}_j(\tilde{f})\right) \\
\Fsheaf_{j+1}(e) & = \Fsheaf_j(e\leq v)\left(\Fsheaf_{j+1}(v)\right)\cap\Fsheaf_j(e\leq w)\left(\Fsheaf_{j+1}(w)\right) \\
\Rsheaf{F}_{j+1}(\tilde{e}) & = \Rsheaf{F}_j(\tilde{e}\leq \tilde{v})\left(\Rsheaf{F}_{j+1}(\tilde{v})\right)\cap\Rsheaf{F}_j(\tilde{e}\leq \tilde{w})\left(\Rsheaf{F}_{j+1}(\tilde{w})\right) \\
\Fsheaf_{j+1}(e\leq v) & = \Fsheaf_j(e\leq v)\big|_{\Fsheaf_{j+1}(v)} \\
\Rsheaf{F}_{j+1}(\tilde{e}\leq \tilde{v}) & = \Rsheaf{F}_j(\tilde{e}\leq \tilde{v})\big|_{\Rsheaf{F}_{j+1}(\tilde{v})}
\end{align*}
Since $\Fsheaf_j$ and $\Rsheaf{F}_j$ are locally isomorphic and the above equations are all determined by intersections, restrictions and (pre)images of equivalent maps, $\Fsheaf_{j+1}$ and $\Rsheaf{F}_{j+1}$ must also be locally isomorphic.
Theorem~\ref{thm:algorithms} says the Epification Algorithm applied to $\Fsheaf$ terminates, therefore the algorithm must also terminate for $\Rsheaf{F}$.
In particular, $\Esheaf$ and $\Rsheaf{E}$ must be locally isomorphic.
A dual argument shows the monofication of $\Fcosheaf$ and $\Rcosheaf{F}$ are locally isomorphic.
Finally, the restriction of $F_\sigma$ and $\dot{F}_\sigma$ to $\Esheaf(\sigma)$ and $\Rsheaf{E}(\tilde{\sigma})$ (respectively) are equivalent, so their images are isomorphic and hence the PLS's are isomorphic.
\end{proof}

In the same way covering space theory allows us to alternate between looking at the periodic space and its quotient space, Lemma~\ref{lem:unravel} allows us to freely alternate between the corresponding bisheaves and PLS's.
The local isomorphisms means we need only look at the finite bisheaf over $\Sp^1$ to compute information about $\Hg_\bullet(K)$, but we can pass to the bisheaf over $\R$ to determine theoretical properties of this.

In the following result, we utilise the correspondence between $\Hg_\bullet(G)$ and the isobisheaf $\Ibisheaf$ around $G$ alluded to in Diagram~\ref{diag:LES} and Remark~\ref{rem:epi-co}.
Explicitly, each cycle of $G$ will correspond to a subclass of $\Esheaf$ by applying homology to the pairs $(G,\pi^{-1}(\St(\sigma)))\subsetneq (G,G)$ for each cell $\sigma$ to construct a representative of each cycle over each cell.

Notice that it is specifically a subclass of the episheaf by functoriality of homology applied to the family of pairs $(G,\sigma_{\St})\subseteq (G,\tau_{\St}) \subsetneq (G,G)$.
This is summarised in the bottom two rows of Diagram ~\ref{diag:correspondence}.
In turn, we can map down to the monocosheaf to in turn have a correspondence of $\Hg_\bullet(G)$ in the PLS, whose structure is given by the bottom row of Diagram~\ref{diag:correspondence}.

\begin{equation}
\begin{tikzcd}[column sep={1.2cm,between origins},row sep={1.4cm,between origins}]
	& & & {\Hg_{\bullet+1}(G)} & & & \\
        {\Esheaf_{\bullet+1}(\tau)} & & & &  & & {\Esheaf_{\bullet+1}(\sigma)}\\
        {\local(G)_\bullet(\tau)} & & & & & & {\local(G)_\bullet(\sigma)} \\
	\arrow["{\Hg_{\bullet+1}\left((G,G\setminus \tau_{\St})\supset  (G,\emptyset)\right)}"',from=1-4, to=2-1]
	\arrow["{\Hg_{\bullet+1}\left((G,G\setminus \sigma_{\St})\supset  (G,\emptyset)\right)}",from=1-4, to=2-7]
	\arrow[two heads,"{\Hg_{\bullet+1}\left((G,G\setminus \tau_{\St})\supseteq  (G,G\setminus \sigma_{\St})\right)}",from=2-1, to=2-7]
	\arrow[hook,"{\Hg_{\bullet}\left(\sigma_{\St}\subseteq \tau_{\St}\right)}",from=3-7, to=3-1]
	\arrow["{\frown\varphi^\tau}", from=2-1, to=3-1]
	\arrow["{\frown\varphi^\sigma}", from=2-7, to=3-7]
 \end{tikzcd}
 \label{diag:correspondence}
\end{equation}

\begin{lemma} \label{lem:1d-nontoroidal}
Diagram~\ref{diag:correspondence} restricts to a trivial representation of non-toroidal cycles in the PLS $\local(G)$.
\end{lemma}
\begin{proof}
Similar to Diagram~\ref{diag:correspondence}, there is a correspondence between cycles of $K$ and subclasses of $\Rsheaf{E}$ by maps $\Hg_\bullet(K)\to\Rsheaf{E}(\tilde{\sigma})$.
Let $\gamma$ be a non-toroidal cycle of $G$ and let $\tilde{\gamma}$ be a cycle such that $q(\tilde{\gamma})=\gamma$.
For ease of notation, we set $\gamma_\sigma$ to be the representative of $\gamma$ in $\Esheaf(\sigma)$ and $\tilde{\gamma}_{\tilde{\sigma}}$ the representative of $\tilde{\gamma}$ in $\Rsheaf{E}(\tilde{\sigma})$.
If $[\gamma_\sigma]=0$ for each cell $\sigma$ then the result is trivial, so otherwise assume $[\gamma_\sigma]\neq 0$ for some $\sigma$.
By construction, the following diagram commutes
\[
\begin{tikzcd}[column sep={3cm,between origins},row sep={1.5cm,between origins}]
         {H_{\bullet}(K)} & {\bigoplus_{\tilde{\sigma}\in q^{-1}(\sigma)}\Rsheaf{E}(\tilde{\sigma})} \\
         {H_{\bullet}(G)} & {\Esheaf(\sigma)}
        \arrow[from=1-1, to=1-2]
        \arrow[from=2-1, to=2-2]
        \arrow[from=1-1, to=2-1]
        \arrow[from=1-2, to=2-2]
\end{tikzcd}
\]
where $\bigoplus_{\tilde{\sigma}\in q^{-1}(\sigma)}\Rsheaf{E}(\tilde{\sigma})\to \Esheaf(\sigma)$ is given by the collections of isomorphisms $\Rsheaf{E}(\tilde{\sigma})\cong \Esheaf(\sigma)$ identified in Lemma~\ref{lem:unravel}.
Thus, it suffices to show that the representative of $\tilde{\gamma}$ in $\Rsheaf{E}(\tilde{\sigma})$ is in $\ker \dot{I}_{\tilde{\sigma}}$ for each $\tilde{\sigma}$.

$\tilde{\gamma}$ must have finite support, so let $\tilde{\sigma}$ be the smallest valued cell on $\R$ for which $[\tilde{\gamma}_{\tilde{\sigma}}]\neq 0$ and $\tilde{\tau}$ be the greatest valued cell for which this is true.
If $\tilde{\sigma}$ is a vertex $\tilde{v}$ and $\tilde{e}\leq\tilde{v}$ for $\tilde{e}$ more negative in $\R$ than $\tilde{v}$, then $[\tilde{\gamma}_{\tilde{e}}]=0$.
This means
\[
\dot{I}_v [\tilde{\gamma}_{\tilde{v}}] = \left(\Rcosheaf{M}(\tilde{e}\leq\tilde{v})\circ \dot{I}_e \circ \Rsheaf{E}(\tilde{e}\leq\tilde{v})\right) [\tilde{\gamma}_{\tilde{v}}] = \left(\Rcosheaf{M}(\tilde{e}\leq\tilde{v})\circ \dot{I}_e\right) [\tilde{\gamma}_{\tilde{e}}] = 0.
\]
Otherwise, if $\tilde{\sigma}$ is an edge $\tilde{e}$ and $\tilde{e}\leq\tilde{v}$ for $\tilde{v}$ more negative in $\R$ than $\tilde{e}$, then $[\tilde{\gamma}_{\tilde{v}}]=0$ and 
\[
\left(\Rcosheaf{M}(\tilde{e}\leq\tilde{v})\circ \dot{I}_e\right) [\tilde{\gamma}_{\tilde{e}}] = \left(\Rcosheaf{M}(\tilde{e}\leq\tilde{v})\circ \dot{I}_e \circ \Rsheaf{E}(\tilde{e}\leq\tilde{v})\right) [\tilde{\gamma}_{\tilde{v}}] = \dot{I}_v [\tilde{\gamma}_{\tilde{v}}] = 0.
\]
Since $\Mcosheaf$ is a monocosheaf, $\Rcosheaf{M}(\tilde{e}\leq\tilde{v})$ is injective.
So $\dot{I}_e [\tilde{\gamma}_{\tilde{e}}]=0$.

In either case, $\dot{I}_{\tilde{\sigma}} [\tilde{\gamma}_{\tilde{\sigma}}]=0$ which means $\tilde{\gamma}$ has trivial representation in $\Rcosheaf{M}(\tilde{\sigma})$.
Repeating iteratively for all adjacent cells with with non-trivial representation in $\Esheaf$ until reaching $\tilde{\tau}$ then shows that $\tilde{\gamma}$ is trivial in all components of the quotient-monocosheaf.
In particular, $\tilde{\gamma}$ has trivial representation in $\local(K)$ and hence $\local(G)$.
\end{proof}

For our main result, we wish to show there is an embedding of $\Hg_{\bullet+1}(G)/\Ig_{\bullet+1}$ into the classes of $\local(G)_\bullet$, however these objects are in different categories so we need to make this concept more precise.
Explicitly, consider the functor $\iota:\vect\to\colocalsystemcat(\Sp^1)$ which assigns to each vector space $V$ the constant colocal system $\cosheaf{\iota(V)}(\sigma)= V$ for every cell $\sigma$ and $\cosheaf{\iota(V)}(\sigma\leq \tau)=id$ for each face relation $\sigma\leq\tau$, and where each morphism $\varphi:V\to W$ is assigned $\cosheaf{\iota(\varphi)}:\cosheaf{\iota(V)}\to \cosheaf{\iota(W)}$ a morphism of colocal systems defined pointwise by $\varphi$.
Clearly $\iota$ is an embedding.
Thus, by an embedding of $\Hg_{\bullet+1}(G)/\Ig_{\bullet+1}$ into $\local(G)_\bullet$ we mean a monomorphism from a constant colocal system whose values are all $\Hg_{\bullet+1}(G)/\Ig_{\bullet+1}$ into $\local(G)_\bullet$.

\begin{theorem} \label{thm:1d-toroidal}
Diagram~\ref{diag:correspondence} restricts to the canonical embedding $\iota(\Hg_{\bullet+1}(G)/\Ig_{\bullet+1})\inj\local(G)_\bullet$.
\end{theorem}

\begin{remark}
A subtle point of Diagram~\ref{diag:correspondence} and Lemma~\ref{lem:1d-nontoroidal} is that non-toroidal cycles may already correspond to a trivial cycle within the (epi)sheaf.
For instance, the cross-section of an infinite cylinder will be a relative boundary\footnote{See Section~\ref{sec:needing-sheaves} for another example of this.}.
In contrast, an immediate Corollary of Theorem~\ref{thm:1d-toroidal} is that this cannot be true for toroidal cycles.
\end{remark}

\begin{proof}
Let $\gamma$ be a cycle of $G$.
As with Lemma~\ref{lem:1d-nontoroidal} we work primarily with $\Rbisheaf{F}$ and leverage the representation of $\gamma$ in $\Esheaf(\sigma)$ by $[\gamma_\sigma]$.
Let $\tilde{\gamma}$ be the \emph{chain} of $K$ which maps to $\gamma$ with support on the half open interval $[0,1)$ and write $\D\tilde{\gamma}=\tran(\tilde{\zeta})-\tilde{\zeta}$ for $\tilde{\zeta}$ supported on $\{0\}$.
By Lemma~\ref{lem:1d-nontoroidal}, Diagram~\ref{diag:correspondence} will restrict to well-defined maps $\Hg_{\bullet+1}(G)/\Ig_{\bullet+1}\to\local(G)_\bullet(\sigma)$.
By the correspondence of Diagram~\ref{diag:correspondence}, we may choose consistent bases for the image of $\Hg_\bullet(G)$ under each cell component of $\local(G)_\bullet$
Changing bases if necessary so that the PLS restricts to a constant colocal system on these classes.
So to show that $\iota(\Hg_{\bullet+1}(G)/\Ig_{\bullet+1})\inj\local(G)_\bullet$ we will prove a contrapositive result: that $I_\sigma[\gamma_\sigma]= 0$ implies $\gamma$ is non-toroidal.

First, suppose $F_\sigma[\gamma_{\sigma}]=0$.
We will show that $\gamma$ must be non-toroidal by relating $F_\sigma[\gamma_{\sigma}]$ to $\tilde{\zeta}$ in order to construct a cycle in $K$ which maps to $\gamma$.
Let $\varphi$ the the 1-cocycle of $G$ defined by the pullback of the orientation of $\Sp^1$.
In particular, we take $\varphi$ to be supported on some interval $[a,b]$ contained in $\St(\sigma)$, so that $F_\sigma[\gamma_{\sigma}]=[\gamma\frown \varphi]$.
Without loss of generality, assume that $0\leq a<b<1$, recellulating if necessary.
Let also $\psi$ be a $0$-cochain on $K$ defined as the pullback of the $0$-cochain on $\R$ which evaluates to $n$ on $[n+b,n+a+1]$ for every $n\in\Z$ and $0$ elsewhere.
Then $q^*\varphi = \delta\psi$ where $\delta$ is the coboundary map.

Set $\xi=\gamma\frown\varphi$ and let $\tilde{\xi}$ be its unique lift in $\St(\tilde{\sigma})\subsetneq K$.
Since $\gamma=q_*(\tilde{\gamma})$, the identity $q_*(\tilde{\gamma})\frown \varphi = q_*(\tilde{\gamma}\frown q^*\varphi)$ tells us $\xi = q_*(\tilde{\gamma}\frown q^*\varphi)$.
Further, $\tilde{\gamma}$ is supported on $[0,1)$, so $\tilde{\gamma}\frown q^*\varphi$ is supported on $\St(\tilde{\sigma})$ and hence $\tilde{\xi}=\tilde{\gamma}\frown q^*\varphi$ by local isomorphism.
Applying the boundary identity of the cap product, we then have
\[
\D\left(\tilde{\gamma}\frown \psi\right) = \D\tilde{\gamma}\frown \psi - \tilde{\gamma}\frown \delta\psi = \D\tilde{\gamma}\frown \psi - \tilde{\gamma}\frown q^*\varphi.
\]
Rearranging this equation and observing $\D\tilde{\gamma}\frown\psi = \tran (\tilde{\zeta}) = \tilde{\zeta} + \D\tilde{\gamma}$ by construction of $\psi$ then tells us
\[
\tilde{\xi} = \tilde{\gamma}\frown q^*\varphi = \D\tilde{\gamma}\frown \psi- \D\left(\tilde{\gamma}\frown \psi\right) = \D(\tilde{\gamma} -(\tilde{\gamma}\frown\psi))+\tilde{\zeta}.
\]
Hence $[\tilde{\xi}]=[\tilde{\zeta}]$.

And $[\xi]=0$ in $\Fcosheaf(\sigma)$ by assumption, so by local isomorphism $[\tilde{\xi}]=\dot{F}_{\tilde{\sigma}}[\tilde{\gamma}_{\tilde{\sigma}}]=0$ in $\Rcosheaf{F}(\tilde{\sigma})$.
Applying the homology functor to the embedding $\St(\tilde{\sigma})\inj \R$ then implies $[\tilde{\zeta}]=[\tilde{\xi}]=0$ in $K$, so there is some $\tilde{\alpha}$ such that $\tilde{\zeta}=\D(\tilde{\alpha})$.
But then $\D(\tilde{\gamma}-\tran(\tilde{\alpha})+\tilde{\alpha})=0$ and $q_*(\tilde{\gamma}-\tran(\tilde{\alpha})+\tilde{\alpha})=\gamma$, so $\gamma$ is non-toroidal.

To finish the proof, let $\Fcosheaf_j$ denote the cosheaf at the $j^\mathrm{th}$ iteration of Algorithm~\ref{alg:mono} where $\Fcosheaf_0:=\Fcosheaf$ is the initial cosheaf.
Furthermore, let $\Kcosheaf_j$ be the $j^\mathrm{th}$ kernel so that $\Kcosheaf_0:=\cosheaf{0}$ and $\Fcosheaf_j=\Fcosheaf/\Kcosheaf_j$.
We similarly define $\Rcosheaf{F}_j$ and $\Rcosheaf{K}_j$ and claim that if $F_\sigma[\gamma_\sigma]\in\Kcosheaf_j$ for some $j\geq0$ and some cell $\sigma$, then $\gamma$ must be non-toroidal.
Indeed, since the Monofication Algorithm terminates by Theorem~\ref{thm:algorithms}, $F_\sigma[\gamma_\sigma]\in\Kcosheaf_j(\sigma)$ for some $j\geq0$ if and only if $F_\sigma[\gamma_\sigma]$ is trivial in $\Mcosheaf(\sigma)$ and in turn $\local(G)$, and the contrapositive statement to the claim finishes the proof of the Proposition.

Since $F_\sigma[\gamma_{\sigma}]=0$ implies $\gamma$ is non-toroidal, we have the desired result for the initialisation step.
Otherwise, once again consider $\dot{F}_{\tilde{\sigma}}[\tilde{\gamma}_{\tilde{\sigma}}]$ in the cosheaf under $\R$ with $\tilde{\gamma}$ and $\tilde{\sigma}$ chosen as before.
Observe that for any edges $\tilde{e},\tilde{f}$ and vertices $\tilde{u},\tilde{v},\tilde{w}$ where $\tilde{e}
\leq \tilde{v}$, $\tilde{e}\leq \tilde{w}$, $\tilde{f}\leq \tilde{u}$ and $\tilde{f}\leq \tilde{v}$ (where possibly $u=w$ but $\tilde{u}\neq \tilde{v}$ and $\tilde{v}\neq \tilde{w}$) we may define $\Rcosheaf{K}_j$ inductively so that
\begin{align*}
\Rcosheaf{K}_{j+1}(\tilde{e}) & = \Rcosheaf{K}_{j}(\tilde{e})+\Rcosheaf{F}(\tilde{e}\leq \tilde{v})^{-1}\left[\Rcosheaf{K}_{j}(\tilde{v})\right]+\Rcosheaf{F}(\tilde{e}\leq \tilde{w})^{-1}\left[\Rcosheaf{K}_{j}(\tilde{w})\right]\\
\Rcosheaf{K}_{j+1}(\tilde{v}) & = \Rcosheaf{K}_{j}(\tilde{v})+\Rcosheaf{F}(\tilde{e}\leq \tilde{v})\left[\Rcosheaf{K}_{j+1}(\tilde{e})\right]+\Rcosheaf{F}(\tilde{f}\leq \tilde{v})\left[\Rcosheaf{K}_{j+1}(\tilde{f})\right]
\end{align*}
where $A+B$ denotes the linear span of $A$ and $B$.
Thus in all cases, for $j\geq 1$ we can write $\Rcosheaf{K}_{j}(\tilde{\sigma})$ as a linear span of spaces which are a sequence of images and preimages of a kernel.
Moreover, recall that all maps of the cosheaf are induced by inclusion.
So if $\dot{F}_{\tilde{\sigma}}[\tilde{\gamma}_{\tilde{\sigma}}]\in\Rcosheaf{K}_{j}(\tilde{\sigma})$ for $j\geq 1$ then we can always write $\dot{F}_{\tilde{\sigma}}[\tilde{\gamma}_{\tilde{\sigma}}]$ as represented the sum of cycles of the form $\tilde{\alpha}_{\tilde{\tau}}$ such that $\tilde{\alpha}_{\tilde{\tau}}$ is a cycle in $\St(\tilde{\sigma})$ but homologous to a boundary in $\St(\tilde{\tau})$.
But this means $\dot{F}_{\tilde{\sigma}}[\tilde{\gamma}_{\tilde{\sigma}}]=[\tilde{\xi}]$ is a boundary in $K$ and $\gamma$ must be non-toroidal by an identical argument to the $j=0$ case.
\end{proof}

\begin{remark}
In Theorem~\ref{thm:1d-toroidal} we need to show that $[\tilde{\gamma}\frown \varphi]\neq 0$ in $\Mcosheaf(\sigma)$ instead of in $\Hg_\bullet(G)$, since it is possible for $\gamma\frown\varphi$ to be a boundary in $G$ but for $\tilde{\gamma}\frown q^*\varphi$ to not be a boundary in $K$.
See Section~\ref{sec:needing-sheaves} for an example of this.
\end{remark}

\subsection{The \texorpdfstring{$d$}{d}-periodic Case} \label{sec:full-case}

For the following, we let $\{\tran_1,\dots,\tran_d\}$ be a basis for $T$ and denote $\hat{T}_i=\langle\tran_1,\dots,\tran_{i-1},\tran_{i+1},\dots,\tran_d\rangle$, where $\langle\bullet\rangle$ denotes the group generated by the specified elements.
Additionally, we define $\hat{K}_i=K/\hat{T}_i$.
Observe that if $K$ is $d$-periodic for $d\geq 1$ then $\hat{K}_i$ is 1-periodic and $G\cong\hat{K}_i/\langle\tran_i\rangle$.

\begin{proposition}\label{prop:toroidal-compression}
Suppose $K$ is a $d$-periodic complex for $d\geq 1$.
If $\gamma$ is a toroidal cycle in $G$ with respect to the quotient $K\surj G$ then there exists $i\in\{1,\dots,d\}$ such that $\gamma$ is a toroidal cycle with respect to the quotient $\hat{K}_{i}\surj G$.
\end{proposition}
\begin{proof}
We can decompose the quotient map into $q = q_1 \circ q_2 \circ \ldots \circ q_d$ where any permutation of the application of quotient map commutes. 
To prove the lemma is equivalent to showing that if $\gamma$ is a cycle in $\coker(q)$ then there is a permutation such that $\gamma\in\coker(q_d)$, i.e. the cokernel of the final quotient map.
We first show that this holds for any lift $\gamma$ which is supported on $[0,1]^d$. The relations induced on this representative are given by $q_i \cong (0\sim 1)_i$ or the identification in the $i$-th coordinate.
As the support is restricted to $[0,1]^d$ all chain groups are finite, the boundary $\partial \gamma$ is in the span of the $q_i$'s. Further, since the $q_i$'s are linearly independent, there is a unique linear combination 
$$\partial \gamma =\sum\limits_i \lambda_i q_i $$
Any permutation which places a $q_i$ where $\lambda_i\neq 0$ in at the end, i.e. $\tilde{q}_d$, will have $\gamma$ will be in the cokernel of the induced map $q^*_{d}$. 

Next  consider the case where there is a different lift $\gamma'$ but it is still supported on $[0,1]^d$. Assume that we subtract all non-toroidal cycles from the lift. In this case, as they are homologous cycles in the quotient, the chains must be homologous in the lifts (as the quotient maps are injective on toroidal cycles).  Let $c$ be the lift of the bounding chain. Applying the quotients in the same order as above, there must be an $i$ such that $\gamma'$ is a cycle and $\gamma$ is not. However, this implies that $\partial \cdot \partial c = \partial (\gamma+\gamma') = \partial \gamma \neq 0$ which is a contradiction. We observe that if any non-toroidal cycle is added, it either in the kernel of $\tilde{q}_d$, in which case the statement still holds, or it is in the kernel of $\tilde{q}_i$ for $i<d$ in which case we are in the case described above. 

Finally we show that the same holds for any choice of lifts, not only those supported on $[0,1]^d$. Applying the quotients in the same order as above, after applying the first $(d-1)$ quotients,   $\gamma$ is supported on $[0,1]^{d-1}\times \mathbb{R}$. Hence $\gamma$ can be written as the sum of disjoint chains. Note that if it is not disjoint, then it is supported on $[0,1]^d$. As the boundary of the sum is not zero, there must be an component of the chain which has a non-zero boundary and so it cannot be a toroidal cycle at $K_{d-1}$, completing the proof.
\end{proof}

For the remainder of the paper we use the notation that $\local_L(G)$ is the persistent local system obtained from the isobisheafification of the bisheaf attributed to $\pi_L:G\to\torus^\ell$.
In the case where $L=\{i\}$ we further simplify this to $\local_i(G)$.
We may either look at each $\local_i(G)$ separately for each $i$ or consider them together as the object $(\local_1(G),\dots,\local_d(G))$ in the product category $\Pi_{i=1}^d\colocalsystemcat(\Sp^1)$.

\begin{theorem} \label{thm:toroidal}
If $K$ is a $d$-periodic cellular complex then there is a canonical embedding $\iota(\Hg_{\bullet+1}(G)/\Ig_{\bullet+1})\inj(\local_1(G)_\bullet,\dots,\local_d(G)_\bullet)$.
In other words, for every non-zero $[\gamma]\in\Hg_{\bullet}(G)/\Ig_{\bullet}$ there exists $i\in\{1,2,\dots,d\}$ and a corresponding canonical embedding of $\iota(\F)\inj\local_i(G)$.
\end{theorem}
\begin{proof}
This follows immediately from Theorem~\ref{thm:1d-toroidal} and Proposition~\ref{prop:toroidal-compression}.
\end{proof}

\begin{remark}
We note that for sufficiently nice translation actions on $\R^d$ -- such that the translation group is isomorphic to $\Z^k$ and the quotient space is a manifold, e.g. a free and properly discontinuous action, then the above results apply directly, as the quotient space must be homologically equivalent to a torus\footnote{A stronger equivalence holds, but is not relevant for our results}.  
\end{remark}

\section{Algorithms} \label{sec:algorithms}
In this section, we describe the algorithms which we use to compute both local systems.

\subsection{Constructing Bisheaves}

Suppose we are given a complex $G$ in an $n$-dimensional box with periodic boundary conditions in the first $d$ coordinates.
That is, $G$ is a quotient space of a $d$-periodic  complex $K$ in $\R^n$ and $G$ is embedded in $[0,1]^n/\sim \;\cong \torus^d\times[0,1]^{n-d}$ where $\sim$ is the equivalence relation $0\sim 1$ for the first $d$ coordinates.
See for example Figure~\ref{fig:periodic-boundary} where we glue together the left and right sides of the highlighted box.

\begin{figure}[ht]
\centering
    \includegraphics[trim=0 2cm 0 0,clip,page=3,width=0.7\textwidth]{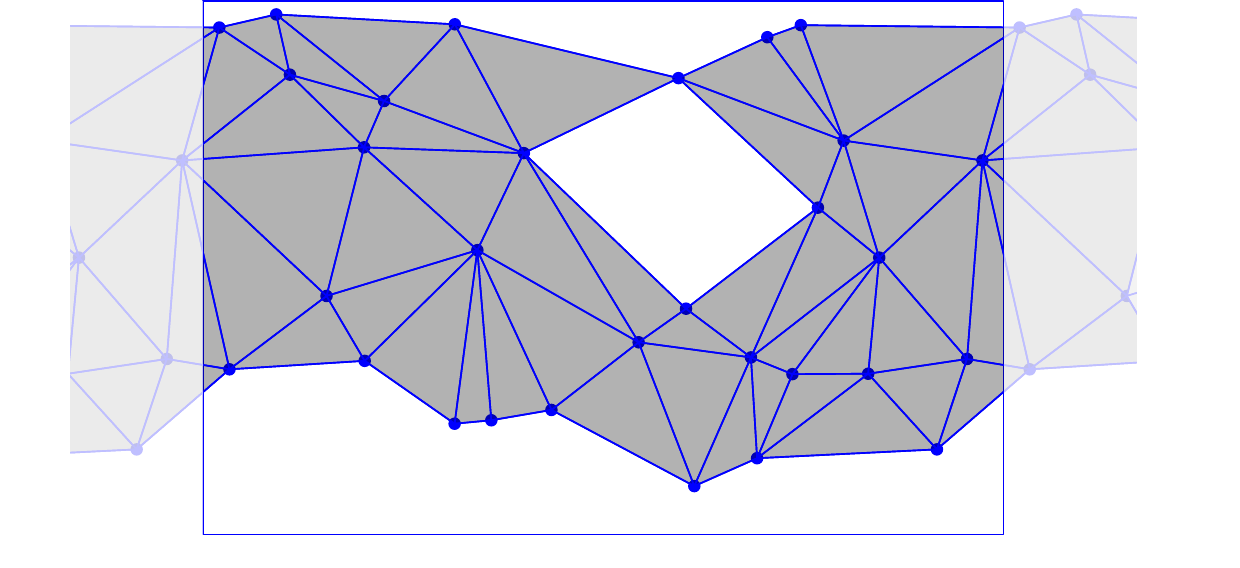}
    \caption{An periodic simplicial complex over $\mathbb{S}^1$.}
    \label{fig:periodic-boundary}
\end{figure}
We restrict ourselves to the case of embedded cellular complexes such as subsets of triangulations or cubical complexes and as a reminder we compute the bisheaves corresponding to the $d$-coordinate projections $\pi_d:G\to\torus^d$ (rather than more general maps).
To compute the bisheaves we need to be able to do the following operations:
\begin{itemize} 
\item compute homology (with representative cycles)
\item  compute cohomology (with representative cocycles)
\item compute relative homology (with representative cycles)
\item compute cap products
\end{itemize}
These computations will be done over preimages $\pi^{-1}_d$ of open sets of $\torus^d$. Therefore, we first must construct a cellulation of $\torus^d$ such that all open sets are captured by open stars of the cellulation. This condition is equivalent to saying that sheaves/cosheaves/bisheaves are locally constant over the open star of every cell in the cellulation. This allows us to then use cellular homology computations for the operations above.

In the case where $G$ is a cubical complex, this is trivial as any the projection will correspond to a $d$-face of the ($n$-dimensional) cubical complex. Hence any preimage of open (or respectively closed) cells in $\torus^d$ is naturally an open (closed) cubical subcomplex of $G$.

If $G$ is a simplicial complex, a constrained triangulation of $\torus^d$ must be computed, where each simplex maps to a sum of simplices in the triangulations of  $\torus^d$. Note that while the requirement of $G$ being embedded is for convenience, the requirement that the cellulation of $\torus^d$ be embedded is necessary. 
As a simple example in the case of a 1-periodic complex as in Figure~\ref{fig:divide}, a cellulation of $\Sp^1$ is given by projecting all the vertices to the corresponding coordinate. As any change in the complex must occur at a vertex, this is sufficient for our purposes. One dimension higher (from $\torus^3 \rightarrow \torus^2$), we show an example in Figure~\ref{fig:2d}. With sufficient subdivision, we can obtain a triangulation with the required properties --- subdividing whenever two simplices which do not intersect in $G$ intersect in the image of $\pi_d(G)$. To the best of our knowledge, there are no non-obvious ways to achieve this, with a better run time for example, and we leave it for future work to further investigate.

\begin{figure}[ht]
\includegraphics[page=5,width=0.5\textwidth]{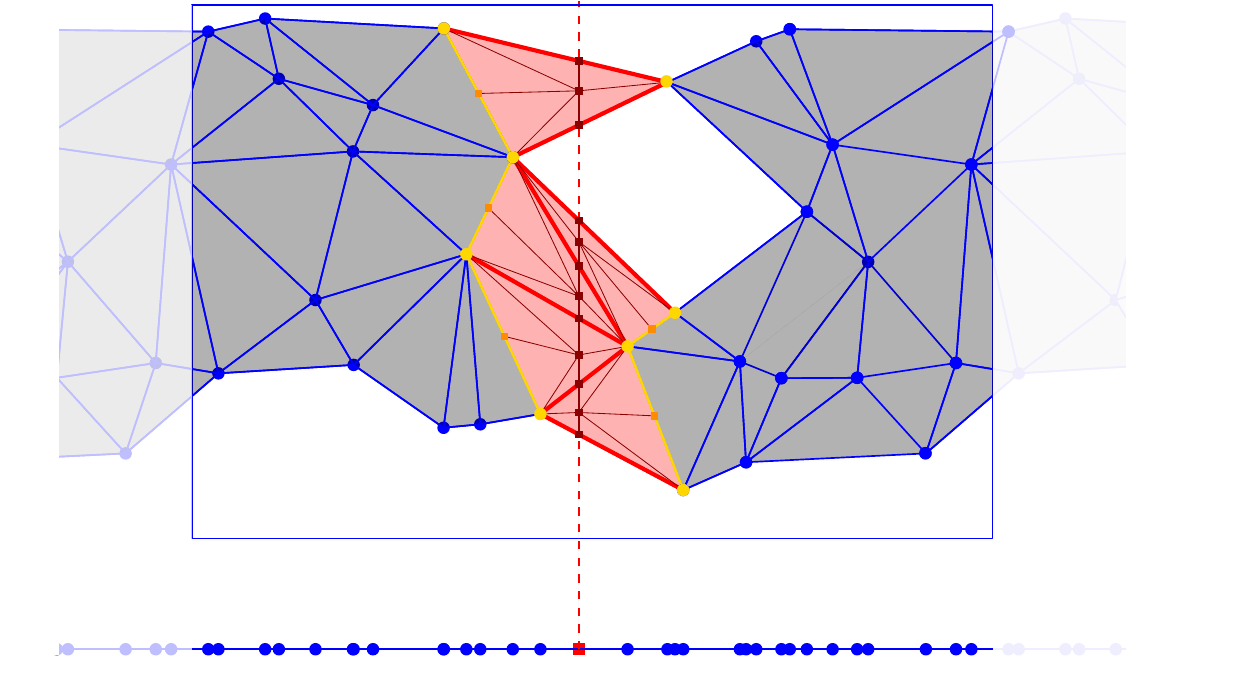}
\includegraphics[page=6,width=0.5\textwidth]{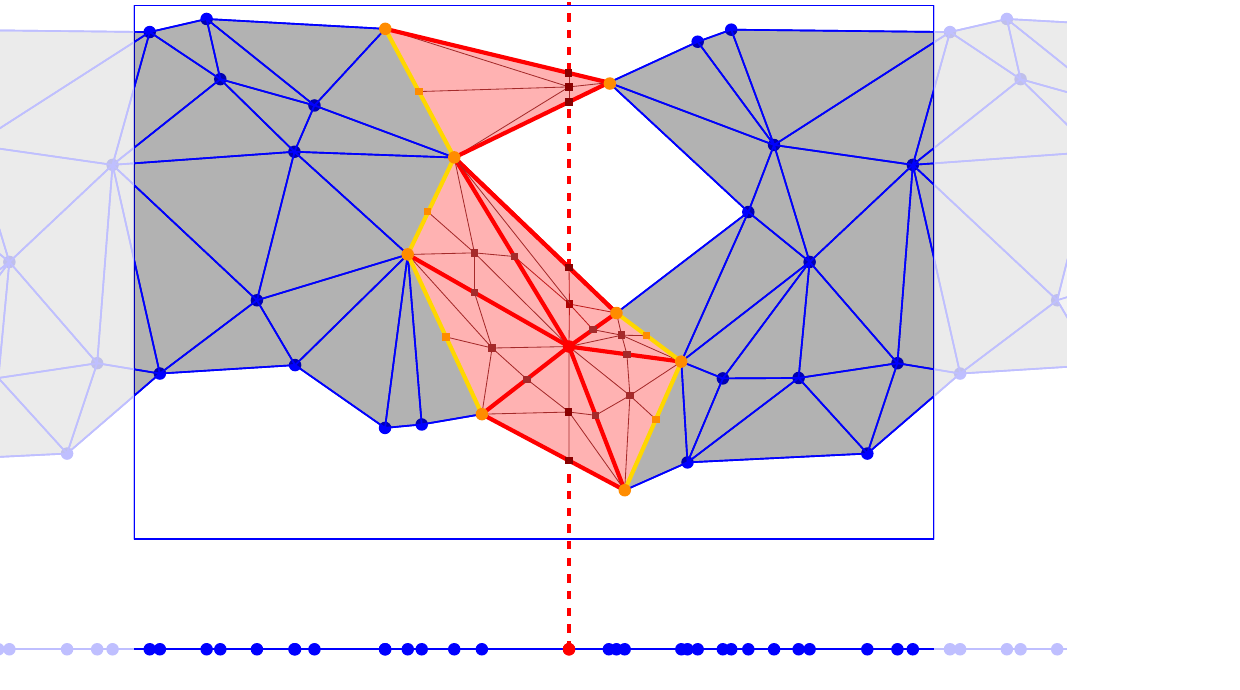}
\caption{Successive subdivisions of the triangulation in Figure~\ref{fig:periodic-boundary} so that the preimage of the highlighted vertices in $\Sp^1$ are subcomplexes.}
    \label{fig:divide}
\end{figure}

\begin{figure}[ht]
\centering \includegraphics[width=0.4\textwidth]{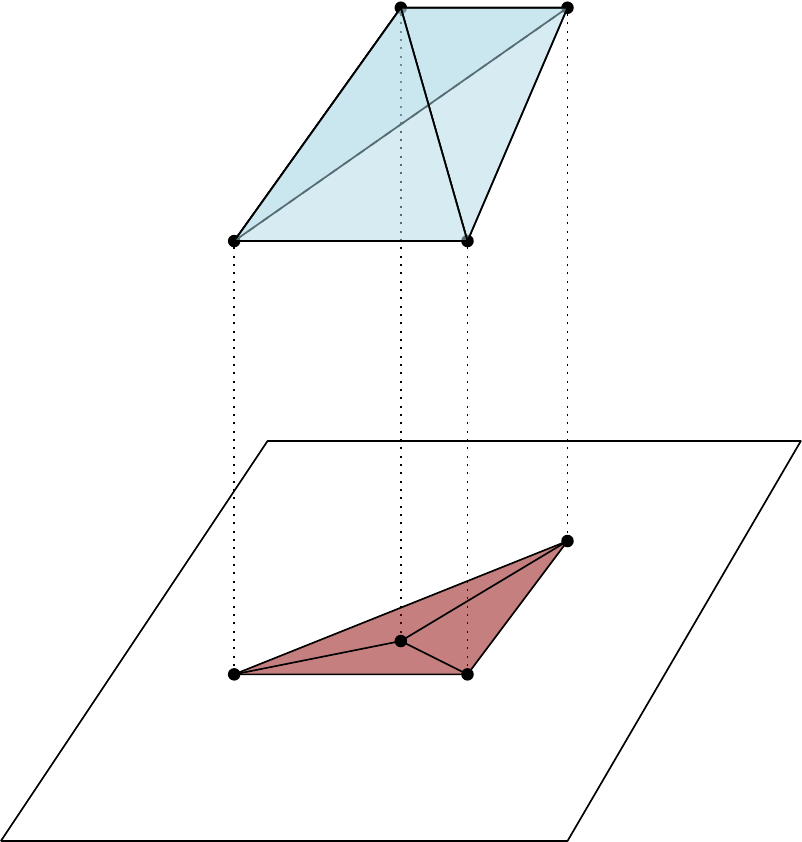}
\caption{This projection of a tetrahedron to the $\torus^2$, requires a triangulation of 4 triangles so that open star of each simplex has a constant preimage.}
    \label{fig:2d}
\end{figure}

Once the cellulation of $\torus^d$, which we denote $\mathcal{T}$, has been fixed, we must compute the simplices which lie in the preimage of each cell of $\mathcal{T}$.  In the case of cubical complexes, this is again trivial. In the case of  simplicial complexes, we must project each simplex $\sigma \in G$ to $\torus^d$ and find which simplices $\tau\in \mathcal{T}$ intersect $\pi_d(\sigma)$. To the best of our knowledge, no algorithm beyond direct search of the intersections is known. 

With this information in hand, we can compute the bisheaves. For the sheaf, we must compute the relative homology (and later also the cohomology) of each open cell. Here we use the fact that if we were to subdivde $\mathcal{T}$, each open star in the subdivision would correspond to an open cell in $\mathcal{T}$ -- so we omit the subdivision. For each cell $\tau\in \mathcal{T}$, we compute
\[
\Hg_\bullet(\pi_d^{-1}(\Cl(\St(\tau))),\pi_d^{-1}(\Cl(\St(\tau)))-\pi^{-1}_d(\St(\tau)))
\]
and
\[
\Hg^\bullet(\pi_d^{-1}(\Cl(\St(\tau))),\pi_d^{-1}(\Cl(\St(\tau)))-\pi^{-1}_d(\St(\tau)))
\]
where $\Cl(\cdot)$ denotes the closure. Here by $\pi^{-1}_d(\tau)$ we mean all open cells which intersect $\St(\tau)$ in the projection. It is straightforward to check that this deform retracts to the preimage. This is well-defined and can be computed cellularly as we assume finiteness. The induced maps between preimages of cells are indued by chain maps which are restrictions. This includes all the required data for the sheaves. 

For the cosheaves, we must  compute the homology of $\Hg_\bullet(\pi^{-1}_d)$.  Here we cannot simply take the preimages of open cells but rather we must take a the (barycentric) subdivision of $\pi^{-1}_d$ which is a subset of $\Sub(G)$, the subdivision of $G$. Again, it is straightforward to show that $\pi^{-1}_d$ deform retracts to the maximum subcomplex of $\Sub(\pi^{-1}_d)$, i.e. the large subset of simplices which form a complex. The cosheaf is then given as the homology of this maximal subcomplex. One can directly verify that the chain maps induced by inclusion maps of cells in $\mathcal{T}$, are well-defined. For $\tau_1 \subset \tau_2 \in \mathcal{T}$, observe that $\St(\tau_2) \subset \St(\tau_1)$.
It is straightforward to check that the maps $\pi^{-1}_d(\St(\tau_2)) \hookrightarrow \pi^{-1}_d(\St(\tau_1))$ are cellular between maximal subcomplexes of the subdivisions. Hence the cosheaf information can be computed cellularly.

This is illustrated in Figure \ref{fig:divide}. On the left hand side, we see the preimage of an edge of the cellulation of $\mathbb{S}^1$ shown in red along with its subdivision. The closure, used in the sheaf computation, is shown in yellow and orange. The maximum subcomplex, is given by the dark red points (and the edges connecting them). Thus, the homology of the preimage of this edge is of rank 2 (i.e. consists of 2 components). On the right hand side, we see the case for a preimage of a vertex. The maximum subcomplex is again shown in red, but is bigger, but the homology is again rank 2.  

Finally, as the cap product is computed pointwise, we have all the information required to compute it. 
As all (co)homological computation is computed on the chain space of the subdivision, the orderings will be consistent with the initial complex which ensures each cocycle we use for the cap products are consistent.
Taking the cap product of the relative homology group with the relative cohomology group yields a map to homology which is isomorphic  to the homology group of the maximal subcomplex described above. 
Note that from a practical perspective, it is simpler from an implementation standpoint to perform the calculations in the sheaf on the subdivision thereby making all the chain groups directly relatable (i.e. without any deformation retracts needed in the maps).

\paragraph{Example:} As a simple example, we describe how this for the case of a subcomplex of $\mathbb{R}^2$ under a $\mathbb{Z}$-action, i.e. where the projection map is to $\mathbb{S}^1$. As described above, the triangulation of $\mathbb{S}^1$ in this case is the projection of all vertices in the complex in $\mathbb{R}^2$, and connect them in order (of the required coordinate) with edges into a simple cycle. 

We must first find the required preimages of the projection:
\begin{enumerate}
\item For the preimage of the open star of a vertex, take the simplices whose projection contains the vertex. This is simply the test of an intersection between an interval (the projection of the simplex to $\Sp^1$) and a point (the vertex coordinate).
\item For the preimage of an edge, it is similar but the test is for the intersection of two intervals. 
\item We observe that the open star of an edge is just the edge, while the open star of a vertex is the union of preimage of the vertex and its two adjacent edges.
\end{enumerate}

We refer to these subsets of simplices for each vertex and edge in $\Sp^1$ by $X(\sigma)\subset G$ where $G$ is periodic complex, and $\sigma \in \{v,e\}$ for vertices ($v$) and edges ($e$) in the triangulation of $\Sp^1$. Note that $X(\sigma)$ is  not a complex but $\Cl(X(\sigma))$ is. To compute the bisheaf, 
\begin{enumerate}
\item Take a subdivision $\Sub\;\Cl(X(\sigma))$
\item To compute the relative homology/cohomology at $\sigma$,  construct the relative chain complex
\[C_\bullet (\Sub\;\Cl(X(\sigma)), \Sub\;\Cl(X(\sigma)) - \Sub\; X(\sigma)) \]
and take homology/cohomology.
\item To compute the cosheaf, construct
\[Y(\sigma) = \{ \tau \in \Sub\; X(\sigma) : \partial \tau \in \Sub\; X(\sigma)\}\]
where $\partial \tau$ denotes the faces of $\tau$. The value of the cosheaf at $\sigma$ is $\Hg_\bullet (Y(\sigma))$.
\end{enumerate}
For both the sheaf and cosheaf, the inclusion restriction maps on the chain groups give the required morphisms and the standard simplicial cap product computation may be applied.  

We make two further observations:
\begin{enumerate}
\item These can be computed independently for each vertex and edge in the triangulation of $\Sp^1$.
\item By working with subdivision, the chain groups in the sheaf and cosheaf are comparable, namely any chain in the sheaf can be restricted to a chain in the cosheaf. 
\item The main complexity lies in the dimensionality of the translation action rather than the ambient space -- so in the case above, where the translation action is 1-dimensional, the consontruction is quite efficient. 
\end{enumerate}

We conclude this section with a summary of the main algorithmic difficulties when we consider higher dimensional translation actions:
\begin{itemize}
\item computing a  cellulation of the projection such that the preimage of each cell islocally constant;
\item finding the preimage of each cell.
\end{itemize}
We note that for $\torus^2$, computing a cellulation can be done through the computation of an arrangement, but we defer this question to future work.





\subsection{Isobisheafification} \label{sec:isobisheaf}

In this section we provide an explicit construction of a canonical isobisheaf and persistent local system associated to a bisheaf.
We achieve this through \textit{isobisheafification} algorithms, consisting of two parts: epification and monofication.

Recall that the epification of a sheaf $\Fsheaf$ is the largest sub-episheaf $\Esheaf\inj\Fsheaf$.
In Algorithm~\ref{alg:epi} we introduce our algorithm for epification.
The first while loop in Algorithm~\ref{alg:epi} establishes an ordering of cells.
If $X$ is a simplicial, cubical or polyhedral complex then we need not consider this loop, as we may simply set $V_\ell$ to be the $\ell$-dimensional cells.
In the second loop we first work backwards from the maximal cells in a process we consider as \textit{projecting onto a family of epimorphisms}.
However this leaves some morphisms not well-defined, i.e. potentially $\im\left(\Esheaf_j(\sigma\leq\tau)\right)\supsetneq\Esheaf_j'(\sigma)$.
Therefore, in the second half of the second while loop, we work forwards from the minimal cells to ensure that all morphisms are well-defined.
We consider this process as \textit{projecting onto a functor}.

\begin{algorithm}
\caption{Epification of a sheaf}
\label{alg:epi}
	\begin{algorithmic}
	\REQUIRE A sheaf $\Fsheaf$ valued in $\Ab$ over (finite) cell complex $X$
	\ENSURE The maximal sub-episheaf $\Esheaf\inj\Fsheaf$
	
	\STATE $\mathbb{V}\gets K$
	\STATE $\ell\gets 0$
	\WHILE{$\mathbb{V}\neq \emptyset$}
		\STATE $\ell\gets\ell+1$
		\STATE $V_\ell\gets\emptyset$
		\FOR{$\sigma\in \mathbb{V}$}
			\IF{$\sigma$ maximal with respect to face relations on $\mathbb{V}\cup\{\sigma\}$}
				\STATE $V_\ell\gets V_\ell\cup\{\sigma\}$
			\ENDIF
		\ENDFOR
		\STATE $\mathbb{V}\gets \mathbb{V}\setminus V_\ell$
	\ENDWHILE
	
	\STATE $\Fsheaf_0\gets\Fsheaf$
        \STATE $j\gets 0$
	\WHILE{$\Fsheaf_j$ not episheaf}
		\FOR{$i$ descending from $\ell$ \TO $1$}
			\FOR{$\sigma\in V_i$}
				\STATE $\Fsheaf'_j(\sigma)\gets \mathrm{Pullback}_{\tau:\sigma\leq\tau}\left[\im\left(\Fsheaf_j(\sigma\leq\tau)\right)\inj \Fsheaf_j(\sigma)\right]$
                    \STATE $\iota^\sigma_j:\Fsheaf'_j(\sigma)\inj\Fsheaf_j(\sigma)\gets$ induced by the pullback above
			\ENDFOR
		\ENDFOR
		\FOR{$i$ increasing from $1$ \TO $\ell$}
			\FOR{$\sigma\in V_i$}
                    \FOR{$\tau:\tau <\sigma$}
                        \STATE $\mathcal{E}_j^{\sigma\tau} \gets \mathrm{Pullback}\left(\left[\Fsheaf_j(\tau\leq\sigma):\Fsheaf_j(\sigma)\to\Fsheaf_j(\tau)\right],\left[\Fsheaf_{j+1}(\tau)\inj\Fsheaf'_j(\tau)\xhookrightarrow{\iota^\tau_j}\Fsheaf_j(\tau)\right]\right)$
                    \ENDFOR
                    \STATE $\mathcal{E}^{\sigma\sigma}_j\gets \Fsheaf'_j(\sigma)$
                    \STATE $\Fsheaf_{j+1}(\sigma)\gets\mathrm{Pullback}\left(\left[\iota^\sigma_j:\Fsheaf'_j(\sigma)\inj\Fsheaf_j(\sigma)\right],\mathrm{Pullback}_{\tau:\tau\leq\sigma}\left[\mathcal{E}_j^{\sigma\tau}\inj\Fsheaf_j(\sigma)\right]\right)$
				\STATE $\Fsheaf_{j+1}(\tau\leq\sigma)\gets$ induced by the pullback above
			\ENDFOR
		\ENDFOR
            \STATE $j\gets j+1$
	\ENDWHILE
	\STATE $\Esheaf\gets\Fsheaf_j$
	\RETURN $\Esheaf$
\end{algorithmic}
\end{algorithm}

\begin{remark}
Algorithm~\ref{alg:epi} naturally generalises to sheaves over any category with images and pullbacks, such as \textbf{Set}.
\end{remark}
In Appendix~\ref{sec:diagrams} we present the pullback diagrams described in Algorithm~\ref{alg:epi}.
If we consider the specific example $\Ab=\vect$, we may replace all pullbacks with intersections, which allows the following identifications.
\begin{itemize}
    \item $\Fsheaf_j'(\sigma)=\bigcap_{\tau:\sigma\leq\tau}\im\left(\Fsheaf_{j-1}(\sigma\leq\tau)\right)$ is the intersection of the images of maps entering $\sigma$.
    \item $\Fsheaf_j(\sigma)=\bigcap_{\tau:\tau\leq\sigma}\Fsheaf(\tau\leq\sigma)^{-1}\left(\Fsheaf_j'(\tau)\right)$ is the intersection of the preimage of $\Fsheaf_j'$ for all maps leaving $\sigma$.
    \item $\Fsheaf_j(\tau\leq\sigma)=\Fsheaf(\tau\leq\sigma)_{\Fsheaf_j(\sigma)}$ is the natural restriction map.
\end{itemize}

Recall now that the monofication of a cosheaf $\Fcosheaf$ is the smallest quotient-monocosheaf $\Fcosheaf\surj\Mcosheaf$.
In Algorithm~\ref{alg:mono} we now introduce our algorithm for monofication.
As with Algorithm~\ref{alg:epi}, the first while loop of Algorithm~\ref{alg:mono} may be omitted if $X$ is a simplicial, cubical or polyhedral complex.
As a dual concept to Algorithm~\ref{alg:epi}, however, in the first half of the second while loop we work backwards from the maximal cells to \textit{project onto a family of monomorphisms}, and in the second half we work forwards from the minimal cells to \textit{project onto a functor}.

\begin{algorithm}
\caption{Monofication of a cosheaf}
\label{alg:mono}
	\begin{algorithmic}
	\REQUIRE A cosheaf $\Fcosheaf$ valued in $\Ab$ under (finite) cell complex $X$
	\ENSURE The minimal quotient-monocosheaf $\Fcosheaf\surj\Mcosheaf$

	\STATE $\mathbb{V}\gets K$
	\STATE $\ell\gets 0$
	\WHILE{$\mathbb{V}\neq \emptyset$}
		\STATE $\ell\gets\ell+1$
		\STATE $V_\ell\gets\emptyset$
		\FOR{$\sigma\in \mathbb{V}$}
			\IF{$\sigma$ maximal with respect to face relations on $\mathbb{V}\cup\{\sigma\}$}
				\STATE $V_\ell\gets V_\ell\cup\{\sigma\}$
			\ENDIF
		\ENDFOR
		\STATE $\mathbb{V}\gets \mathbb{V}\setminus V_\ell$
	\ENDWHILE

	\STATE $\Fcosheaf_0\gets\Fcosheaf$
        \STATE $j\gets 0$
	\WHILE{$\Fcosheaf_j$ not monocosheaf}
		\FOR{$i$ descending from $\ell$ \TO $1$}
			\FOR{$\sigma\in V_i$}
				\STATE $\Fcosheaf'_j(\sigma)\gets \mathrm{Pushout}_{\tau:\sigma\leq\tau}\left[ \Fcosheaf_j(\sigma)\surj\mathrm{coim}\left(\Fcosheaf_j(\sigma\leq\tau)\right)\right]$
                    \STATE $\rho^\sigma_j:\Fcosheaf_j(\sigma)\surj\Fcosheaf'_j(\sigma)\gets$ induced from pushout above
			\ENDFOR
		\ENDFOR
		\FOR{$i$ increasing from $1$ \TO $\ell$}
			\FOR{$\sigma\in V_i$}
                    \FOR{$\tau:\tau<\sigma$}
                        \STATE $\mathcal{N}^{\sigma\tau}_j \gets \mathrm{Pushout}\left(\left[\Fcosheaf_j(\tau\leq\sigma):\Fcosheaf_j(\tau)\to\Fcosheaf(\sigma)\right],\left[\Fcosheaf_j(\tau)\xtwoheadrightarrow{\rho^\sigma_j}\Fcosheaf'_j(\tau)\surj\Fcosheaf_{j+1}(\tau)\right]\right)$
                    \ENDFOR
                    \STATE $\mathcal{N}^{\sigma\sigma}_j \gets \Fcosheaf'_j(\sigma)$
				\STATE $\Fcosheaf_{j+1}(\sigma)\gets \mathrm{Pushout}\left(\left[\rho^\sigma_j:\Fcosheaf_j(\sigma)\surj\Fcosheaf'_j(\sigma)\right],\mathrm{Pushout}_{\tau:\tau\leq\sigma}\left[\Fcosheaf_j(\sigma)\surj\mathcal{N}^{\sigma\tau}_j\right]\right)$
				\STATE $\Fcosheaf_{j+1}(\tau\leq\sigma)\gets$ induced by the pushout above
			\ENDFOR
		\ENDFOR
            \STATE $j\gets j+1$
	\ENDWHILE
	\STATE $\Mcosheaf \gets \Fcosheaf_j$
	\RETURN $\Mcosheaf$
\end{algorithmic}
\end{algorithm}

\begin{remark}
Algorithm~\ref{alg:mono} naturally generalises to sheaves over any category with kernels and pushouts, which in this case does not apply to \textbf{Set}.
\end{remark}
In Appendix~\ref{sec:diagrams} we also present the pushout diagrams described in Algorithm~\ref{alg:mono}.
If we again consider the specific example $\Ab=\vect$, we iteratively construct $\Mcosheaf$, as follows.
\begin{itemize}
    \item $\Kcosheaf_j'(\sigma)=\sum_{\tau:\sigma\leq\tau}\ker\left(\Fcosheaf_{j-1}(\sigma\leq\tau)\right)$ is the linear span of all kernels of maps out of $\sigma$.
    \item $\Kcosheaf_j(\sigma)=\sum_{\tau:\tau\leq\sigma}\Fcosheaf(\tau\leq\sigma)\left(\Kcosheaf_j'(\tau)\right)$ is the linear span of the image of all copies of $\Kcosheaf_j$ mapping into $\sigma$.
    \item $\Kcosheaf_j(\tau\leq\sigma)=\Fcosheaf(\tau\leq\sigma)_{\Kcosheaf_j(\sigma)}$ is the natural restriction map.
    \item $\Fcosheaf_j(\sigma)=\Fcosheaf(\sigma)/\Kcosheaf_j(\sigma)$ and $\Fcosheaf_j(\tau\leq\sigma)$ are the corresponding induced maps.
\end{itemize}

\begin{theorem} \label{thm:algorithms}
Suppose $K$ contains $n$ cells of maximum dimension $d\leq n$ and so that every cell has at most $c$ faces and $C$ cofaces.
Moreover, suppose $\Fsheaf$ and $\Fcosheaf$ are $\vect$-(co)sheaves with vector spaces bounded by dimension $D$ for each cell, and let $\omega$ be the exponent of matrix multiplication complexity.
\begin{enumerate}
\item Algorithm~\ref{alg:epi} computes the epification $\Esheaf\inj\Fsheaf$ in $O\left(n(d+ (c^\omega+C^\omega)D^{1+\omega})\right)$.
\item Algorithm~\ref{alg:mono} computes the monofication $\Fcosheaf\surj\Mcosheaf$ in $O\left(n(d+ (c^\omega+C^\omega)D^{1+\omega})\right)$.
\end{enumerate}
\end{theorem}

Note that $d$ refers to the geometric dimension of a cell and $D$ refers to the algebraic dimension of a vector space.

\begin{corollary}
Algorithm~\ref{alg:epi} and Algorithm~\ref{alg:mono} combine to compute the isobisheafification of $\Fbisheaf$ in polynomial time.
\end{corollary}

\begin{proof}[Proof of Theorem~\ref{thm:algorithms}]
To show the complexity, we study the two while loops of each algorithm.
At the $j^\mathrm{th}$ iteration of the first loop we may ignore cells of dimension greater than $d-j+1$, and in worst case scenario we must look at $n-j+1$ cells.
Thus, there are at worst $d$ total iterations and at most $n$ cells saved in each iteration, so the first loop terminates in $O(nd)$.

Each iteration of the second loop must reduce the dimension of at least one one vector space of the (co)sheaf until it terminates.
In the worst case scenario, this means we must iterate through the second loop $nD$ times -- one for each dimension of each vector space.
Further, calculating intersections, kernels, images and preimages are all bounded by matrix multiplication time (c.f. \cite{bunch1974triangular,milosavljevic2011zigzag}).
Each kernel, image and preimage is calculated from spaces of dimension bounded by dimension $D$, so these calculations have a complexity $O(D^\omega)$.
Taking the intersection or linear sum of $k$ such spaces (where $k\leq c$ in one half of the loop and $k\leq C$ in the other) will have computational complexity $O(k^\omega D^\omega)$.
Thus the second loop contributes $O\left(n(c^\omega+C^\omega)D^{1+\omega}\right)$.
Combining the two loops gives the desired complexity.
In particular, both algorithms must terminate.

Finally, we prove correctness of Algorithm~\ref{alg:epi}, as a dual argument completes the proof for Algorithm~\ref{alg:mono}.
By construction, Algorithm~\ref{alg:epi} must output an episheaf, as otherwise the second loop will never terminate.
Let $\Esheaf_\mathrm{max}$ be the maximal subepisheaf of $\Fsheaf$.
By construction, the output $\Esheaf$ of Algorithm~\ref{alg:epi} satisfies $\Esheaf\inj\Esheaf_\mathrm{max}\inj\Fsheaf$.
Since we work over $\vect$, this means $\Esheaf(\sigma)\subseteq\Esheaf_\mathrm{max}(\sigma)$ for each cell $\sigma$.
Conversely, we show by induction on $j$, the number of iterations on the second while loop of Algorithm~\ref{alg:epi} that $\Esheaf_\mathrm{max}(\sigma)\subseteq \Fsheaf_j(\sigma)$ for each $\sigma$, which completes the proof, since all maps of $\Esheaf,\Fsheaf_j,\Esheaf_\mathrm{max}$ are restrictions of maps of $\Fsheaf$.
Indeed, when $j=0$ we have $\Fsheaf_0=\Fsheaf$, and by definition of $\Esheaf_\mathrm{max}$ being a subepisheaf this means $\Esheaf_\mathrm{max}(\sigma)\subseteq\Fsheaf_0(\sigma)$ for each $\sigma$.
Now assume for some $r\geq0$ that $\Esheaf_\mathrm{max}(\sigma)\subseteq\Fsheaf_r(\sigma)$ for each $\sigma$.
Then
\[
\Fsheaf_{r+1}'(\sigma)=\bigcap_{\tau:\sigma\leq\tau}\im\left(\Fsheaf_{r}(\sigma\leq\tau)\right) = \bigcap_{\tau:\sigma\leq\tau}\Fsheaf(\sigma\leq\tau)\left(\Fsheaf_r(\tau)\right) \supseteq \bigcap_{\tau:\sigma\leq\tau}\Fsheaf(\sigma\leq\tau)\left(\Esheaf_\mathrm{max}(\tau)\right) = \Esheaf_\mathrm{max}(\sigma)
\]
where the containment comes from the assumption that $\Esheaf_\mathrm{max}\inj \Fsheaf_r$ and the final inequality holds because $\Esheaf_\mathrm{max}$ is an episheaf.
Hence $\Esheaf_\mathrm{max}(\sigma)\subseteq \Fsheaf'_{r+1}(\sigma)$ for each $\sigma$.
Likewise,
\[
\Fsheaf_{r+1}(\sigma)=\bigcap_{\tau:\tau\leq\sigma}\Fsheaf(\tau\leq\sigma)^{-1}\left(\Fsheaf_{r+1}'(\tau)\right) \supseteq \bigcap_{\tau:\tau\leq\sigma}\Fsheaf(\tau\leq\sigma)^{-1}\left(\Esheaf_\mathrm{max}(\tau)\right) = \Esheaf(\sigma)
\]
where the final equality is true for all sheaves.
In particular, if $\Esheaf_\mathrm{max}(\sigma)\subseteq\Fsheaf_r(\sigma)$ for each $\sigma$ this means $\Esheaf_\mathrm{max}(\sigma)\subseteq\Fsheaf_{r+1}(\sigma)$.
\end{proof}

\begin{remark}
The bounds for computational complexity in Theorem~\ref{thm:algorithms} can almost certainly be reduced, and we only intend this result to provide a proof of concept.
Among other things, this bound does not factor in the highly parallelizable nature of computations using sheaves, and we leave further optimization to future work.
\end{remark}

\section{From Local Systems to Toroidal Cycles}


In this section we outline how one can in practice extract the topological information of $K$ from a finite quotient $G$.

\subsection{Extracting the relevant data}

Suppose we have calculated the $d$ persistent local systems associated to $G$ as in Section~\ref{sec:full-case}.

Consider the PLS $\local_i(G)$ which has an isomorphism for each face relation.
We can entirely describe $\local_i(G)$ by an invertible matrix $M_i$ representing the monodromy of $G$.
Explicitly, if $\Sp^1$ has a cellulation with $\ell$ vertices, then given the sequence of face relations connecting $0$ and $1$
\begin{equation}
  v_1 \geq e_{12} \leq v_2 \geq e_{23} \leq \cdots \leq v_\ell \geq e_{\ell 1} \leq v_1  
  \label{eq:zigzag}
\end{equation}

we define
\[
M_i = \local_i(G)(e_{\ell 1} \leq v_1)^{-1}\circ \local_i(G)(e_{\ell 1} \leq v_\ell) \circ \cdots \circ \local_i(G)(e_{12} \leq v_2)^{-1} \circ \local_i(G)(e_{12} \leq v_1)
\]
The $1$-eigenvectors of $M_i$ correspond exactly to the toroidal cycles of $G$, as they are the homology classes that wrap around $\Sp^1$ and connect to themselves.
Similarly, the $1$-eigenvectors of of $M_i^k$ correspond to the toroidal cycles of the $k$-fold cover of $G$, $G_{i,k}$, whose $i^\mathrm{th}$ coordinate is contained in $[0,k]/0\sim k$.

We anticipate that under our assumption that $K$ is a locally compact and paracompact cellular complex, for some value of $k$ the embeddings in Theorem~\ref{thm:1d-toroidal} and Theorem~\ref{thm:toroidal} will in fact be an isomorphism $\iota(\Hg_{\bullet+1}(G_{i,k})/\Ig_{\bullet+1})\cong\local_i(G)$.
This would mean all classes of $\local_i(G)$ must represent a toroidal cycle of $G_{i,k}$ for some value of $k$.
Equivalently, we believe that $M_i$ is a root of the identity matrix, and therefore all eigenvalues must be roots of unity.
Instead of creating larger complexes, this would mean we are able to recover information about toroidal cycles at all scales simply by reading off the eigenvalues and Jordan block decomposition of $M_i$.
\begin{conjecture}
The eigenvalues of $M_i$ must be roots of unity.
\end{conjecture}

Now, having determined $M_i$, take a basis $\mathcal{B}_{v_1}$ of the $1$-eigenspace of $M_i$ at $\local_{i}(G)(v_1)$.
For each other cell of $\Sp^1$, map $\mathcal{B}_{v_1}$ to the (linearly independent) collection of vectors $\mathcal{B}_{\sigma}$ of $\local_{i}(G)(\sigma)$ via the sequence of face relations in Equation~\ref{eq:zigzag}.
For each $\sigma$, choose a (linearly independent) lift of $\mathcal{B}_{\sigma}$ to $\Esheaf_i(\sigma)_\bullet$ in the preimage of the cap product map.
This step is easily implemented in Sage with the preimage command, for instance, and the commutative diagrams of the isobisheaf $\Ibisheaf_i$ ensures this step is well defined.

Finally, we use the fact that all maps $\Esheaf_i(G)(e\leq v)$ are induced by inclusion and sum all local chains together to recover a toroidal cycle.
Note that the insistence to use the 1-eigenspace in the first step ensures we recover a cycle.
Moreover, all steps in this pipeline are well defined, since if we use a different choice of lift for an cells in the second step then the output cycle will differ by at most a non-toroidal cycle.
This means we recover a toroidal cycle basis independent of input conditions.

Of course, one may not only want toroidal cycles but also the dual concepts of cycles in $K$ which are trivial when mapped to $G$.
This is a much more difficult question to answer -- how can we recover information that we aren't directly presented with?
For $\Hg_0(K)$, thankfully no such cycles will exist, however the problem exists in degree $1$ and above.
For $\Hg_1(K)$, \cite{onus2022quantifying} show that all disappearing cycles can be recovered algorithmically as a ``commutator type'' sum of toroidal 1-cycles.
We conjecture that a similar procedure can be used to recover cycles in higher dimensions.
For example, suppose $K=(\R^2\times\Z)\cup(\Z\times\R\times\Z)\cup(\Z\times\R^2)$ is the collection of planes at integer coordinates in $\R^3$.
Then $K$ is 3-periodic and the smallest quotient space $G$ is a subset of $(\R/\Z)^3\cong\torus^3$.
Under this construction the 2-chain $[0,1]^2\times\{0\}$ and all its coordinate permutations are toroidal 2-cycles of $G$, and the boundary of $[0,1]^3$ is a cycle of $K$ which vanishes in $G$ and can be written in terms of these $2$-cycles and their translations.
See Figure~\ref{fig:toroidal-builds-disappearing} for a visualization.

\begin{figure}
    \centering
    \includegraphics[width=0.4\linewidth]{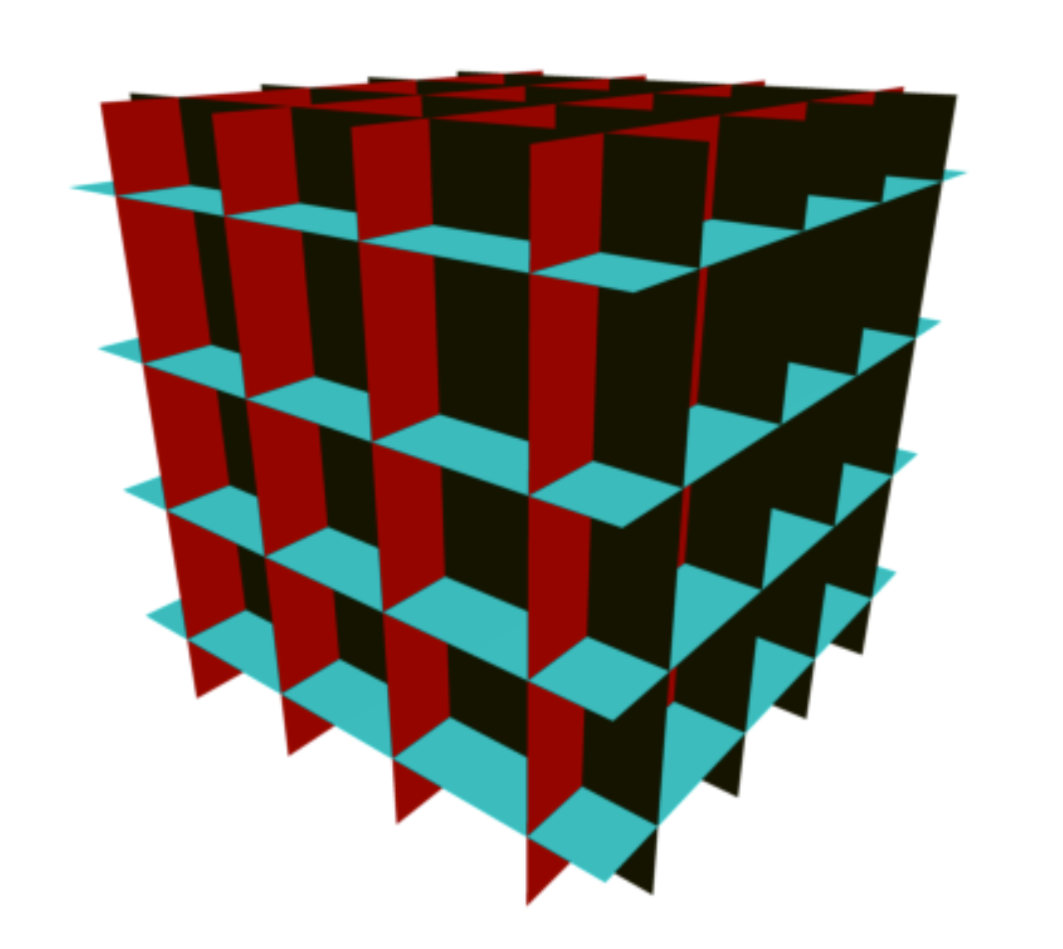}
    \label{fig:toroidal-builds-disappearing}
    \caption{A segment of the space $K=(\R^2\times\Z)\cup(\Z\times\R\times\Z)\cup(\Z\times\R^2)$ built from intersecting planes with integer $x$ (in red), $y$ (shaded) and $z$ (in cyan) coordinates. The boundary of the box $[0,1]^3$ and all its translated copies are non-trivial cycles in $K$ but are trivial when quotienting by the maximal translational group $\Z^3$. }
\end{figure}

\begin{conjecture}
All cycles in $K$ which are trivial in $G$ can be written in terms of lifts of toroidal cycles.
\end{conjecture}

\subsection{Examples}

We now illustrate the process above with three worked examples which highlight how and why we detect toroidal cycles through persistent local systems.

\subsubsection{The Running Example}

Recall the running example in Figure~\ref{fig:running-example}.
In Section~\ref{sec:prelim} we constructed the associated bisheaf, whose corresponding isobisheaf $(\Esheaf_1,\Mcosheaf_0,I)$ is
\[
\begin{tikzcd}[ampersand replacement=\&,column sep = 6em, row sep=3em]
\cdots 
\& \F^5
	\arrow[two heads,l]
	\arrow[two heads,r,"{\scalemath{\scalesize}{\begin{bsmallmatrix} 1&0&1&0&0\\0&0&1&0&0\\0&0&0&0&1\\0&0&0&1&0 \end{bsmallmatrix}}}" description]
	\arrow[two heads,d,"{\scalemath{\scalesize}{\begin{bsmallmatrix} 1&1&1&1&1  \end{bsmallmatrix}}}" description]
\& \F^4
	\arrow[two heads,d,"{\scalemath{\scalesize}{\begin{bsmallmatrix} 1&1&1&1  \end{bsmallmatrix}}}" description]
\& \F^5
	\arrow[two heads,l,swap,"{\scalemath{\scalesize}{\begin{bsmallmatrix} 1&1&0&0&0\\0&0&0&1&0\\0&1&0&0&0\\0&0&0&0&1 \end{bsmallmatrix}}}" description]
	\arrow[two heads,r,"{\scalemath{\scalesize}{\begin{bsmallmatrix} 1&0&1&0&0\\0&0&1&0&0\\0&0&0&0&1\\0&0&0&1&0 \end{bsmallmatrix}}}" description]
	\arrow[two heads,d,"{\scalemath{\scalesize}{\begin{bsmallmatrix} 1&1&1&1&1 \end{bsmallmatrix}}}" description]
\& \F^4
	\arrow[two heads,d,"{\scalemath{\scalesize}{\begin{bsmallmatrix} 1&1&1&1  \end{bsmallmatrix}}}" description]
\& \cdots
	\arrow[two heads,l]
\\
\cdots 
	\arrow[r,"{id}" description]
\& \F
\& \F
	\arrow[l,swap,"{id}" description]
	\arrow[r,"{id}" description]
\& \F
\& \F
	\arrow[l,swap,"{id}" description]
	\arrow[r,"{id}" description]
\& \cdots
\end{tikzcd}
\]
In this case the persistent local system $\local(G)_0$ is constant and 1-dimensional, so Theorem~\ref{thm:1d-toroidal} says $\Hg_1(G)/\Ig_1=\F$.
There are many choices of classes of $\Esheaf$ which map to this class via $I$, all of which are represented by toroidal 1-cycles in $G$ which lift in $K$ to infinite paths from $-\infty$ to $\infty$, any two of which differ by copies of the non-toroidal triangles.
In particular, the horizontal line $\R\times\{1\}$ is one such toroidal cycle.

\subsubsection{An example for why we need bisheaves} \label{sec:needing-sheaves}

Consider the 1-periodic shape $K$ in Figure~\ref{fig:torsion}.
Explicitly, the outer blue cylinder is the set
\[
\R\times \left\{(y,z)\,:\,y^2+z^2=1\right\},
\]
the two inner red tubes are the sets
\[
\left\{\left(t,\pm\frac{1}{2}\cos(\pi t)+\frac{1}{5}\cos(2\pi s), \pm\frac{1}{2}\sin(\pi t)+\frac{1}{5}\sin(2\pi s)\right)\,:\,t,s\in\R\right\}
\]
and the yellow discs are the sets
\[
\left\{(t,s,n)\,:\,t^2+s^2\leq 1\right\}\bigg\backslash \left\{(t,s,n)\,:\,\left(t\pm\frac{1}{2}\right)^2+s^2<\frac{1}{25}\right\}
\]
for $n\in\Z$.
Take a cellulation of $\Sp^1$ (or $\R$) which has one vertex located below every yellow disc and one vertex half way between two discs, and consider a quotient space $G$ of $K$.
In this case both the degree~1 and degree~2 bisheaves will be non-trivial, so both 1- and 2-cycles may be toroidal.

\begin{figure}[ht]
    \centering
    \includegraphics[width=0.52\linewidth]{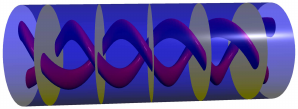}
    \includegraphics[width=0.43\linewidth]{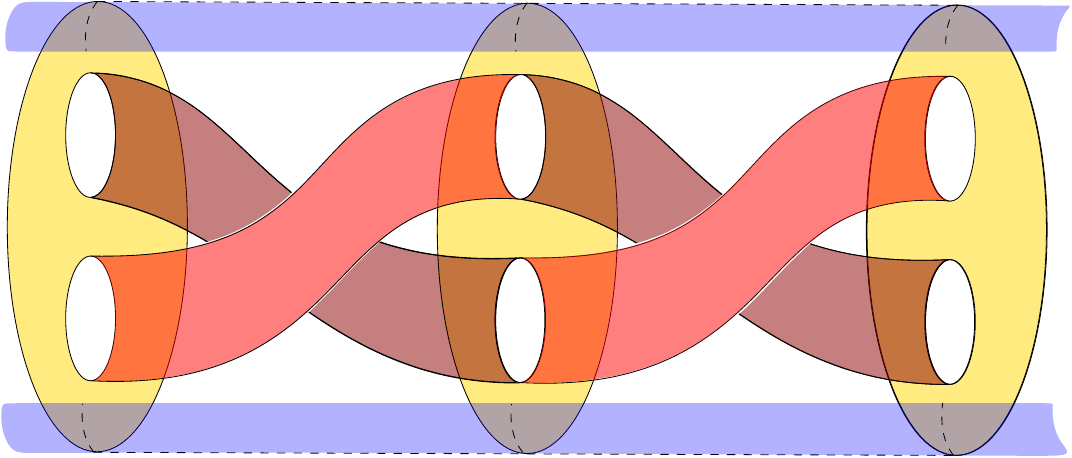}
    \caption{Left: a 2-dimensional 1-periodic stratified space in $\R^3$. The shape is the union of an outer cylinder $\R\times\Sp^1$ (in blue), two inner twisting tubes (in red) and discs at integer intervals (in yellow) with holes removed where they intersect the inner tubes. Right: an interior diagram of the shape.}
    \label{fig:torsion}
\end{figure}

The degree~1 isobisheaf $(\Esheaf_1,\Mcosheaf_0,I)$ will take the form below
\[
\begin{tikzcd}[ampersand replacement=\&,column sep = 5em, row sep=3em]
\cdots 
\& \F^4
	\arrow[two heads,l]
	\arrow[two heads,r,"{\scalemath{\scalesize}{\begin{bsmallmatrix} 1&0&0&0\\0&1&0&0\\0&0&1&0 \end{bsmallmatrix}}}" description]
	\arrow[two heads,d,"{\scalemath{\scalesize}{\begin{bsmallmatrix} 1&1&1&0 \end{bsmallmatrix}}}" description]
\& \F^3
	\arrow[d,"{\scalemath{\scalesize}{\begin{bsmallmatrix} 1&1&1 \end{bsmallmatrix}}}" description]
\& \F^3
	\arrow[two heads,l,"{id}" description]
	\arrow[two heads,r,"{\scalemath{\scalesize}{\begin{bsmallmatrix} 1&0&0\\0&0&1\\0&1&0 \end{bsmallmatrix}}}" description]
	\arrow[d,"{\scalemath{\scalesize}{\begin{bsmallmatrix} 1&1&1 \end{bsmallmatrix}}}" description]
\& \F^3
	\arrow[d,"{\scalemath{\scalesize}{\begin{bsmallmatrix} 1&1&1 \end{bsmallmatrix}}}" description]
\& \F^4
	\arrow[two heads,l,"{\scalemath{\scalesize}{\begin{bsmallmatrix} 1&0&0&1\\0&1&0&-1\\0&0&1&0 \end{bsmallmatrix}}}" description]
	\arrow[d,"{\scalemath{\scalesize}{\begin{bsmallmatrix} 1&1&1&0 \end{bsmallmatrix}}}" description]
	\arrow[two heads,r]
\& \cdots
\\
\cdots 
	\arrow[r,"{id}" description]
\& \F
\& \F
	\arrow[l,"{id}" description]
	\arrow[r,"{id}" description]
\& \F
\& \F
	\arrow[r,"{id}" description]
	\arrow[l,"{id}" description]
\& \F
\& \cdots
	\arrow[l,"{id}" description]
\end{tikzcd}
\]
In this case $\local(G)_0$ is again constant and 1-dimensional, so $\Hg_1(G)/\Ig_1 = \F$.
One such representative is a horizontal line across the outer blue tube, although one may also choose to take a horizontal path along either of the red tubes instead.
The difference of any pair of these cycles is non-toroidal, and this is captured in $\Esheaf_1$.
In particular, this difference can be written in terms of non-toroidal 1-cycles corresponding to paths which start on a yellow disc, travel to the next yellow disc along the blue tube, and return to the starting point via one of the red tubes.
Note, however, that the non-toroidal 1-cycles formed by the cross section of each tube pass to trivial classes of all components of the isobisheaf.

On the other hand, the degree~2 isobisheaf $(\Esheaf_2,\Mcosheaf_1,I)$ will take the form
\[
\begin{tikzcd}[ampersand replacement=\&,column sep = 5em, row sep=3em]
\cdots 
\& \F^4
	\arrow[two heads,l]
	\arrow[two heads,r,"{\scalemath{\scalesize}{\begin{bsmallmatrix} 1&0&0&0\\0&1&0&0\\0&0&1&0 \end{bsmallmatrix}}}" description]
	\arrow[two heads,d,"{\scalemath{\scalesize}{\begin{bsmallmatrix} 1&0&1&0\\0&1&1&0 \end{bsmallmatrix}}}" description]
\& \F^3
	\arrow[two heads, d,"{\scalemath{\scalesize}{\begin{bsmallmatrix} 1&0&1\\0&1&1 \end{bsmallmatrix}}}" description]
\& \F^3
	\arrow[two heads,l,"{id}" description]
	\arrow[two heads,r,"{\scalemath{\scalesize}{\begin{bsmallmatrix} 0&1&0\\1&0&0\\0&0&1 \end{bsmallmatrix}}}" description]
	\arrow[two heads, d,"{\scalemath{\scalesize}{\begin{bsmallmatrix} 1&0&1\\0&1&1 \end{bsmallmatrix}}}" description]
\& \F^3
	\arrow[two heads, d,"{\scalemath{\scalesize}{\begin{bsmallmatrix} 1&0&1\\0&1&1 \end{bsmallmatrix}}}" description]
\& \F^4
	\arrow[two heads,l,"{\scalemath{\scalesize}{\begin{bsmallmatrix} 1&0&0&1\\0&1&0&1\\0&0&1&-1 \end{bsmallmatrix}}}" description]
	\arrow[two heads, d,"{\scalemath{\scalesize}{\begin{bsmallmatrix} 1&0&1&0\\0&1&1&0 \end{bsmallmatrix}}}" description]
	\arrow[two heads,r]
\& \cdots
\\
\cdots 
	\arrow[r]
\& \F^2
\& \F^2
	\arrow[l,"{id}" description]
	\arrow[r,"{id}" description]
\& \F^2
\& \F^2
	\arrow[r,"{id}" description]
	\arrow[l,"{\scalemath{\scalesize}{\begin{bsmallmatrix} 0&1\\1&0 \end{bsmallmatrix}}}" description]
\& \F^2
\& \cdots
	\arrow[l]
\end{tikzcd}
\]
In this case $\local(G)_1$ is 2-dimensional, and in particular, the two generating classes correspond exactly to the non-toroidal cross-sections of the red tubes which weren't detected by $\Esheaf_1$.
However, not all maps of $\local(G)_1$ are identity, so we need to do more work in order to detect torooidal cycles.
By composing all maps, the path from the vertex under one yellow disc to it's adjacent copy induces the endomorphism
\[
A=\begin{bmatrix} 0&1\\1&0 \end{bmatrix}
\]
where $A^{2k}=id$ and $A^{2k+1}=A$ for each $k\in\N$.
Thus for quotient space where we quotient out by even unit translations, this means we have $\Hg_1(G)/\Ig_1 = \F^2$, where we may represent the toroial 2-cycles by the two red tubes.
However, when we quotient out by odd unit translations, our toroidal 2-cycles correspond to the $1$-eigenspace of $A$, which is $1$-dimensional, so $\Hg_1(G)/\Ig_1 = \F$.
The generating toroidal cycle can be represented by the \textit{sum} of the two red tubes -- which in this case will be identified as one tube in the quotient -- or otherwise by the outer blue cylinder.
The connection between the two cases is that if we begin with the cross-section of one red tube and build a toroidal cycle from this, if we quotient with respect to translation by an odd number of units, then we must wrap twice around $G$ before we will create a boundaryless cycle.

This is also an example of why we require the bisheaf construction and cannot simply detect toroidal cycles with the cap product of $G$.
Suppose, for instance, that $G$ is the minimal quotient space of $K$, where we identify all points one unit translation away from each other so that the quotient space only has a single yellow disc.
Then the difference of both inner cycles of the disc (where the red tubes intersect) is the boundary of one of the red tubes, meaning the sum of these cycles is 0 in characteristic 2.
But the cap product of the outer blue cylinder with the orientation of $\Sp^1$ is homologous to the outer circle of the yellow disc, which is in turn homologous to the sum of the inner cycles.
So in $G$ for $\F$ in characteristic~2, the cap product maps the blue cylinder to 0 as well as non-toroidal cycles, however in the bisheaf this map is non-zero and we still recover that the cylinder is a toroidal 2-cycle.
One can also extend this example to find further counterexamples for homology with field (or even ring) coefficients which are not characteristic~2 by incorporating torsion in the form of a ``Klein Bottle twist'' to one of the red tubes between any pair of yellow disks.

\subsubsection{An example in higher dimensions}

Consider $K$, the Schwarz~P minimal surface in Figure~\ref{fig:schwarz}, which is a 3-periodic surface in $\R^3$ satisfying the equation
\begin{equation}
P(x)=\cos(2\pi x)+\cos(2\pi y)+\cos(2\pi z)=0.
\label{eq:schwarz}
\end{equation}
\begin{figure}[ht]
    \centering
    \includegraphics[scale=1.5]{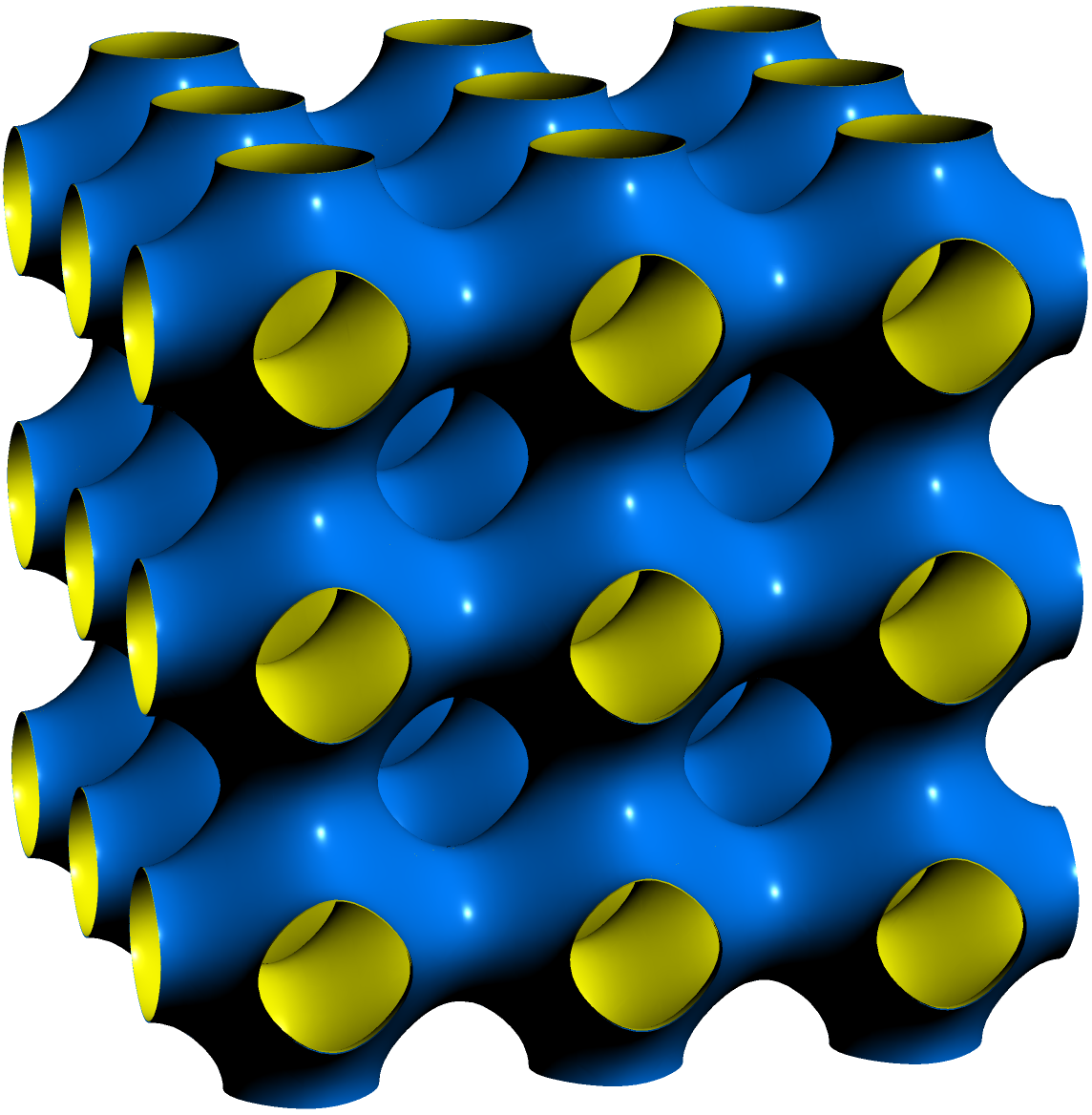}
    \caption{The segment of the Schwarz~P minimal surface in $\R^3$ which satisfies Equation~\ref{eq:schwarz}. The surface is oriented, with the exterior coloured blue and the interior coloured yellow.}
    \label{fig:schwarz}
\end{figure}

Take the quotient space of $K$ with respect to $n$-integer translations.
That is, let $G$ be $K\cap [0,n]^3$ with periodic boundary conditions imposed, and take a cellulation of $\Sp^1$ with vertices at integer and half-integer coordinates, and give $\torus^3$ the product of this cellulation.
Let us look for toroidal 2-cycles.
While $K$ is 3-periodic, the rotational symmetries of the surface embedding mean that all three bisheaves around the coordinates of $\Sp^1$ will be identical, and likewise for all three bisheaves around $\torus^2$.
The bisheaf around $\torus^3$ will also have trivial PLS, since the cap product of $2$-cells with the (degree~3) orientation of $\torus^3$ will always be $0$.
Thus we need only look at one 1-periodic and one 2-periodic PLS.

The 1-periodic isobisheaf will take the following form
\[
\begin{tikzcd}[ampersand replacement=\&,column sep = 5em, row sep=3em]
\cdots 
\& \F
	\arrow[two heads,l,"{id}" description]
	\arrow[two heads,r,"{id}" description]
	\arrow[two heads,d,"{id}" description]
\& \F
	\arrow[two heads,d,"{id}" description]
\& \F
	\arrow[two heads,l,"{id}" description]
	\arrow[two heads,r,"{id}" description]
	\arrow[two heads,d,"{id}" description]
\& \F
	\arrow[two heads,d,"{id}" description]
\& \F
	\arrow[two heads,l,"{id}" description]
	\arrow[two heads,r,"{id}" description]
	\arrow[two heads,d,"{id}" description]
\& \cdots
\\
\cdots 
	\arrow[r,"{id}" description]
\& \F
\& \F
	\arrow[l,"{id}" description]
	\arrow[r,"{id}" description]
\& \F
\& \F
	\arrow[r,"{id}" description]
	\arrow[l,"{id}" description]
\& \F
\& \cdots
	\arrow[l,"{id}" description]
\end{tikzcd}
\]
where both the episheaf and monocosheaf are (co)local systems.
In particular, the single class of the PLS $\local_1(G)_1$ is represented by the sum of cycles at each cross-section of the surface, and in turn we easily recover that the toroidal 2-cycle is (uniquely) represented by the entire surface.

Similarly, the episheaf component of the 2-periodic isobisheaf is simply a local system
\[
\begin{tikzcd}[ampersand replacement=\&,column sep = 4em, row sep=2.5em]
\ddots 
\& \vdots
\& \vdots
\& \vdots
\& \vdots
\& \vdots
\&
\\
\cdots 
\& \F
	\arrow[l,"{id}" description]
	\arrow[u,"{id}" description]
	\arrow[d,"{id}" description]
	\arrow[r,"{id}" description]
\& \F
	\arrow[u,"{id}" description]
	\arrow[d,"{id}" description]
\& \F
	\arrow[l,"{id}" description]
	\arrow[u,"{id}" description]
	\arrow[d,"{id}" description]
	\arrow[r,"{id}" description]
\& \F
	\arrow[u,"{id}" description]
	\arrow[d,"{id}" description]
\& \F
	\arrow[l,"{id}" description]
	\arrow[u,"{id}" description]
	\arrow[d,"{id}" description]
	\arrow[r,"{id}" description]
\& \cdots
\\
\cdots 
\& \F
	\arrow[l,"{id}" description]
	\arrow[r,"{id}" description]
\& \F
\& \F
	\arrow[l,"{id}" description]
	\arrow[r,"{id}" description]
\& \F
\& \F
	\arrow[l,"{id}" description]
	\arrow[r,"{id}" description]
\& \cdots
\\
\cdots 
\& \F
	\arrow[l,"{id}" description]
	\arrow[u,"{id}" description]
	\arrow[d,"{id}" description]
	\arrow[r,"{id}" description]
\& \F
	\arrow[u,"{id}" description]
	\arrow[d,"{id}" description]
\& \F
	\arrow[l,"{id}" description]
	\arrow[u,"{id}" description]
	\arrow[d,"{id}" description]
	\arrow[r,"{id}" description]
\& \F
	\arrow[u,"{id}" description]
	\arrow[d,"{id}" description]
\& \F
	\arrow[l,"{id}" description]
	\arrow[u,"{id}" description]
	\arrow[d,"{id}" description]
	\arrow[r,"{id}" description]
\& \cdots
\\
\cdots 
\& \F
	\arrow[l,"{id}" description]
	\arrow[r,"{id}" description]
\& \F
\& \F
	\arrow[l,"{id}" description]
	\arrow[r,"{id}" description]
\& \F
\& \F
	\arrow[l,"{id}" description]
	\arrow[r,"{id}" description]
\& \cdots
\\
\cdots 
\& \F
	\arrow[l,"{id}" description]
	\arrow[u,"{id}" description]
	\arrow[d,"{id}" description]
	\arrow[r,"{id}" description]
\& \F
	\arrow[u,"{id}" description]
	\arrow[d,"{id}" description]
\& \F
	\arrow[l,"{id}" description]
	\arrow[u,"{id}" description]
	\arrow[d,"{id}" description]
	\arrow[r,"{id}" description]
\& \F
	\arrow[u,"{id}" description]
	\arrow[d,"{id}" description]
\& \F
	\arrow[l,"{id}" description]
	\arrow[u,"{id}" description]
	\arrow[d,"{id}" description]
	\arrow[r,"{id}" description]
\& \cdots
\\
\phantom{v}
\& \vdots
\& \vdots
\& \vdots
\& \vdots
\& \vdots
\& \ddots
\end{tikzcd}
\]
However, the PLS $\local_{\{1,2\}}(G)_0$ will be trivial.
One can see this for instance by considering what occurs at the vertex $v$ of $\torus^2$ located at the origin.
Here, the map $G\supsetneq\pi^{-1}(\St(v))\to \St(v)$ is not surjective since $P(x)>1>0$ in some neighbourhood of the origin. Hence the cap product of any cycle of $\pi^{-1}(\St(v))$ with the pullback of any orientation cocycle with support in this neighbourhood will be zero.

This is an example of why we need to look at 1-periodic PLS's, $\local_{i}(G)$ for $i=1,2,3$, in order to detect toroidal cycles.
The entire surface is a toroidal 2-cycle, however because it is not a 2-cycle of the ambient 3-torus, it will not be detected by the cap product map.
So even though the 2-cycle is 3-periodic, what we really detect is that its cross-sections in all three directions are 1-cycles of the ambient torus. 
That being said, looking at these higher-periodicity PLS's like $\local_{\{i,j\}}(G)$ for $1\leq i<j\leq 3$ will help us learn more about the shape of our periodic space.
In this case, the fact $\local_{\{1,2\}}(G)$ is trivial here helps us distinguish the Schwarz P surface from the periodic space in Figure~\ref{fig:toroidal-builds-disappearing} which will have toroidal 2-cycles represented in the 2-dimensional PLS.

Another point to notice in all of these examples is that the local construction of the PLS is (locally) independent of the quotient space chosen.
For 1-periodic spaces this was an essential property of covering space theory which helped us prove Theorem~\ref{thm:1d-toroidal}, but in general for a $d$-periodic space this will not be the case for $\local_{L}(G)$ with $|L|<d$.
Indeed, for the Schwarz~P minimal surface, if we calculated $\local_1(G)_0$ we would find a dependence on $n$.
Ideally, we would like this theory to produce results entirely independent of the choice of quotient space, and while this is not yet complete we believe this can still be achieved, and we leave the study of this to future work.

\section*{Acknowledgements}\label{ackref}
The authors would like to thank Amit Patel for numerous discussion, suggestions, and  clarifications. 
A.O.~would also like to thank Ulrike Tillman for her insightful comments, which helped simplify the approach in Section~\ref{sec:toroidal} when they were invited to present this work at the University of Oxford.
A.O.~acknowledges the support of the Additional Funding Programme for Mathematical Sciences, delivered by EPSRC (EP/V521917/1) and the Heilbronn Institute for Mathematical Research.

\bibliographystyle{plainurl} 
\bibliography{refs}

\appendix

\section{Existence of Persistent Local Systems} \label{sec:PLS-lemmas}

In order to prove the existence of persistence local systems (PLS's) in arbitrary abelian categories, we first establish the existence of a maximal sub-episheaf.
We will show that given any two sub-episheaves there is a third sub-episheaf which is larger than both.
Then existence of a maximal episheaf follows by a standard application of Zorn's Lemma on the poset category of sub-episheaves (c.f. Section~3 of \cite{macpherson2021persistent}).

\begin{lemma} \label{lem:bigger-epi}
Given two sub-episheaves $\Esheaf_1\inj\Fsheaf$ and $\Esheaf_2\inj\Fsheaf$ there exists another sub-episheaf $\Esheaf_3\inj\Fsheaf$ such that $\Esheaf_1\inj\Esheaf_3\hookleftarrow\Esheaf_2$.
\end{lemma}

Note that in full generality, the following proof applies to sheaves valued in any category where all images, equalizers and binary coproducts exist, and all canonical embedding maps of coproducts are monic.
This includes non-abelian categories like $\mathbf{Set}$.

\begin{proof}
$\sheafcat(X)$ is abelian, so all binary biproducts (in particular coproducts) exist.
Given $f_1:\Esheaf_1\inj\Fsheaf$ and $f_2:\Esheaf_2\inj\Fsheaf$ this means there is a canonical map $f_1\oplus f_2:\Esheaf_1\oplus\Esheaf_2\to\Fsheaf$ and $\sheaf{\im(f_1\oplus f_2)}\inj\Fsheaf$.
Let $\Esheaf_3:=\sheaf{\im(f_1\oplus f_2)}$; we will show that $\Esheaf_3$ is an episheaf, and that there are canonical monomorphisms $\Esheaf_1\inj\Esheaf_3\hookleftarrow\Esheaf_2$, which concludes the proof.

First, for $i=1,2$, let $j_i:\Esheaf_i\inj\Esheaf_1\oplus\Esheaf_2$ be the canonical embedding morphisms.
We notice $\sheaf{\im(f_i)}\cong\Esheaf_i$ since $f_i$ is monic, so by the universality condition of the image, there is a unique monomorphism $g_i:\Esheaf_i\inj\Esheaf_3$ which makes the following diagram commute.
\begin{center}
\begin{tikzcd}[column sep={2.2cm,between origins},row sep={1.8cm,between origins}]
	\Esheaf_i &  \Esheaf_1\oplus\Esheaf_2\\
	\Fsheaf & \Esheaf_3
    \arrow[hook, "f_i", from=1-1, to=2-1]
    \arrow[hook, "j_i", from=1-1, to=1-2]
    \arrow[hook', from=2-2, to=2-1]
    \arrow[hook, dashed, "g_i", from=1-1, to=2-2]
    \arrow[two heads, "f_1\oplus f_2", from=1-2, to=2-2]
\end{tikzcd}
\end{center}

Next, we show that if $\Esheaf_1$ and $\Esheaf_2$ are episheaves then so is $\Esheaf_1\oplus\Esheaf_2$.
Indeed, consider the following commutative diagrams where the map $\Esheaf_1\oplus\Esheaf_2(\tau\leq\sigma)$ is the unique map from $\Esheaf_1(\sigma)\oplus\Esheaf_2(\sigma)\to\Esheaf_1(\tau)\oplus\Esheaf_2(\tau)$ induced by $\Esheaf_1(\tau\leq\sigma)$ and $\Esheaf_2(\tau\leq\sigma)$.
\begin{center}
\begin{tikzcd}[column sep={2.6cm,between origins},row sep={2cm,between origins}]
	\Esheaf_1(\sigma) &  \Esheaf_1(\sigma)\oplus\Esheaf_2(\sigma) & \Esheaf_2(\sigma) \\
	\Esheaf_1(\tau) & \Esheaf_1(\tau)\oplus\Esheaf_2(\tau) & \Esheaf_2(\tau)
    \arrow[hook, "j_1(\sigma)", from=1-1, to=1-2]
    \arrow[hook', "j_2(\sigma)"', from=1-3, to=1-2]
    \arrow[hook, "j_1(\tau)", from=2-1, to=2-2]
    \arrow[hook', "j_2(\tau)"', from=2-3, to=2-2]
    \arrow[two heads, "\Esheaf_1(\tau\leq\sigma)"', from=1-1, to=2-1]
    \arrow[two heads, "\Esheaf_2(\tau\leq\sigma)", from=1-3, to=2-3]
    \arrow[from=1-1, to=2-2]
    \arrow[from=1-3, to=2-2]
    \arrow[dashed, "\Esheaf_1\oplus\Esheaf_2(\tau\leq\sigma)"  description, from=1-2, to=2-2]
\end{tikzcd}
\end{center}
To show $\Esheaf_1\oplus\Esheaf_2(\tau\leq\sigma)$ is epic, suppose $\alpha,\beta:\Esheaf_1(\tau)\oplus\Esheaf_2(\tau) \to A$ are such that $\alpha\circ\left[\Esheaf_1\oplus\Esheaf_2(\tau\leq\sigma)\right]=\beta\circ\left[\Esheaf_1\oplus\Esheaf_2(\tau\leq\sigma)\right]$.
Precomposing with $j_i$ for $i=1,2$ and chasing the above diagram, it follows that 
\[
\alpha\circ j_i(\tau)\circ\Esheaf_i(\tau\leq\sigma)=\alpha\circ\left[\Esheaf_1\oplus\Esheaf_2(\tau\leq\sigma)\right]\circ j_i(\sigma)=\beta\circ\Esheaf_1\oplus\Esheaf_2(\tau\leq\sigma)\circ j_i(\sigma)=\beta\circ\circ\Esheaf_i(\tau\leq\sigma).
\]
But $\Esheaf_1(\tau\leq\sigma)$ and $\Esheaf_2(\tau\leq\sigma)$ are epic, so $\alpha\circ j_1(\tau)=\beta\circ j_1(\tau)$ and $\alpha\circ j_2(\tau)=\beta\circ j_2(\tau)$.
But if $\alpha\circ j_1(\tau)=\beta\circ j_1(\tau):\Esheaf_1(\tau)\to A$ and $\alpha\circ j_2(\tau)=\beta\circ j_2(\tau):\Esheaf_1(\tau)\to A$ then by universality of the biproduct\footnote{In particular, we only require universality of the coproduct.} $\Esheaf_1(\tau)\oplus\Esheaf_2(\tau)$, it follows $\alpha=\beta$ and hence $\Esheaf_1\oplus\Esheaf_2(\tau\leq\sigma)$ is epic.

Finally, having established that $\Esheaf_1\oplus\Esheaf_2$ is an episheaf, we now use the fact that $\Esheaf_1\oplus\Esheaf_2\surj \Esheaf_3$ (c.f. \cite[Theorem~2.17]{freyd1964abelian}) to show $\Esheaf_3$ is an episheaf.
To this end, consider the following commutative diagram.
\begin{center}
\begin{tikzcd}[column sep={2.6cm,between origins},row sep={2cm,between origins}]
	\Esheaf_1(\sigma)\oplus\Esheaf_2(\sigma) & & \Esheaf_3(\sigma) \\
	\Esheaf_1(\tau)\oplus\Esheaf_2(\tau) & &  \Esheaf_3(\tau)
    \arrow[two heads, "\Esheaf_1\oplus\Esheaf_2(\tau\leq\sigma)"', from=1-1, to=2-1]
    \arrow["\Esheaf_3(\tau\leq\sigma)", from=1-3, to=2-3]
    \arrow[two heads, "f_1(\sigma)\oplus f_2(\sigma)", from=1-1, to=1-3]
    \arrow[two heads, "f_1(\tau)\oplus f_2(\tau)", from=2-1, to=2-3]
\end{tikzcd}
\end{center}
To show $\Esheaf_3(\tau\leq\sigma)$ is epic, suppose $\alpha,\beta:\Esheaf_3(\tau)\to A$ are such that $\alpha\circ\Esheaf_3(\tau\leq\sigma) = \beta\circ \Esheaf_3(\tau\leq\sigma)$.
Precomposing with $f_1(\sigma)\oplus f_2(\sigma)$ we get 
\[
\alpha\circ\Esheaf_3(\tau\leq\sigma)\circ\left[f_1(\sigma)\oplus f_2(\sigma)\right] = \beta\circ\Esheaf_3(\tau\leq\sigma)\circ\left[f_1(\sigma)\oplus f_2(\sigma)\right]
\]
and chasing the above diagram then shows
\[
\alpha\circ \left[f_1(\tau)\oplus f_2(\tau)\right]\circ \left[ \Esheaf_1\oplus\Esheaf_2(\tau\leq\sigma)\right] = \beta\circ \left[f_1(\tau)\oplus f_2(\tau)\right]\circ \left[ \Esheaf_1\oplus\Esheaf_2(\tau\leq\sigma)\right]
\]
But then $\alpha\circ \left[f_1(\tau)\oplus f_2(\tau)\right]= \beta\circ \left[f_1(\tau)\oplus f_2(\tau)\right]$ since $\Esheaf_1\oplus\Esheaf_2(\tau\leq\sigma)$ is epic and in turn $\alpha=\beta$ since $f_1(\tau)\oplus f_2(\tau)$ is epic.
This $\Esheaf_3(\tau\leq\sigma)$ and hence $\Esheaf_3$ is an episheaf.

\end{proof}

Next, we will establish the existence of a minimal quotient-monocosheaf.
We will show that given any two quotient-monocosheaves there is a third quotient-monocosheaf which is smaller than both.
Then, as before, existence of a minimal quotient-monocosheaf follows from Zorn's Lemma  on the poset category of quotient-monocosheaves (c.f. Section~4 of \cite{macpherson2021persistent}).

\begin{lemma}
Given two quotient-monocosheaves $\Fcosheaf\surj\Mcosheaf_1$ and $\Fcosheaf\surj\Mcosheaf_2$ there exists another quotient-monocosheaf $\Fcosheaf\surj\Mcosheaf_3$ such that $\Mcosheaf_1\twoheadleftarrow\Mcosheaf_3\surj\Mcosheaf_2$.
\end{lemma}

Note that in full generality, the following proof applies to cosheaves valued in any category where all coimages, coequalizers and binary products exist, and all canonical projection maps of products are epic.

\begin{proof}
$\cosheafcat(X)$ is abelian, so all binary products (in particular products) exist.
Given $\rho_1:\Fcosheaf\surj\Mcosheaf_1$ and $\rho_2:\Fcosheaf\surj\Mcosheaf_2$ this means there will be a canonical map $\rho_1\oplus\rho_2:\Fcosheaf\to\Mcosheaf_1\oplus\Mcosheaf_2$ and $\Fcosheaf\surj \cosheaf{\mathrm{coim}(f_1\oplus f_2)}$.
Then $\Mcosheaf_3:=\cosheaf{\mathrm{coim}(f_1\oplus f_2)}$ is the desired quotient-monocosheaf and the proof that $\Mcosheaf_3$ is a monocosheaf and that there are canonical epimorphisms $\Mcosheaf_1\twoheadleftarrow\Mcosheaf_3\surj\Mcosheaf_2$ follows from dual arguments of the proof in Lemma~\ref{lem:bigger-epi}.
\end{proof}

The above results as well as the discussion in Section~\ref{sec:prelim} establishes the well-definedness of isobisheafification.
Finally, to complete the construction of PLS's, we present a proof of Lemma~\ref{lem:colocal}.
Note that in full generality, the proof applies to bisheaves valued in any category where all images and equalizers exist and all morphisms which are both epic and monic are isomorphisms.
This includes non-abelian categories like $\mathbf{Set}$.

\begin{proof}[Proof of Lemma~\ref{lem:colocal}]
Let $\local$ denote the image of $I$.
By the category theoretic definition of images we have the following commutative diagram in $\Ab$:
\[
\begin{tikzpicture}[thick]
\node (sheaft) at (0,2) { $\Esheaf(\tau)$};
\node (sheafs) at (4,2) {$\Esheaf(\sigma)$};
\node (cosheaft) at (0,0) { $\Mcosheaf(\tau)$};
\node (cosheafs) at (4,0) { $\Mcosheaf(\sigma)$};
\node (localt) at (-2,-1) { $\local(\tau)$};
\node (locals) at (6,-1) { $\local(\sigma)$};

\path
(sheaft) edge[->>] node [above]  {$\Esheaf(\sigma\leq\tau)$} (sheafs)
(cosheafs) edge[left hook ->] node [above] {$\Mcosheaf(\sigma\leq\tau)$} (cosheaft)
(locals) edge[left hook ->] (cosheafs)
(localt) edge[right hook ->] (cosheaft)
(localt) edge[right hook ->,dashed] node [above] {$!f$} (locals)
(sheaft) edge[->] node [right]  {$I_\tau$} (cosheaft)
(sheafs) edge[->] node [right]  {$I_\sigma$} (cosheafs)
(sheaft) edge[->>] node [right]  {$\tilde{I}_\tau$} (localt)
(sheafs) edge[->>] node [right]  {$\tilde{I}_\sigma$} (locals);
\end{tikzpicture}
\]
where $\tilde{I}_\sigma$ and $\tilde{I}_\tau$ are epic by \cite[Theorem~2.17]{freyd1964abelian} and $\local(\tau)\inj\local(\sigma)$ is a unique monomorphism by the universal property of $\local(\tau)\inj\Fcosheaf(\tau)$.
All morphisms in an abelian category which are both epic and monic are isomorphisms (c.f. \cite[Theorem~2.12]{freyd1964abelian}), so it suffices to show that $f:\local(\tau)\surj\local(\sigma)$ is epic, where the inverse $f^{-1}=\local(\sigma\leq\tau):\local(\sigma)\to\local(\tau)$ will be the isomorphism of the desired colocal system.

Let $\alpha,\beta:\local(\tau)\to A$ be such that $\alpha\circ f=\beta\circ f$.
But composing these with $\tilde{I}_\tau$ we observe
\[
\alpha\circ\tilde{I}_\sigma\circ\Esheaf(\sigma\leq\tau)=\alpha\circ f\circ \tilde{I}_\tau=\beta\circ f\circ \tilde{I}_\tau = \beta\circ \tilde{I}_\sigma\circ\Esheaf(\sigma\leq\tau).
\]
Finally, $\Esheaf(\sigma\leq\tau)$ being epic implies $\alpha\circ \tilde{I}_\sigma=\beta\circ \tilde{I}_\sigma$ and $\tilde{I}_\sigma$ being epic implies $\alpha=\beta$.
Hence $f$ is epic.
\end{proof}

\section{Isobisheafification Diagrams} \label{sec:diagrams}

Recall that in Algorithm~\ref{alg:epi} we first project onto a collection of epimorphisms and then project onto a functor in order to ensure all maps are well-defined.
When $\sigma$ is a maximal cell, we set $\Fsheaf'_j(\sigma)=\Fsheaf_j(\sigma)$ and $\iota^\sigma_j=id$.
Otherwise, if $\tau_1,\dots,\tau_k$ are all the faces of $\sigma$, i.e. $\sigma\leq \tau_i$ for $i=1,\dots,k$, we assign $\Fsheaf'_j(\sigma)$ the pullback of all canonical image maps $\im\left(\Fsheaf_j(\sigma\leq\tau_i)\inj\Fsheaf_j(\sigma)\right)$.
Thus, we project onto an epimorphism in the sense that $\Fsheaf_j(\tau_i)\surj\im\left(\Fsheaf_j(\sigma\leq\tau_i)\right)$ and the pullback gives a natural embedding 
For the pullback of two maps, this looks like the following diagram.

\begin{center}
\begin{tikzcd}[column sep={1.5cm,between origins},row sep={1.6cm,between origins}]
	& & & & {\Fsheaf_j(\tau_1)}\\
	& & \Fsheaf'_j(\sigma) & & \im\left(\Fsheaf_j(\sigma\leq\tau_1)\right) \\
	\Fsheaf_j(\tau_2) & & \im\left(\Fsheaf_j(\sigma\leq\tau_2)\right) & & \Fsheaf_j(\sigma)
	\arrow[two heads,from=3-1, to=3-3]
	\arrow[two heads,from=1-5, to=2-5]
        \arrow[hook, from=2-3, to=2-5]
        \arrow[hook, from=2-3, to=3-3]
        \arrow[hook, from=3-3, to=3-5]
        \arrow[hook, from=2-5, to=3-5]
        \arrow[hook, "\iota^\sigma_j",  from=2-3, to=3-5]
        \arrow[bend right=20, "\Fsheaf_j(\sigma\leq\tau_2)"', from=3-1, to=3-5]
        \arrow[bend left=90, "\Fsheaf_j(\sigma\leq\tau_1)", from=1-5, to=3-5]
\end{tikzcd}
\end{center}

After projecting onto epimorphisms, however, we have no canonical maps $\Fsheaf'_j(\tau)\to\Fsheaf'_j(\sigma)$ for each face relation $\sigma\leq \tau$, so we must further pullback to create such maps (i.e. projecting onto the functor).
To do this, when $\sigma$ is now a minimal cell, we set $\Fsheaf_{j+1}(\sigma)=\Fsheaf'_j(\sigma)$.
Otherwise, if $\tau_1,\dots,\tau_k$ are all the cofaces of $\sigma$, i.e. $\tau_i< \sigma$ for $i=1,\dots,k$, we have already determined $\Fsheaf_{j+1}(\tau_i)$.
Inductively, there will be an embedding $\Fsheaf_{j+1}(\tau_i)\inj\Fsheaf'_j(\tau_i)$ and for each $i$ we will have the following pullback diagram.

\begin{center}
\begin{tikzcd}[column sep={2cm,between origins},row sep={1.4cm,between origins}]
	\mathcal{E}^{\sigma\tau_i}_j & & & {\Fsheaf_j(\sigma)}\\
	& & \mathrm{im}\left(\Fsheaf'_j(\tau_i\leq\sigma)\right) &  \\
	\Fsheaf_{j+1}(\tau_i) & \Fsheaf'_j(\tau_i) & & \Fsheaf_j(\tau_i)
    \arrow[hook, from=1-1, to=1-4]
    \arrow[from=1-1, to=3-1]
    \arrow[hook, from=3-1, to=3-2]
    \arrow[hook, from=3-2, to=3-4]
    \arrow[hook, from=3-2, to=2-3]
    \arrow[two heads, from=1-4, to=2-3]
    \arrow[hook, from=2-3, to=3-4]
    \arrow["\Fsheaf_j(\tau_i\leq\sigma)", from=1-4, to=3-4]
\end{tikzcd}
\end{center}

However, these pullbacks are only with respect to $\Fsheaf_j(\sigma)$, so finally, we must pullback with respect to $\iota^\sigma_j:\Fsheaf'_j(\sigma)\inj\Fsheaf_j(\sigma)$.
For the pullback of two maps, as well as $\iota^\sigma_j$, this looks like the following diagram.

\begin{center}
\begin{tikzcd}[column sep={1.5cm,between origins},row sep={1.4cm,between origins}]
	\Fsheaf_{j+1}(\sigma) & & \mathcal{E}^{\sigma\tau_1} \\
	& \Fsheaf'_j(\sigma) &  \\
	\mathcal{E}^{\sigma\tau_2} & & \Fsheaf_j(\sigma)
        \arrow[hook, from=1-1, to=1-3]
        \arrow[hook, from=1-1, to=3-1]
        \arrow[hook, from=1-1, to=2-2]
        \arrow[hook, from=3-1, to=3-3]
        \arrow[hook, from=1-3, to=3-3]
        \arrow[hook, "\iota^\sigma_j", from=2-2, to=3-3]
\end{tikzcd}
\end{center}

All concepts of Algorithm~\ref{alg:mono} are exactly dual to the construction of Algorithm~\ref{alg:epi}, so recall that we first project on a family of monomorphisms and then project onto a functor.
When $\sigma$ is a maximal cell, we set $\Fsheaf'_j(\sigma)=\Fsheaf_j(\sigma)$ and $\rho^\sigma_j=id$.
Otherwise, if $\tau_1,\dots,\tau_k$ are all the faces of $\sigma$, i.e. $\sigma\leq \tau_i$ for $i=1,\dots,k$, we assign $\Fcosheaf'_j(\sigma)$ the pushout of all canonical coimage maps $\Fcosheaf_j(\sigma)\surj \mathrm{coim}\left(\Fcosheaf_j(\sigma\leq\tau_i)\right)$.
Thus, we project onto a monomorphism in the sense that $\mathrm{coim}\left(\Fcosheaf_j(\sigma\leq\tau_i)\right)\inj\Fcosheaf_j(\tau_i)$ and the pushorward gives a natural projection.
For the pushout of two maps, this looks like the following diagram.

\begin{center}
\begin{tikzcd}[column sep={1.5cm,between origins},row sep={1.6cm,between origins}]
	\Fcosheaf_j(\sigma) & & \mathrm{coim}\left(\Fcosheaf_j(\sigma\leq\tau_2)\right) & & {\Fcosheaf_j(\tau_2)}\\
	\mathrm{coim}\left(\Fcosheaf_j(\sigma\leq\tau_1)\right) & & \Fcosheaf'_j(\sigma) & &  \\
	\Fcosheaf_j(\tau_1) & & & &
	\arrow[hook,from=1-3, to=1-5]
	\arrow[hook,from=2-1, to=3-1]
        \arrow[two heads, from=1-1, to=1-3]
        \arrow[two heads, from=1-1, to=2-1]
        \arrow[two heads, from=1-3, to=2-3]
        \arrow[two heads, from=2-1, to=2-3]
        \arrow[two heads, "\rho^\sigma_j",  from=1-1, to=2-3]
        \arrow[bend left=20, "\Fcosheaf_j(\sigma\leq\tau_2)", from=1-1, to=1-5]
        \arrow[bend right=90, "\Fcosheaf_j(\sigma\leq\tau_1)"', from=1-1, to=3-1]
\end{tikzcd}
\end{center}

After projecting onto monomorphisms, however, we have no canonical maps $\Fcosheaf'_j(\sigma)\to\Fcosheaf'_j(\tau)$ for each face relation $\sigma\leq \tau$, so we must further pushout to create such maps (i.e. projecting onto the functor).
To do this, when $\sigma$ is now a minimal cell, we set $\Fcosheaf_{j+1}(\sigma)=\Fcosheaf'_j(\sigma)$.
Otherwise, if $\tau_1,\dots,\tau_k$ are all the cofaces of $\sigma$, i.e. $\tau_i< \sigma$ for $i=1,\dots,k$, we have already determined $\Fcosheaf_{j+1}(\tau_i)$.
Inductively, there will be a projection $\Fcosheaf'_j(\tau_i)\surj\Fcosheaf_{j+1}(\tau_i)$ and for each $i$ we will have the following pushout diagram.

\begin{center}
\begin{tikzcd}[column sep={2cm,between origins},row sep={1.4cm,between origins}]
	 \Fcosheaf_j(\tau_i) & & \Fcosheaf'_j(\tau_i) & {\Fcosheaf_{j+1}(\tau_i)}\\
	& \mathrm{coim}\left(\Fcosheaf'_j(\tau_i\leq\sigma)\right) & &  \\
	\Fcosheaf_j(\sigma) & & & \mathcal{N}^{\sigma\tau_i}_j
    \arrow["\Fcosheaf_j(\tau_i\leq\sigma)"', from=1-1, to=3-1]
    \arrow[two heads, from=1-1, to=1-3]
    \arrow[two heads, from=1-3, to=1-4]
    \arrow[two heads, from=1-1, to=2-2]
    \arrow[two heads, from=2-2, to=1-3]
    \arrow[two heads, from=3-1, to=3-4]
    \arrow[hook', from=2-2, to=3-1]
    \arrow[from=1-4, to=3-4]
\end{tikzcd}
\end{center}

However, these pushouts are only with respect to $\Fcosheaf_j(\sigma)$, so finally, we must pushback with respect to $\rho^\sigma_j:\Fcosheaf_j(\sigma)\surj\Fcosheaf'_j(\sigma)$.
For the pushout of two maps, as well as $\rho^\sigma_j$, this looks like the following diagram.

\begin{center}
\begin{tikzcd}[column sep={1.5cm,between origins},row sep={1.4cm,between origins}]
	\Fcosheaf_j(\sigma) & & \mathcal{N}^{\sigma\tau_1} \\
	& \Fcosheaf'_j(\sigma) &  \\
	\mathcal{N}^{\sigma\tau_2} & & \Fcosheaf_{j+1}(\sigma)
        \arrow[two heads, from=1-3, to=3-3]
        \arrow[two heads, from=3-1, to=3-3]
        \arrow[two heads, from=1-1, to=1-3]
        \arrow[two heads, from=1-1, to=3-1]
        \arrow[two heads, from=2-2, to=3-3]
        \arrow[two heads, "\rho^\sigma_j", from=1-1, to=2-2]
\end{tikzcd}
\end{center}

\end{document}